\documentclass[a4paper, 11pt]{article}

\newcommand\SANSCOMMENTAIRE[1]{}
\usepackage{todonotes}

\RequirePackage[utf8]{inputenc}
\RequirePackage[T1]{fontenc}
\RequirePackage{textcomp}
\RequirePackage{graphicx}
\RequirePackage[left=3cm,right=3cm,top=3cm,bottom=3cm]{geometry}
\RequirePackage{amsmath}
\RequirePackage{amsthm}
\RequirePackage{amsfonts}
\RequirePackage{amssymb}
\RequirePackage{mathtools}
\RequirePackage{stmaryrd}
\RequirePackage{csquotes}
\RequirePackage{mathrsfs} 
\RequirePackage{cancel} 
\RequirePackage{enumitem}
\RequirePackage{xcolor}
\RequirePackage[section]{placeins}
\RequirePackage{caption}
\RequirePackage[final]{pdfpages}
\RequirePackage{tikz}
\usepackage{etoolbox}
\usepackage[linesnumbered]{algorithm2e}
\usepackage{footmisc}
\usepackage{hyperref}

\usepackage{authblk}

\usetikzlibrary{chains, decorations.pathmorphing, calc}

\theoremstyle{plain}
\newtheorem{theorem}{Theorem}[section]
\newtheorem{remark}[theorem]{Remark}
\newtheorem{prop}[theorem]{Proposition}
\newtheorem{cor}[theorem]{Corollary}

\newtheorem{claim}{Claim}[theorem]

\theoremstyle{definition}
\newtheorem{example}[theorem]{Example}
\newtheorem{definition}[theorem]{Definition}

\newtoggle{paper}
\toggletrue{paper}

\newcommand{\say}[1]{``#1''}

\newcommand{\om}{\omega}

\newcommand{\omckx}[1]{\omega_1^{\mathrm{\scriptscriptstyle{CK}}, #1}}
\newcommand{\mo}{On}

\newcommand\Alpha{\mathrm{A}}

\newcommand{\Tau}{\mathrm{T}}

\DeclareMathOperator{\rk}{rk}

\DeclareMathOperator{\Lim}{Lim}

\newcommand{\SG}{\Sigma_\Gamma}
\newcommand{\zG}{\zeta_\Gamma}
\newcommand{\TG}{\Tau_\Gamma}
\newcommand{\eG}{\eta_\Gamma}
\newcommand{\lG}{\lambda_\Gamma}
\newcommand{\gG}{\gamma_\Gamma}

\newcommand{\SGsup}{\Sigma_{\Gamma_{\sup}}}
\newcommand{\SGsupn}{\Sigma_{\Gamma_{\sup}^n}}
\newcommand{\TGsup}{\Tau_{\Gamma_{\sup}}}
\newcommand{\TGsupn}{\Tau_{\Gamma_{\sup}^n}}

\newcommand{\UG}{\mathcal{U}_\Gamma}

\newcommand{\wt}[1]{\widetilde{#1}}

\newcommand{\wh}[1]{\widehat{#1}}
\newcommand{\set}[1]{\wh{#1}}

\makeatletter
\newcommand*\rel@kern[1]{\kern#1\dimexpr\macc@kerna}
\newcommand*\wb[1]{%
  \begingroup
  \def\mathaccent##1##2{%
    \rel@kern{0.8}%
    \overline{\rel@kern{-0.8}\macc@nucleus\rel@kern{0.2}}%
    \rel@kern{-0.2}%
  }%
  \macc@depth\@ne
  \let\math@bgroup\@empty \let\math@egroup\macc@set@skewchar
  \mathsurround\z@ \frozen@everymath{\mathgroup\macc@group\relax}%
  \macc@set@skewchar\relax
  \let\mathaccentV\macc@nested@a
  \macc@nested@a\relax111{#1}%
  \endgroup
}

\newcommand{\eqnum}{\refstepcounter{equation}\textup{\tagform@{\theequation}}}

\makeatother
\newcommand{\nott}[1]{\wb{#1}}

\DefineFNsymbols{mySymbols}{{\ensuremath\dagger}{\ensuremath\ddagger}\S\P
   *{**}{\ensuremath{\dagger\dagger}}{\ensuremath{\ddagger\ddagger}}}
\setfnsymbol{mySymbols}

\begin{document}

\title{\huge Toward higher-order infinite time Turing machines: simulational $\Gamma$-machines \vspace{1em}}

\author[1]{Olivier Bournez}
\author[2]{Olivier Finkel}
\author[1,2,3]{Johan Girardot\thanks{Corresponding author: \texttt{johan.girardot@ens-paris-saclay.fr}}}

\affil[1]{Institut Polytechnique de Paris, École Polytechnique, LIX Laboratory,  
1 Rue Honoré d'Estienne d'Orves, 91128 Palaiseau, France}

\affil[2]{CNRS, Université Paris Cité and Sorbonne Université, IMJ-PRG,  
8 Place Aurélie Nemours, 75013 Paris, France}

\affil[3]{Université Sorbonne Paris Nord, LIPN, Institut Galilée,  
99 Avenue Jean-Baptiste Clément, 93430 Villetaneuse, France}

\maketitle 

\begin{abstract}
Infinite time Turing machines (ITTMs) have been introduced by Hamkins and Lewis in their seminal article \cite{hamkins_ittm}. The strength of the model comes from a limit rule which allows the ITTM to compute through ordinal stages. This rule is simple to describe and another rule would lead to a different model of ordinal computation. The aim of this article is to define a collection of limit rules for which the models of infinite Turing machine they induce have nice properties, akin to that of the ITTM.
Through the analysis of the universal ITTM and of its preponderant role in the study of the ITTM, we devise a set of four constraints.
A limit rule satisfying those constraints yields a model of infinite machine for which we can define a universal machine.
Adding a fifth constraint, we show that any limit rule definable in set theory which meets those constraints produces a model of infinite machine with the desired properties. Among those, the fact that the supremum of the writable ordinal matches that of the clockable ordinal. That is, with the usual notations, the equality $\lambda = \gamma$ holds for any of those limit rules. Eventually, we provide a counter-example to show that the four first constraints alone are not sufficient.

\end{abstract}

This article is extracted from the fifth chapter of the thesis \cite{these} of Johan Girardot. It constitutes the first main contribution of this thesis and it aims to lay the groundwork for a generalized study of infinite Turing machines.

\section{Introduction}

Infinite time Turing machines (ITTMs) have been first introduced by Hamkins and Lewis in \cite{hamkins_ittm}. Since then it has been thoroughly studied. Recent surveys can be found, such as \cite{merlin, welch_summary}, while new studies like \cite{carl_space_time, durand_gaps_admissibles} continue to appear. Generalizations can also be found in \cite{merlin, hypermachines, koepke_otm}. 
We reintroduce briefly the model of the ITTM to explain the nature of those different generalizations. 

The ITTM is a model of computation that computes through the hierarchy of the ordinals. It is built on the same structure as the classical finite Turing machine. Imagine a finite Turing machine $m$ that somehow managed to compute up to some limit ordinal stage $\alpha$. At this stage, it is in some state $q$, has some real number $x $ (an element of $\left\{0, 1 \right\}^{\mathbb{N}}$, which we more conveniently write $^\om 2$) written on its tape and its head is on some cell $i \in \mathbb{N}$. Following \cite{hamkins_ittm}, we call the tuple of those data the \emph{snapshot} of  the machine at this computation stage. Given this snapshot, the computation step from stage $\alpha$ to $\alpha+1$, that is to say the snaphot at stage $\alpha+1$, is simply determined by the program (or the code) of $m$ using the classical rules of finite Turing machines. And so is the computation up to stage $\alpha+k$ for any finite $k$.

We see that the question left out in this description is: how does the machine ``somehow computes'' up to a limit stage? And what is interesting is that it is enough to answer this in order to have an effective infinite model of computation. As described, once a limit stage is reached by the computation, it carries on through the next successor stages as a classical finite Turing machine, until the next limit stage.
The next definition is Hamkins and Lewis' proposition to answer this question. This defines the model of the ITTM.
\begin{definition}[\cite{hamkins_ittm}]
An \emph{infinite time Turing machine (ITTM)} $M$ is a machine which has the same structure as a three-tape finite Turing machine, those three tapes being called input, output and a working tape and such that:
\begin{itemize}
\item At any stage $\alpha +1$, the snapshot of the machine $M$ is determined by the snapshot of $M$ at stage $\alpha$ and by the program of the machine $m$ as done with classical finite Turing machines.
\item At any limit stage $\alpha$, the head of $M$ is on the first cell, its state is some distinguished state $q_{limit}$ and for all cell $i$, writing $C_i$ for the (class) function that maps an ordinal stage to the value of the cell $i$ at this stage:
\begin{align}\label{eq:limit_rule_2}
C_i(\alpha) = 0 \longleftrightarrow \exists \beta < \alpha \, \forall \delta \in [\beta, \alpha[ \ C_i(\delta) = 0
\end{align}
That is the value of the cell $i$ at limit stage $\alpha$ is set to $0$ if and only if there is some stage $\beta < \alpha$ at which the value of cell $i$ stabilized on $0$ up to stage $\alpha$.
\end{itemize}
\end{definition}

The important part of this definition is the \emph{limit rule} (\ref{eq:limit_rule_2}) which uses the history of the machine to set the value of all cells at a limit stage. This limit rule is usually called the $\limsup$ rule: given some limit stage $\alpha$, the value of $C_i(\alpha)$ is simply defined as the $\limsup$ of the $C_i(\delta)$ for $\delta < \alpha$. Again, it is important to observe that, up to this point, this rule could be any kind of rule that unambiguously sets the value of any cell and at any limit stage as a function of this stage history. Changing this rule would simply yield a different model of ordinal computation. 

\subsection{Generalizations, state-of-the-art and the place of this work among those}

From there, we can already distinguish two kind of generalizations: generalizations can be either \emph{rule-wise} if they try to produce new limit rules, either \emph{machine-wise} if they try to enhance the structure of the machine. 

\vspace{.6em}

For some machine-wise generalizations, the reader can refer to \cite{koepke_otm} in which Koepke investigates on enhancements of the basic machine structure: what can be said if the tape is not $\om$ long but $\alpha$ long for some limit ordinal $\alpha$? This yields $\alpha-$ITTMs and very interestingly provides the missing \say{mechanical} equivalent of $\alpha$-recursion. Further, it is possible to enhance the size of the tapes to $\mo$ itself (the class of all ordinals) as done in \cite{koepke_102_alpha_ittm} to define \say{ordinal Turing machines} (OTM). In this case, this yields a model of computations whose behavior and power are closely related to $L$ the Gödel's constructible universe. For example, Koepke shows how in $L$ every set of ordinals is computable with an OTM from a finite set of ordinal parameters. 

The structure of the machine may also be more deeply modified to define infinite time register machine (ITRM). This is done in \cite{koepke_112_enhanced_theory}: an ITRM has the structure of a register machine and at a limit stage its current line and its registers values are set to the $\liminf$ of their respective previous values. The register may be of ordinal size or again of size $\mo$ as described in \cite{koepke_121_ORM}. The monograph \cite{merlin} widely discusses those generalizations.

Even further, the work of Fischbach and Seyfferth in \cite{seyfferth_ordinal_lambda} could be seen as a radical machine-wise generalization. In this article, they develop an ordinal extension of lambda calculus in which a term can be the $\alpha$-fold composition of another term for $\alpha \in \mo$. Then, as in previous models of computation, the normal form procedure is an extension of the usual procedure for finite terms and again it relies, at limit stages, on an inferior limit.

\vspace{.6em}

As for \say{rule-wise} generalization, to this date and to the best of our knowledge, the only work published is that of Friedman and Welch in \cite{hypermachines}. In this article, the authors devise a new limit rule, based on a $\liminf$ working with a dynamically restricted history. More precisely, there is a fourth tape called the \say{rule tape} and at any limit stage, the value of a cell is set as the inferior limit of a subset of its previous values, which subset is determined by the history of the rule tape. The authors show how this yields a model of computation strictly more powerful than ITTMs but which is equally \say{well-behaved} and which also exhibits strong links with the constructible hierarchy.
We discuss this generalization further in the light of our study of the universal machine in Remark \ref{rmk:hypermachines}. 

\vspace{.6em}

So, machine-wise generalizations provide new and interesting models of computation. They do so, for example, by filling the gaps left in the Church-Turing equivalence generalized to ordinal computation, as done by Koepke's OTM or by Fischbach and Seyfferth's ordinal lambda calculus. On the other hand, the concept of rule-wise generalization has it own interesting features and for the moment has been way less studied.
One of those features is the scope and the easiness with which such generalizations can be devised: as seen, any limit rule plugged in the structure of the infinite time Turing machine yield a new model of ordinal computation. And as a limit rule is ultimately a formula in the language of set theory, this yields a rich parallel with logic and set theory.
Understanding which formulas induce well-behaved and interesting models of infinite machines may gives us new pieces of information regarding those.

It is clear that not any formula which can be seen as limit rule will give rise to as a well-behaved and interesting model as that of ITTMs. One way to measure the interest of such models is to compare them to the classical ITTM. In the case of ITTM, its power can easily be equated with the largeness of its writable and clockable ordinals and one very good representative of its theorical interest is the equality $\lambda = \gamma$. This equality, whose truth was left open in \cite{hamkins_ittm} and established in \cite{welch_main}, links two constants of very different nature. Here $\lambda$ is the supremum of the writable ordinals, that is ordinals appearing on the output of some halting ITTM, 
while $\gamma$ is the supremum of the clockable ordinals, that is ordinals corresponding to the halting time of an ITTM. This equality is fruitful as it provides a link between ITTMs and the constructible hierarchy. For example, with it, we know that all writable sets are written before stage $\gamma = \lambda$ and so are in $L_\lambda$ while the fact that all ordinals below $\lambda$ are writable implies that any set in $L_\lambda$ is writable. 

This equality was proven using the universal ITTM and we will see how, more generally, the successes of the model of the ITTM are themselves rooted in the existence of a universal machine. So naturally, in this article, we will be interested in limit rules which yield a model of infinite machines for which there exists a universal machine. In particular, the main result of this paper is Theorem \ref{thm:sigma_tau_looping_condition}. It proves a structural property which generalizes the equality $\lambda = \gamma$ for such models of infinite machines, which also are definable in set theory while satisfying a safeguard condition.
This safeguard condition is moreover shown to be required in Theorem \ref{thm:contre_exemple_machine}.

Hence, in this paper, working our way through the possibility of devising interesting and well-behaved rule-wise generalizations, our first leading question will be: how and why does the universal ITTM work?

\subsection{Results and organization of the paper}

To begin answering this question, \say{how and why does the universal ITTM work?}, we start in Section \ref{sec:introducing_the_universal_machine} with a brief overview of the ITTM and of common state-of-the-art proof techniques. This will put into light the importante role of simulations and, in particular, of the universal ITTM and provide the main intuitions behind this work. More precisely, 
we begin this section by proving a slightly new structural equality (which generalizes the equality $\lambda = \gamma$ proved in \cite{welch_main}) for the usual ITTMs. This proof shows how the universal machine is canonically used in this kind of advanced proofs developed by Welch in \cite{welch_main}. We then try to understand what is implicitly needed in this proof regarding the universal machine and more generally regarding the possibility of simulating other machines. At the end of this discussion, we exhibit five critical characteristics satisfied by the usual $\limsup$ limit rule. The first four of those form a sufficient condition for a given limit rule to yield a model of ITTM in which a universal machine exists. Again, this section serves as an extended introduction presenting the motivations and intuitions for this article. 

In Section \ref{sec:general_definitions} starts, strictly speaking, the formal work. We provide general definitions for the article as well as a general definition of what we call a limit rule. It can indeed be seen as an operator that maps ordinal words, representing the history of the computation, to collection of symbols in the alphabet of the machine, representing the tapes. With this formal definition we then formalize the five conditions found in the previous section.

In Section \ref{sec:higher_order_ittms}, we start by showing in Theorem \ref{thm:2_sym_machines} that the $\limsup$ and the $\liminf$ operators are, surprisingly, the only two-symbol limit operators satisfying this set of five conditions. This steers the search for higher order machines toward a yet unexplored direction: $n$-symbol infinite Turing machine. We then show that for any $n$-symbol Turing machine that satisfies those five conditions and definable by a set-theoretic formula, we can establish many of the results established for the classical ITTMs. The most important one, Theorem \ref{thm:sigma_tau_looping_condition}, whose proof takes the larger part of the section, is a strong constant equality relating the clockable ordinals and the writable ordinals (For the general case, those constants are introduced in Section \ref{sec:general_definitions}). A corollary of this equality is that for those models, the generalization of the equality $\lambda = \gamma$ also holds.

In Section \ref{sec:counter-example} we eventually show in Theorem \ref{thm:contre_exemple_machine} that the fifth condition, acting as a safeguard and not needed for an operator to yield a model of machine with a universal machine, is however needed for the main theorem, i.e. Theorem \ref{thm:sigma_tau_looping_condition}. It is needed in the sense that, without it, we can devise a limit rule for which the equality of the main theorem does not hold.

\section{Eventually clocking and accidentally clocking}\label{sec:introducing_the_universal_machine}

As mentioned, this section is meant to contextualize this work and its seminal observations within the study of the ITTMs and the techniques commonly used.
To this extent we briefly reintroduce some key concepts of Hamkins ITTMs as well as the proofs of some main results.
This will also highlight the importance of some of the implicit properties of the $\limsup$ rules. Everything we define or redefine, in a more or less formal way, will be formally defined in Section \ref{sec:general_definitions} in a more general setting. 

In the model of the ITTM, as in the finite model of Turing machines, the machine can produce an output: when a machine halts, whatever is written on its output tape is understood as its output. While in the finite Turing machines case the output can be seen as an integer, in the infinite case, it will be seen as a real number.  Hence a real number can be written as output and we say that the real number was \emph{written} and so that it is \emph{writable}.

But more interestingly, as done in \cite{hamkins_ittm}, in the infinite setting of ITTMs we can define different degrees of writability for the real numbers.  As said, a real number $x$ can be \emph{writable} if there is a machine that halts (by reaching its distinguished state $q_{halt}$) and, upon halting, which has $x$ written on its output tape. But it can also be \emph{eventually writable} (sometimes abbreviated e.w.) if there is a machine that never stops and that, at some point, has $x$ written on its output tape and never modifies it afterward. It can be \emph{accidentally writable} (sometimes abbreviated a.w.) if there is a machine in which $x$ is written on any of its tapes at any stage of its computation, even if it is modified or erased just after. It is easy to see that being writable implies being eventually writable implies being accidentally writable.

Then, an ITTM can also ``write'' ordinals. It writes an ordinal when it writes a real number that encodes this ordinal: we say that $x$ encodes $\alpha$ when $x$ describes some well-order $\prec_\alpha$ on $\om$ in the following way:
\begin{align*}
n \prec_\alpha m \longleftrightarrow x[\langle n, m \rangle] = 1
\end{align*}
and such that the order-type of $\prec_\alpha$ is $\alpha$. What motivates this definition is a fundamental result of Hamkins and Lewis, Theorem 2.2 of \cite{hamkins_ittm}, which states that, with this encoding, an ITTM can decide whether a real number encodes a well-order.

So, as ordinals can be written, we may wonder what kind of ordinals can be written, or at least accidentally written. First, given the simple fact that ITTMs have $\om$ cells, they can only write countable ordinals. Moreover, as there are countably many ITTMs, the writable and the eventually writable ordinals are bounded in $\om_1$; that is their union is countable as well. However, it may not seem so clear whether this fact holds for the accidentally writable ordinals, since a single machine will accidentally write many different ordinals through its computation. Still it happens that those are bounded as well. Hence, as introduced in \cite{hamkins_ittm} there exists three supremum ordinals, \emph{$\Sigma$, the supremum of the accidentally writable ordinals}, \emph{$\zeta$, the supremum of the eventually writable ordinals} and \emph{$\lambda$, the supremum of the writable ordinals}.

Additionnaly Hamkins and Lewis define the \emph{clockable} ordinals. An ordinal $\alpha$ is clockable when there is some machine which halts just after carrying out its $\alpha^{th}$ step of computation. That is, through this $\alpha^{th}$ step, it transitions from stage $\alpha$ to stage $\alpha+1$ and, in this latter stage, the machine reached the state $q_{halt}$. Then as previously, there is \emph{$\gamma$ the supremum of the clockable ordinals}. One important question left open in \cite{hamkins_ittm} was: does $\lambda = \gamma$? This was answered positively in Theorem~1.1 of \cite{welch_main} using a universal machine; which now brings us to this machine and its importance.

A universal machine $\mathcal{U}$ is a machine that simulates simultaneously all ITTMs. It was already used by Hamkins and Lewis in Theorem~3.4 of \cite{hamkins_ittm} and its importance appears even more clearly in \cite{welch_main}. 
To illustrate this, we give a proof of a slightly more general equality than that established by Welch, but which relies on the same ideas. As there are three different ways for an ordinal to be written (writable, eventually writable and accidentally writable), we would also can devise three ways for an ordinal to be clockable: and this can be done by saying that: $\alpha$ is \emph{eventually clockable} (e.c.) when the output of some converging machine $m$ stabilizes at stage $\alpha$, and $\alpha$ is \emph{accidentally clockable} (a.c.) when there is a computation in which at stage $\alpha$ some real number $x$ appears for the first time in the computation. Then we can write \emph{$\eta$ the supremum of the eventually clockable ordinals} and \emph{$\Tau$ (capital $\tau$) the supremum of the accidentally clockable ordinals}. And with the latter, we can show the following result, with $\Sigma$ being the supremum of the accidentally writable ordinals.

\begin{prop}[\cite{welch_main}]\label{prop:sigma_2_tau_2}
$\Sigma = \Tau$. That is the supremum of the accidentally writable ordinals is equal to that of the accidentally clockable ordinals.
\end{prop}
This proposition, while stated here in a new way, is very close to Corollary 3.5 of \cite{welch_main} which states that the collection of sets encoded by accidentally writable real numbers is actually $L_\Sigma$. 
In this article from Welch, it was a corollary of the fact that $\lambda = \gamma$ and this fact is itself a corollary of Theorem \ref{prop:sigma_2_tau_2} here, as will be proved in a more general setting in Corollary \ref{cor:lambda_gamma}. Another corollary is the fact that, with previous notations, $\zeta = \eta$. So this downward chain of consequences likely justify this new phrasing. Moreover, as said, the proof of this slightly more general result uses the same technique as that of the main proposition of \cite{welch_main}. Finally, we take for granted in this proof that there exists a universal ITTM $\mathcal{U}$ and that we can design, in a natural way, machines which simulate other machines. Those fact are actually true and we will come back to those.

\begin{proof}
First, suppose that $\Sigma > \Tau$. This implies that $\Tau$ is accidentally writable by some machine $m_\Tau$. So we consider the following machine $\mathcal{M}$. It will use the code of $\Tau$ that appears at some point in $m_\Tau$ to write something for the first time after stage $\Tau$, hence reaching a contradiction.
To do this, first, it keeps the first cell of its working tape always set to $0$ and the first cell of its output tape always set to $1$.
So when using its working tape, it uses cells $i > 0$. Then, it simulates $m_\Tau$ on some part of its working tape (again, without using cell $i=0$) and each time an ordinal $\alpha$ is written in the simulation of $m_\Tau$, it counts for at least $\alpha$ steps. Remember that by \cite[Theorem 2.2]{hamkins_ittm} it can recognize real numbers describing well-orders. This is done by trying to count through it, so it takes at least $\alpha$ steps to recognize that a real number encodes a well-order of order type $\alpha$.
Then, it copies $\alpha$ on its output tape, without using the first cell that it keeps to $1$. 
Now, when $\Tau$ appears for the first time at stage $\alpha_\Tau$: as $\Tau$ appears for the first time in $m_\Tau$ and $\Tau$ is the supremum of the a.c., that is of the stages of first appearance in a given machine, we have $\alpha_\Tau < \Tau$. Then, $\mathcal{M}$ counts for at least $\Tau$ steps and after those writes $\Tau$ on its output tape. This is the first time that $\Tau$ appears on this tape. Also, the whole content of this tape, with the first cell set to $1$ cannot have appeared earlier in the working tape whose first cell is left on $0$.
Hence the real number it writes was not written before and it is written for the first time at least at stage $\alpha_\Tau + \Tau \geqslant \Tau$. And $\alpha_\Tau + \Tau$ is consequently a.c.\ which is a contradiction. So $\Sigma \leqslant \Tau$.

For the other direction, to show that $\Sigma \geqslant \Tau$, we will show that in any computation that does not halt, for any cell on any tape, the values of this cell at stages $\zeta$ and $\Sigma$ match and moreover that a cell set to $0$ at stage $\zeta$ stays forever so. This will prove that nothing new is ever written after stage $\Sigma$, as the machine will be repeating indefinitely the segment of computation $\left[ \zeta, \Sigma \right[$ and so that $\Sigma \geqslant \Tau$. Let $m$ be a machine and $i \in \om$ be some cell on one of its tape.
\begin{itemize}
\item Suppose $C_i(\Sigma) = 0$. Then, by the limsup rule, the cell must have converged to $0$ at some least ordinal stage $\alpha < \Sigma$. That is $C_i(\alpha) = 0$ and stays so up to $\Sigma$. We show that $\alpha < \zeta$ (this is actually Main Proposition of \cite{welch_main}). We design a machine $\mathcal{M}$ that does the following: it launches a copy $\mathcal{U}_1$ of the universal machine and for each ordinal $\beta$ appearing in the computation of $\mathcal{U}_1$, that is accidentally written in a machine simulated by $\mathcal{U}_1$, $\mathcal{M}$ writes this ordinal on its output tape and simulates $m$ up to this ordinal $\beta$. If $C_i^m(\beta) = 0$, it launches a new copy $\mathcal{U}_2$ of the universal machine and each time some ordinal $\beta' > \beta$ is appears in $\mathcal{U}_2$, it simulates a fresh copy of $m$ up to stage $\beta'$ and looks whether the cell $i$ was set back to $1$ between $\beta$ and $\beta'$. If it was, $\mathcal{M}$ goes back to its simulation of the universal machine $\mathcal{U}_1$ and looks for the next such ordinal $\beta$, that is such that $C_i^m(\beta) = 0$. When $\beta > \alpha$, the machine never find $1$'s anymore in the history of cell $i$ after this $\beta$, as all $\beta' > \beta$ generated are strictly between $\alpha$ and $\Sigma$. Hence, $\beta$ is written on the output and $\mathcal{M}$ looks indefinitely for $\beta' > \beta$ at which a $1$ appear on cell $i$. And so $\mathcal{M}$ actually eventually writes some $\beta > \alpha$. As such, $\beta$ is eventually writable and we have $\alpha < \beta < \zeta < \Sigma$ and $C_i(\zeta) = 0$ as well.
\item Suppose now that $C_i(\zeta) = 0$. Then, the value of the cell converged at some $\alpha < \zeta$ and there exists a machine $m_\alpha$ that eventually writes $\alpha$.
We consider the following computation from a machine that we call $\mathcal{W}$: It simulates $m_\alpha$ on some part of its working tape, which will eventually write $\alpha$. Each time an ordinal $\beta$ is written on the simulated output tape of $m_\alpha$, it does the following: while $\beta$ does not change (that is the simulation of $m_\alpha$ is intertwined by dovetailing with the rest of the computation) it simulates $\mathcal{U}$ and, each time $\mathcal{U}$ (accidentally) writes an ordinal $\beta' > \beta$, $\mathcal{W}$ starts by simulating $m$ up to $\beta$ and then up to $\beta'$. If it finds a $1$ in the cell $i$ between stages $\beta$ and $\beta'$, it writes $\beta'$ on the output tape and stops this part of the computation. Note that as the simulation of $m_\alpha$ is still going on, it might start again from the beginning later, if its output value changes. However, at some point $\alpha$ is written in $m_\alpha$ and does not change anymore. At this point, if $\mathcal{W}$ finds a $1$ between stages $\beta = \alpha$ and some $\beta' > \alpha$, this $1$ must appear after $\zeta$ (as the value of the cell converges to $0$ at stage $\alpha$ and up to stage $\zeta$) and then $\beta'$, which it has consequently eventually written is greater that $\zeta$, which is a contradiction. Hence, the value of $C_i$ never changes after stage $\zeta$ and $C_i(\Sigma) = 0$ and this concludes the proof.
\end{itemize}
\end{proof}

The proof of Proposition \ref{prop:sigma_2_tau_2} is somewhat of a classical 
proof in the theory of ITTMs, at least in its usage of simulations and in particular in that of the simulation of the universal machine $\mathcal{U}$. Hence, the question that we postponed: how does the universal ITTM work? Does it actually exist and how can we describe it? 

To answer this, observe that $u$, the classical three-tape finite universal machine, gives almost immediately rise to an infinite time universal machine that simulates a single other ITTM. Indeed, in the finite setting, $u$ is such that given the code of some finite machine $m$, it simulates $n$ steps of $m$ in less than $C n^2$ steps for some constant $C$, supposing that $u$ works in a straightforward fashion. How does it work? Without going into the finicky details, each cell $i$ of $m$ is simulated by some cell $I$ of $u$. That is, for any finite stage  $k$, the value of $I$ in $u$ at stage $C k^2$ is the value of $i$ in $m$ at stage $k$, and the next value to be written in $I$, after some steps, is that of the cell $i$ at stage $k+1$. Outside of those cells $I$, the other cells of $u$ are used to keep track of the simulation, in particular of the position of  the head and of the state of the simulated machine $m$.

Now consider the ITTM $\mathcal{U}_0$ whose code is the same as $u$. Remember indeed that the structure of an ITTM is exactly that of a three-tape classical Turing machine. Similarly, we can see $m$ as an ITTM, which we denote by $M$. What happens when $\mathcal{U}_0$ computes with the code of $M$ as input (which we write $\langle M \rangle$)? By construction, for any finite stage $k$, the snapshot of $\mathcal{U}_0[\langle M \rangle]$ at stage $k$ will be the same as that of $u[\langle M \rangle]$. But $M$ and $m$ share the same code, so this is the same thing as $u[\langle m\rangle]$. Hence, for the first $\om$ steps, the computation of $\mathcal{U}_0[\langle M \rangle]$ is the same as that of $u[\langle m \rangle]$. In other words, the ITTM $\mathcal{U}_0$ with the code of $M$ as input simulates the ITTM $M$ through all the finite stages. 

But what happens a stage $\om$? That is, what is the snapshot of $\mathcal{U}_0[\langle M \rangle]$ at stage $\om$? Take some cell $i$ of $M$. It is simulated by some cell $I$ of $\mathcal{U}_0$. Hence, suppose $i$ in $M$ stabilizes on $0$ before stage $\om$. That is, with the $\limsup$ rule, $C^M_i(\om) = 0$. Then, so does $I$ in $\mathcal{U}_0[\langle M \rangle]$ and $C^{\mathcal{U}_0}_I(\om) = 0$ and same goes for the cells $i'$ such that  $C^M_{i'}(\om) = 1$. Hence, the simulation of the limit rule for the simulated cell $i$ in $M$ comes directly from the limit rule itself applied to the simulating cell $I$ in $\mathcal{U}_0$. This means that at stage $\om$ in the computation $\mathcal{U}_0[\langle M \rangle]$, the content of the different cells $I$ describes exactly that of the cells $i$ in $M$ at stage $\om$. From there, to simulate the ITTM $M$ further through the ordinals, only a few things are missing at limit stages: $(1)$ the auxiliary cells used for the simulation will likely all be set to $1$ and those need to be tidied up, $(2)$ the simulated head of $M$ should be back on the first simulated cell and $(3)$ the simulated state of $M$ should be set to $q_{limit}$. But since at any limit stage $\mathcal{U}_0$ reaches its own distinguished state $q_{limit}$, it is easy from there to modify the machine $\mathcal{U}_0$ in order to have it tidy up its auxiliary cells and to set the correct head position and state for the simulated machine. This takes a finite number of steps and after those the simulation goes on, in time for the next limit ordinal stage.

This previous description yields a universal infinite time machine $\mathcal{U}'_0$, slightly adapted from $\mathcal{U}'_0$, which simulates a single ITTM. From it we can simply enough describe the universal machine $\mathcal{U}$ that simulates all ITTMs at once. $\mathcal{U}$ works with its working tape as if it were $\om$ different tapes in which it simulates in parallel all computations $\mathcal{U}'_0[m]$ for all machines $m \in \om$. For $\mathcal{U}$ to split its working tape, this can be done as follows: first, all even cells (that is cells on the workings tape whose index $i$ is even) are kept aside (metaphorically) and they will constitute a virtual working tape of its own that the machine uses to keep track of its simulations of all the $\mathcal{U}_0[m]$'s. The idea being that its easy for it to ``stay'' on this virtual tape while it reads it as it just needs to shift its head twice to the left or to the right of the real tape to move on this virtual tape. And using this first virtual tape it can split the odd cells into $\om$ tapes, e.g.\ the virtual tapes of some machine $m$ will be constituted of the cells $(f(m, i))_i$ for any computable function (in the finite sense) $f$ that partitions the odd integers into $\om$ unbounded sets. Then, when $\mathcal{U}$ begins, it initializes all its virtual tapes $m$ as would $\mathcal{U}'_0[m]$ initializes its working tape, which takes $\om$ steps. Then begins the simulation itself. By dovetailing, that is simulating the machines $m$ in cascade, $k$ steps of the first $k$ machines, $k+1$ steps of the first $k+1$ machines, etc. it simulates $\om$ steps of all of those machines in $\om$ steps of itself; which should be enough for the description of $\mathcal{U}$.
 
Now, why does the universal machine works? That is, what is implicitly used in this explanation? First, as a cell $i$ of some machine $m$ (we now drop the capital $m$ and only consider ITTMs) will likely be simulated by a cell $I \neq i$ in $\mathcal{U}$, it uses the fact that all cells are governed by the same limit rule. Also, it uses the fact that the limit value of the cell $I$ only depends of the history of this very cell. The context of those cells (e.g. the neighboring cells) may greatly vary. So we rely on the fact that the limit rule only looks at the histories $h_i$ and $h_I$, the histories of respectively the cell $i$ in $m$ and the cell $I$ in $\mathcal{U}$. When a rule only looks at the history of a given cell to define the limit value of this cell, we will say that the rule is $\emph{cell-by-cell}$. 
Further, for a limit ordinal $\alpha$, in order to say that $C_i^m(\alpha) = C_I^{\mathcal{U}}(\alpha)$, where $C_i^m(\alpha)$ denotes the value of cell $i$ in $m$ at stage $\alpha$, we use the fact that $h_i$ and $h_I$ are \emph{somewhat} the same. 

But what is meant by this ? Are the histories $h_i$ and $h_I$ not actually equal? Suppose that $m$ is a very simple machine that continuously blinks its cell $i$. To do so, first, the head of $m$ needs $i$ steps to reach the cell $i$. As the cell $i$ is initially set to $0$, its history begins with $0^i$, that is the word made from the letter $0$ repeated $i$ times. At stage $i+1$, the cell $i$ in $m$ is set to $1$. Its history is now $0^i 1$. And then back to $0$, then again to $1$ and so on. Hence, at some limit stage $\alpha$, its history reads $0^i(10)^\alpha$. 
Now what does the history of cell $I$ in $\mathcal{U}$ read? We can merge in $\mathcal{U}$ the $\om$ steps of initialization with the actual simulation, but in any case, the machine will begin the actual simulation of $m$ after some $k$ strictly greater than $i$. Hence, $h_I$ begins with $0^k$. At stage $k+1$, $\mathcal{U}$ writes a $1$ in cell $I$ and the history of $H_i$ reads $0^k 1$. But, and this is one of important point, after having simulated this step from $k$ to $k+1$ in $m$, $\mathcal{U}$ will then be simulating at least the machine $m+1$ for a certain number of steps and. And it won't be coming back to $m$ before some finite number of steps $k_2$. In the meantime, during those $k_2$ steps, the cell $I$ stays set to $1$. So, after those $k_2$ steps, $H_i$ reads $0^k 1^{k_2}$. And so, as the computation goes on, $H_i$ will look like: $0^k 1^{k_2} 0^{k_3} \ldots 1^{k_\alpha} 0^{k_{\alpha+1}} \ldots$ where each $k_\nu$ are finite and strictly greater than $1$. So, to some extent, $h_i$ and $H_i$ are not quite the same. Still we could say that $h_i$ is the \emph{contraction} of $H_i$, that is the word obtained by contracting any $1^k$ or $0^{k'}$ respectively to a $1$ or a $0$. And $\mathcal{U}$ operates as it should because, as we will say, the $\limsup$ rule is \emph{contraction-proof}. The limit value of a cell does not change if we contract (or conversely dilate) its history. And here, we could say that the dilatation from $h_i$ to $H_i$ is only finite as all $k_\nu$ are finite, but in practice, when we use $\mathcal{U}$ or when we simulate any other machine, it is very common to keep it ``on hold'' for extensive periods of computation, as in the first part of the proof of Proposition~\ref{prop:sigma_2_tau_2}. In such a case the $k_\nu$'s may be infinite ordinals and \emph{a priori} unboundedly big. 

Other than that, keeping a simulated machine ``on hold'' for some infinite ordinal amount of steps $\alpha$, implicitly makes use of the fact that the cells that aren't written on (i.e. whose content is not modified) won't change their content at limit stages. We can say that the $\limsup$ rule is \emph{stable} as the content of a cells which stabilized up to a limit stage does not change at said limit stage.

Eventually, there is a last feature of the Hamkins and Lewis ITTMs that is used implicitly. Take the computation of some machine $m$ up to some limit stage $\alpha$. Then consider a fresh computation of $m$ in which the initial snapshot (that is the content of the tapes, the state and the head position) is set to match that of $m$ after those $\alpha$ steps. In this second computation, stage $0$ corresponds to stage $\alpha$ in the first computation. Still in both cases, observe that it leads to the same subsequent computation. Hence, for any stage $\beta$ we would like, for all cell $i$, writing $s_\alpha$ for the snapshot of $m$ at stage $\alpha$ and $C^{m, s}_i$ for the function that maps an ordinal to the content of the cell $i$ at this ordinal stage in the computation of $m$ that starts with $s$ as its snapshot, that
\begin{align*}
C_i^m(\alpha + \beta) = C_i^{m, s_\alpha}(\beta)
\end{align*}
This clearly holds in the classical model of ITTM for finite $\alpha$ and $\beta$. For $\alpha$ and $\beta$ limits, it means that even if the history $h_i$ of length $\alpha+\beta$ is truncated to the final segment $h'_i$ of length $\beta$, it still yields the same limit value with respect to the $\limsup$ rule. This means more generally that for any histories $h$ and $h'$ such that $h'$ is a final segment (or suffix) of $h$, then $h$ and $h'$ yield the same limit value. This easily seen to be true for the $\limsup$ limit rule and we will say that it is \emph{asymptotical}.

This quick glance into the characteristics usual limit rule and which are implicitly used in the universal machine leads us to distinguish four characteristics (formalized in the next section): those $\limsup$ and $\liminf$ limit rules are \emph{cell-by-cell}, \emph{contraction-proof}, \emph{asymptotic} and \emph{stable}. And, as we will see, those conditions are sufficient for the universal machine to be easily described and may likely be necessary for it to exist in a satisfying form. To those, we can add the \emph{looping stability} which somehow equates to the fact that a machine can be seen to be looping without looking at its entire computation through $\mo$. For example, a $\limsup$ ITTM is seen to be looping if there are two stages sharing the same snapshot such that cells that are set to $0$ at both of those stages are also set to $0$ for the whole segment of computation that spans between those two stages.
Without it, we could easily conceive a pathological machine that repeats for some gigantic ordinal amount of time and after which, thanks to its rule, that escapes the repeating pattern.

As stated, the interesting result, which we show at the beginning of Section~\ref{sec:higher_order_ittms} is that the $\limsup$ and $\liminf$ rules are the only $2$-symbol rules that satisfy this set of five conditions. In other words, they likely are the only $2$-symbol rules allowing for natural simulations and having an easy-to-describe universal machine while also satisfying the condition of looping stability.
Further, as mentioned, the main theorem of this paper shows the following strong result: for a $n$-symbol rule definable by a set-theoretic formula and that satisfies those five conditions, the supremum of the accidentally writable ordinals matches with that of the a.c.\ ordinals, that is $\Sigma = \Tau$ for all those rules. In the last section, we provide a counter-example to show that this theorem is tight. Namely, we show that there exists a limit rule that satisfies all those conditions but that of looping stability and such that, for the machines ruled by it, the equality $\Sigma = \Tau$ does not hold.

\section{General definitions and conditions on operators}\label{sec:general_definitions}

In order to formaly study limit rules, we need first to define what a limit rule is. In \cite{welch_operator}, Welch introduces the concept of operator. A $n$-symbol limit rule can be seen as an operator $\Gamma: {}^{<On}(^\om n) \rightarrow {}^\om n$, that is a function that given any computation history seen as an ordinal-indexed limit sequence of $n$-symbol real numbers (corresponding to tape contents) yields a $n$-symbol real number that will represent the tape content at the next limit stage. 
As noted, the data of an operator $\Gamma$ is enough to produce a model of infinite time Turing machine whose limit stage behavior is ruled by this operator. Such machines will be called $n$-symbol $\Gamma$-machines. We start by introducing some general concepts. 

\subsection{General definitions on ordinal words}

\begin{definition}[Additively closed ordinals]
An ordinal $\alpha$ is \emph{additively closed} when, for all $\beta, \delta < \alpha$, $\beta + \delta < \alpha$. The additively closed ordinals are known to be the ordinals of the forms $\om^\beta$ for some ordinal $\beta$.
\end{definition}

\begin{definition}[Multiplicatively closed ordinals]
An ordinal $\alpha$ is \emph{multiplicatively closed} when, for all $\beta, \delta < \alpha$, $\beta \cdot \delta < \alpha$. The multiplicatively closed ordinals are known to be the ordinals of the forms $\om^{\om^\beta}$ for some ordinal $\beta$.
\end{definition}

\begin{definition}[Ordinal word]
An \emph{ordinal word} $w$ of length $\lambda$ on some alphabet $\Alpha$ is a function from $\lambda$ to $\Alpha$ that maps every ordinal $\iota < \lambda$, seen as a \emph{position} or \emph{index} in the word, to a letter in $\Alpha$. We write $w[\iota]$ for the letter at position $\iota$ in $w$ and $w[\iota, \kappa[$ the subword of $w$ staring at position $\iota < \kappa$ up to position $\kappa$ non included. We write $|w|$ for the length of the word $w$ and when $|w|$ is finite, this coincides with the definition of words in the ordinary theory of formal languages. When $|w|$ is a limit ordinal, we say that $w$ is a limit word. Eventually, we write ${}^{\alpha}\Alpha$ for the set of ordinals words of length $\alpha$ on the alphabet $\Alpha$ and we write ${}^{<\mo}\Alpha$ for the class of ordinal words of any length on the alphabet~$\Alpha$. That is:
\begin{align*}
{}^{<\mo}\Alpha = \bigcup_{\alpha \in \mo} {}^{\alpha}\Alpha
\end{align*}
\end{definition}

\begin{example}
Consider an infinite Turing machine $m$ with two symbols, $0$ and $1$, that computes with any limit rule and whose computation reaches some stage $\alpha$. For any cell $i$, this induces a function $C_i$ that maps $\beta \leqslant \alpha$ to the value of the cell $i$ at stage $\beta$. This also induces an ordinal word $h_i$ of length $\alpha+1$ on the alphabet $2 = \left\{0, 1 \right\}$ such that for any $\beta \leqslant \alpha$, $h_i[\beta] = C_i(\beta)$. Also, the content of the three tapes at any stage $\beta$ can be described, as an ordinal word $w_\beta \in {}^\om 2$, that is a word of length $\om$ on the alphabet $2$. This can simply be done using an usual encoding to describe the three tapes of length $\om$ in a single word of length $\om$.
Combining both those points of view, this computation induces an ordinal word $W \in {}^{\alpha+1}(^\om 2)$, that is a word of length $\alpha+1$ on the alphabet $^\om 2$ such that for any stage $\beta \leqslant \alpha$, $W[\beta] = w_\beta$.
\end{example}

\begin{definition}[Stutter-free]\label{def:stutter-free}
Let $h$ be an ordinal word on some alphabet $\Alpha$. We say that $h$ is stutter-free when for all $\alpha+1$ index of $h$, $h[\alpha] \neq h[\alpha+1]$; that is the $\alpha^{th}$ letter of $h$ is different from the $\alpha+1^{th}$.
\end{definition}

\begin{definition}[Suffixes and prefixes]
Given $u$ and $v$ two ordinal words, we say that $u$ is a \emph{prefix} or an \emph{initial segment} of $v$, written $u \sqsubseteq v$ when $|u| \leqslant |v|$ and for all $\iota < |u|$, $u[\iota] = v[\iota]$. We say that $u$ is a \emph{suffix} or a \emph{final segment} of $v$, written $v \sqsupseteq u$ when $|u| \leqslant |v|$ and there exists $\alpha$ such that $\alpha + |u| = |v|$ and for all $\iota \in \left[\alpha, |v| \right[$, that is for all $\iota = \alpha + \iota' < |v|$, we have $u[\iota'] = v[\iota]$.
\end{definition}

\begin{definition}[Operations on words]
Given two words $u$ and $v$ and $\alpha$ an ordinal, we write $u v$ for the concatenation of $u$ and $v$ of ordinal length $|u| + |v|$ and $u^\alpha$ for the word made from $u$ concatenated $\alpha$ times to itself, of length $|u| \cdot \alpha$.
\end{definition}

\subsection{General definitions on machines}

\begin{definition}[Operators and $\Gamma$-machines, \cite{welch_operator}]
For a natural number $n$, an \emph{$n$-symbol limit rule operator} $\Gamma$ is a (class) function $\Gamma: {}^{<On}(^\om n) \rightarrow {}^\om n$ that maps ordinal words on the alphabet $^\om n$, seen as computation histories of $n$-symbol machines, to element of $^\om n$, seen as tape contents. A \emph{$\Gamma$-machine} is an infinite Turing machine whose limit rule is that induced by the operator $\Gamma$.
\end{definition}

We now introduce basic definitions regarding $\Gamma$-machines. Naturally, as $\Gamma$-machines and Hamkins' ITTM have the same structure, most of those definitions are identical to the definition given in the specific context of Hamkins' ITTM.

\begin{definition}
We write $\Gamma_{\sup}$ and $\Gamma_{\inf}$ the operators associated respectively to the $\limsup$ and the $\liminf$ rule.
\end{definition}

\begin{definition}[Looping, \cite{hamkins_ittm}]
For $\Gamma$ a limit operator and $m$ a $\Gamma$-machine, we say that $m$ is \emph{looping} when the machine never halts and some interval of computation $\left[ \alpha, \alpha + \beta \right[$ repeats itself through the whole computation of $m$ after stage $\alpha$. That is for any stage $\nu \geqslant \alpha$, writing $\nu = \alpha + \beta \cdot \delta + \nu'$ with $\delta$ maximal, the snapshot of $m$ at stage $\nu$ is the same as that of $m$ at stage $\alpha + \nu'$.
\end{definition}

\begin{definition}[Writable real numbers, \cite{hamkins_ittm}]
For $\Gamma$ a limit operator, a real number $x$ is $\Gamma$-\emph{writable} if there is a $\Gamma$-machine which, when computing from the empty input, halts with $x$ written on its output.
\end{definition}

\begin{definition}[Converging, \cite{hamkins_ittm}]
For $\Gamma$ a limit operator and $m$ a machine computing on some input $y$ we say that $m$ \emph{converges} or \emph{converges definitively} when it never halts and when after some stage its output tape is never modified again. We say that $m$ \emph{converges up to some ordinal $\alpha$} when it does not halt before stage $\alpha$ and after some stage $\beta < \alpha$ its output tape is never modified before stage $\alpha$. Observe that the definition of converging involves only the output tape: often, converging machines will continue to modify their working tape through the whole computation.
\end{definition}

\begin{definition}[Eventually writable real numbers, \cite{hamkins_ittm}]
For $\Gamma$ a limit operator, a real number $x$ is $\Gamma$-\emph{eventually writable} if there is a $\Gamma$-machine which converges when computing from the empty input and has $x$ written on its output tape when the content of its output tape stabilized.
\end{definition}

\begin{definition}[Accidentally writable real numbers, \cite{hamkins_ittm}]
For $\Gamma$ a limit operator, a real number $x$ is $\Gamma$-\emph{accidentally writable} if there is a $\Gamma$-machine which, when computing from the empty input, at some stage has $x$ written on any of its tapes
\end{definition}

Then to capture ordinals using real numbers, we need some kind of encoding.

\begin{definition}
We say that a real number $x$ 
\emph{encodes} some ordinal $\alpha$ when the real number encodes a relation $\prec$ on $\om$ of order type $\alpha$ in the the following fashion.
\begin{align*}
i \prec j \longleftrightarrow x[\langle i, j \rangle] = 1
\end{align*}
\end{definition}

\begin{definition}[Writable ordinals, \cite{hamkins_ittm}]
For $\Gamma$ a limit operator, an ordinal $\alpha$ is $\Gamma$-\emph{writable} (resp.\ $\Gamma$-\emph{eventually writable} and $\Gamma$-\emph{accidentally writable}) if some $x$ that encodes $\alpha$ is $\Gamma$-writable (resp.\  $\Gamma$-eventually writable and $\Gamma$-accidentally writable.) We write $\lambda_\Gamma$, $\zG$ and $\SG$ respectively for the supremum of the $\Gamma$-writable, $\Gamma$-e.w.\ and $\Gamma$-a.w.\ ordinals.
\end{definition}

\begin{definition}[Clockable ordinals, \cite{hamkins_ittm}]
An ordinal $\alpha$ is $\Gamma$-\emph{clockable} if there is a $\Gamma$ machine computing from the empty input that halts a stage $\alpha$. That is, it computes for $\alpha$ steps and then, on its next transition, it reaches the state $q_{halt}$. We write $\gamma_\Gamma$ for the supremum of the $\Gamma$-clockable ordinals.
\end{definition}

With those definitions, what Hamkins and Lewis asked in \cite{hamkins_ittm} is whether $\lambda_{\Gamma_{\sup}} = \gamma_{\Gamma_{\sup}}$. As noted, Welch answered this question positively in \cite{welch_main}. More generally, we may wonder under which conditions on $\Gamma$ does the equality $\lambda_\Gamma = \gamma_\Gamma$ holds. As done in the discussion of Section \ref{sec:introducing_the_universal_machine}, we can get a better insight on the question by introducing the following extensions of the concept of clockable, akin to that of the concept of writable.

\begin{definition}[Eventually clockable ordinals]
An ordinal $\alpha$ is $\Gamma$-\emph{eventually clockable} ($\Gamma$\emph{-e.c.} or simply \emph{e.c.} when the context is clear) if there is a $\Gamma$-machine which, computing from the empty intput, never halts and whose output tape stabilizes at stage $\alpha$; that is, it changes its content upon reaching this stage and it never subsequently changes it. We write $\eta_\Gamma$ for the supremum of the $\Gamma$-eventually clockable ordinals.
\end{definition}

\begin{definition}[Accidentally clockable ordinals]
An ordinal $\alpha$ is $\Gamma$-\emph{accidentally clockable} ($\Gamma$\emph{-a.c.} or simply \emph{a.c.} when the context is clear) if there is a $\Gamma$-machine which, when computing from the empty input, writes at stage $\alpha$ on one of its tapes a real number that wasn't written at any previous stage of this computation on any tape. We write $\Tau_\Gamma$, capital $\tau$, for the supremum of $\Gamma$-accidentaly clockable ordinals.
\end{definition}

We sum those constants up in Figure \ref{fig:letters_writable_clockable_with_eta_tau}.

\begin{figure}[h]
\begin{center}
\begin{tabular}{l l l l}
\hline
          & simply & eventually & accidentally \rule{0pt}{2.5ex}\\
writable  &  $\lambda$ & $\zeta$ & $\Sigma$  \\
clockable &  $\gamma$  & $\eta$  & $\Tau$ \\
\hline
\end{tabular}
\end{center}
\caption{Greek letters associated to each definition.}
\label{fig:letters_writable_clockable_with_eta_tau}
\end{figure}

Now, as before, given some limit operator $\Gamma$, the question becomes: does $\Sigma_\Gamma = \Tau_\Gamma$? This is answered positively in Theorem \ref{thm:sigma_tau_looping_condition} for a wide range of operators. And we show in Corollary \ref{cor:zeta_eta} and Corollary \ref{cor:lambda_gamma} that $\Sigma_\Gamma = \Tau_\Gamma$ implies both the equalities $\zeta_\Gamma = \eta_\Gamma$ and $\lambda_\Gamma = \gamma_\Gamma$.

\subsection{Conditions on operators}

\subsubsection{Simulational operators}

\begin{definition}[History of computation]
Given an $n$-symbol operator $\Gamma$ and a $\Gamma$-machine $m$ that does not halt before stage $\alpha$, the \emph{history $H$ of the computation of $m$ up to $\alpha$} is an element of $^\alpha({}^\om n)$, that is, it is an ordinal word of length $\alpha$ on the alphabet made from $n$-symbol real numbers, and such that for all $\beta < \alpha$, $H[\beta] = C^m(\beta)$, meaning that the $\beta^{th}$ letter of $H$ is the content of the tapes of $m$ at stage  $\beta$ (representing the three tapes with a single real). For a given cell $i$, the \emph{cell history $h_i$ of the computation of $m$ up to $\alpha$} is an element of $^\alpha n$ such that for all $\beta < \alpha$, $h_i[\beta] = C_i^m(\beta)$. Further, $h_i = H|_i$, the restriction of $H$ to the cell $i$. 
When $\alpha$ is a limit ordinal, we say that $H$ (resp. $h_i$) is a \emph{limit history} (resp. a \emph{limit cell history}).
\end{definition}

\begin{definition}[Suitable operator, \cite{welch_operator}]\label{def:suitable_operator}
An operator $\Gamma$ is a \emph{suitable} $n$-symbol operator if there is a set-theoretic first-order formula $\varphi(x_1, x_2, x_3, x_4)$ such that for any machine $m$, input $y$, cell $i$, symbol $k$ and stage $\nu$, writing $H_\nu$ the history of $m$ up to stage $\nu$ and $L_\nu$ the $\nu^{th}$ level of the constructible hierarchy, we have
\begin{align*}
\Gamma(H_\nu)[i] = k \iff L_{\nu}[y] \models \varphi(i, m, y, k)
\end{align*}
When $\varphi$ is a $\Sigma_n$ formula, it is a \emph{$\Sigma_n$ operator} and it produces \emph{$\Sigma_n$ machines}.
\end{definition}

\begin{remark}
We may want to restrict the symbols produced at limit stages to $0$ and $1$, that is $\Gamma$ would be an operator which acts from ${}^{<On}(^\om n)$ to ${}^\om 2$.
In this case, it is enough to provide a first-order formula  $\varphi(x_1, x_2, x_3)$ such that, with the same notations,
\begin{align*}
\Gamma(H_\nu)[i] = 0 \iff L_{\nu}[y] \models \varphi(i, m, y)
\end{align*}
\end{remark}

From there, we can formalize the different conditions exhibited and informally defined in the previous section.

\begin{definition}[Stable]
An operator $\Gamma$ is \emph{stable} when for any real number $x$ and limit ordinal $\alpha$,
we have
\begin{align*}
\Gamma(x^\alpha) = x
\end{align*}
\end{definition}

\begin{definition}[Cell-by-cell]
An operator $\Gamma$ is \emph{cell-by-cell} when there exists $\gamma : {}^{<On}n \rightarrow n$ such that, given any limit history $H$ and for all cell $i$, we have
\begin{align*}
\gamma(H|_i) = \Gamma(H)|_i
\end{align*}
where $|_i$ denotes the restriction of the tape history $H$, possibly made of a single tape content, to a single cell; hence yielding a single cell history, that is an element of ${}^{<On} n$. We say that $\gamma$ is the cell restriction of $\Gamma$.
\end{definition}

\begin{definition}[Asymptotic]
A limit rule operator $\Gamma$ is \emph{asymptotic} when for any limit history $H$ of length $\alpha$ and any non-empty final segment $H'$ (with previously introduced notation, $H \sqsupseteq H'$), we have $\Gamma(H) = \Gamma(H')$
\end{definition} 

\begin{example}\label{ex:asymptotic} Let us see how a $\Gamma$-machine $m$ behaves when $\Gamma$ is asymptotic.
For the sake of simplicity, we suppose that $m$ only has a working tape. We write $H(\alpha)$ for the history word of length $\alpha$ of the working tape up to some stage $\alpha$.
Again to make thing simpler, suppose that at stage $0$ some real number $x$ can be found on the working tape and the machine is in the state $q_{\lim}$. Suppose further that there is a limit stage $\alpha > 0$ such that $\Gamma(H(\alpha)) = x$. Then, we look at the computation of $m$ after this stage $\alpha$ and we show how this gives rise to a repetition in the computation; but a repetition that may not be a loop.

By definition of the model of the ITTM, at any stage $\beta + k$ for a finite $k$, the snapshot of $m$ depends only of its code and of its snapshot at stage $\beta$. In this case, as only the working tape is used, in the computation of $m$ stages $0$ and $\alpha$ share the same snapshot. 
This means that, for any finite $k$, the snapshots at stage $k$ and $\alpha + k$ match. That is, $H(\alpha + \om) \sqsupseteq H(\om)$. Then, what happens at limit stage $\alpha+\om$? By asymptoticity, $\Gamma(H(\alpha + \om)) = \Gamma(H(\om))$. And so, snapshots at stages $\om$ and $\alpha + \om$ also match.
Further, inductively, we can show that $\Gamma(H(\alpha \cdot 2)) = \Gamma(H(\alpha)) = x$. And further that $\Gamma(H(\alpha \cdot k)) = \Gamma(H(\alpha)) = x$ for any finite $k$. Is then $m$ actually looping? Consider stage $\alpha \cdot \om$. This stage is the first stage after $\alpha$ whose ordinal number is additively closed. Consequently, by additive closeness, any non-empty final segment of $H(\alpha \cdot \om)$ has itself $H(\alpha \cdot \om)$ as final segment. Hence applying asymptoticity in this case only yields that $\Gamma(H(\alpha \cdot \om)) = \Gamma(H(\alpha \cdot \om))$ and without more assumptions, the limit rule may very well make $m$ exit the repeating pattern at stage $\alpha \cdot \om$.
\end{example}

For the next condition, we need a prior definition. We will define the contraction of a word $w$, that is the stutter-free word which is made by squeezing any segment in $w$ in which a symbol is repeated into a single symbol, until no more repetition can be found. For example, the contraction of the word $aabc^\om$ would be $abc$.
We define the operator $\operatorname{ctr}$, which maps a word to its contraction.

\begin{definition}[Contraction]\label{def:contraction}
We define $\operatorname{ctr}$ for ordinal words as follows. Let $w$ be a word on some alphabet $A$ of length $\alpha$.
Then, there exists an ordinal sequence $(\nu_\mu) \subset \alpha$ of length at most $\alpha$, unbounded in $\alpha$ and starting at $0$, which describes the blocks of equal symbols in $w$.
That is, for $\nu_\mu$ and $\nu_{\mu+1}$ in this sequence,
\begin{align*}
\exists a \in A \, \forall \iota \in \left[ \nu_\mu, \nu_{\mu+1} \right[ \, w[\iota] = a \wedge w[\nu_{\mu+1}] \neq a \\
\end{align*}

With this sequence, $\operatorname{ctr}(w)$, the contraction of the word $w$ is the word of length $|(\nu_\mu)|$ (or of length $|(\nu_\mu)|-1$ if the length of the sequence is not a limit ordinal) in which each individual block is replaced by a single letter, that is such that for $\iota < |(\nu_\mu)|$,
\begin{align*}
\operatorname{ctr}(w)(\iota) = w(\nu_\iota)
\end{align*}
\end{definition}

\begin{definition}
We say that two words $u$ and $v$ are \emph{equal after contraction}, written $u \simeq_{ctr} v$, whenever $\operatorname{ctr}(u) = \operatorname{ctr}(v)$
\end{definition}

Below are a few examples of contraction of words and words equal up to contraction. This wholly coincides with the intuitive idea of contracting every sequence of a symbol which repeats itself into a single symbol. 

\begin{example}~
\begin{itemize}[label={--}]
\item $aaabcccc \simeq_{ctr} abbbc$.
\item The contraction of the limit word $b^\om$, that is of the limit word in which the letter $b$ is repeated $\om$ times, is the single-letter word $b$, made from the letter $b$.
\item For $\alpha, \beta > 0$ and any $\delta$, $(b^\beta a^\alpha)^\delta \simeq_{ctr} (ba)^\delta$.
\end{itemize}
\end{example}

\begin{definition}[Contraction-proof]
A operator $\Gamma$ is \emph{contraction-proof} when for any limit histories $H$ and $H'$ such that $H \simeq_{ctr} H'$, we have $\Gamma(H) = \Gamma(H')$.
\end{definition}

\begin{remark}
As the contraction of a limit word may be a non limit word (take for example $b$, the contraction of $b^\om$), general operators should not only be defined on limit words. However, any non limit word $w$ finishing with some letter $a$ is equal up to contraction to $w a^\om$. So, when considering contraction-proof operators, as we will do, it is enough for those to be defined only on limit words.
\end{remark}

\begin{remark}\label{rmk:refining_with_cell_by_cell}
Observe that for a cell-by-cell operator, those three characteristics, namely being stable, asymptotic or contraction-proof, are naturally refined to cell-wise properties. For example, if we take the case of the property of a cell-by-cell operator $\Gamma$ being contraction-proof and write $\gamma$ its associated cell restriction. Take $h$ and $h'$ two limit cell histories such that $h \simeq_{ctr} h'$. Then we can consider $H$ and $H'$ such that for all $i$, $H|_i = h$ and $H'|_i = h'$. By the properties of being cell-by-cell and contraction-proof, we have:
\begin{align*} 
\gamma(h) = \gamma(H|_i) = \Gamma(H)|_i = \Gamma(H')|_i = \gamma(H'|_i) = \gamma(h')
\end{align*}
which is simply the fact that $\gamma$ is contraction-proof as well, in a cell-by-cell fashion. Conversely, if $\gamma$ satisfy this cell-by-cell contraction-proof property, it is clearly carried over to $\Gamma$.
\end{remark}

\begin{definition}[Simulational operator]
We say that a limit rule operator $\Gamma$ is a \emph{simulational} limit rule operator when it is stable, cell-by-cell, asymptotic and contraction proof.
\end{definition}

\begin{prop}
Let $\Gamma$ be a simulational limit rule operator. Then there exists a universal machine $U_{\Gamma}$ and more generally, a $\Gamma$-machine can be designed to compute using the simulation of any other $\Gamma$-machine.
\end{prop}
\begin{proof}
We use the same reasoning as in the description of the $\limsup$ universal machine. It is easy, in virtue of the fact that $\Gamma$ is cell-by-cell and given the code of the finite universal machine $u$, to describe a universal $\Gamma$-machine $\mathcal{U}_0$ that simulates a single other $\Gamma$-machine. From there the description of $U_{\Gamma}$ from $\mathcal{U}_0$ is as previously: we can arrange to virtually split the working tape into $\om$ virtual workings tapes in each of which, for $m \in \om$, we run 
$\mathcal{U}_0$ that simulates $m$. We can also run those machines in cascade (that is $n$ steps of the $n^{th}$ first machine, $n+1$ steps of the $(n+1)^{th}$ first machine and so on) so that in $\om$ steps, we run $\om$ steps of each of those machines, and more generally at any $\alpha$ limit the machine simulated $\alpha$ steps of each simulated machine. And this produces the desired computation, that is at limit stages, the tapes of the simulated machine have the same content as the machines would have because $\Gamma$ is contraction-proof.

Now for the second part. What is a natural way for a machine to use the simulation of another? As briefly sketched, it should first be able to run the simulation of any machine from any snapshot and for as long it wants. And second, it should be able to put this simulation on hold also for as long as it wants. As $\Gamma$ is cell-by-cell, asymptotical and contraction-proof, a machine can run simulations of other machine from any snapshot. As it is stable and contraction-proof, it can put the simulation on hold for as long as it wants. 
\end{proof}

\begin{remark}\label{rmk:hypermachines}
Let us now look at the hypermachines developed by Welch and Friedman in \cite{hypermachines}, under the scrutiny of those new notions. Briefly, the structure of the hypermachine is akin to that of a $2$-symbol ITTM with the difference that it has a new tape called a rule tape. At a limit stage, the values of the cells are computed by a discretionary $\liminf$ operator, that is a $\liminf$ operator that considers only a precise subset of the previous stages over which it takes the $\liminf$. And it is the rule tape which dictates, at a given limit stage, which subset of previous stages should be considered in the $\liminf$. 

Now, as this operator needs to look at the whole rule tape, it is not cell-by-cell. However, it is easily seen to be stable. Then, if it was asymptotic, at some limit stage $\alpha \cdot 2$, the limit rule would not depend on the first half of the history. However, the notions of $1$-stability and $1$-correctness, used to define the operator of the hypermachine, are themselves not asymptotic. With those, it is possible to devise an history of computation for the hypermachine which would contradict asymptoticity.
Similarly, those definitions are not contraction-proof. Hence, the behavior of the limit rule may change after a contraction or a dilation of the history.
A more detail analysis of the hypermachine under those four definitions can be found in \cite[Remark 5.3.34]{these}.

In the end, this shows that is not possible to define a universal machine for the hypermachines with the technique used in the previous proposition. It may still exist but the presence of the rule tape makes a constructive approach seem difficult.

\end{remark}

The following proposition can be seen as a closure property. It is of the utmost importance and is ubitquitous in what follows, often under the guise of the formulation given in Remark \ref{rmk:dovetailing}.

\begin{prop}\label{prop:dovetailing_sim_operator}
Let $\Gamma$ be a simulational limit rule operator. If a real number $x$ is $\Gamma$-e.w.\ and $y$ is $x$-$\Gamma$-writable, then $y$ is $\Gamma$-e.w.
\end{prop}
\begin{proof}
This is proved as with classical ITTMs. Let $m_x$ be a machine which eventually writes $x$ and $m_y$ a machine which $x$-writes $y$.
With those, we can design a computation which, in parallel, using a classical dovetailing technique, simulates $m_x$ and tries to writes $y$ using $m_y$ and, as input for $m_y$, whatever is written on the output of $m_x$. It restarts the simulation of $m_y$ each time the output of $m_x$ changes. At some point, $m_x$ eventually wrote $x$ on its output and this output won't be modified again. From this point onward, the computation of $m_y$ will be conducted using $x$, which will eventually yield $y$. See \cite[Theorem 3.5.10]{merlin} and \cite[Proposition 6.3]{winter} for more details.
\end{proof}

\begin{remark}\label{rmk:dovetailing}
Proposition \ref{prop:dovetailing_sim_operator} justifies the following formulation widely used in the following proofs:
\begin{quote}
To show that $y$ is eventually writable, we consider the following machine: first it simulates $m_x$ to eventually write $x$ on some part of its working tape. Then using $x$ it conducts the following computation to write $y$ on its output...
\end{quote}
What is hidden behind this slightly abusive de-interlacing or sequencing of the different writing operations is the more complex but sound dovetailing technique refered to in the previous proof.
\end{remark}

\begin{prop}\label{prop:lim_sup_suitable}
The operator $\Gamma_{\sup}$ is a suitable and simulational $\Sigma_2$ operator. 
\end{prop}
\begin{proof}

To show that $\Gamma_{\sup}$ is a suitable $\Sigma_2$ operator, it is enough to provide a $\Sigma_2$ formula defining the $\limsup$ operator in the sense of Definition \ref{def:suitable_operator}. The only technical part is the fact that this formula does not directly take as argument the history of the machine and that the history is itself defined by this same formula. This can be done in a rather straightforward fashion akin to $\Sigma$-recursion. A more detailed construction of this $\Sigma_2$ formula can be found in the proof of Proposition 5.3.37 in \cite{these}.

As for the fact that it is a simulational operator: it is stable, asymptotical and cell-by-cell by definition and it is clearly enough contraction-proof. 

\end{proof}

\subsubsection{Looping stability}

The possibility for an infinite machine to exit a repeating pattern, as presented in Example \ref{ex:asymptotic}, is in a way what enables those machines to compute for so long and to give rise to intricate computations. At the same time it may lead to pathological patterns from the point of view of those being computation models. As presented in Remark \ref{rmk:looping_stability}, this could result in a $\Gamma$-machine which repeats for a very long time (very long when compared to the time it took to enter this repeating pattern) and then suddenly exits this pattern. To avoid those cases, we introduce in  Definition \ref{def:looping_stability} the notion of \emph{looping stability}. It is a simple assumption that let us develop a rather general and sound subclass of the simulational operators for which we can prove strong results like Theorem \ref{thm:sigma_tau_looping_condition}.
Observe that it thus excludes some more exotic models of ITTM, for example cardinal-recognizing ITTMs as introduced in \cite{habic}.

\begin{definition}[Looping stability]\label{def:looping_stability}
A limit rule operator $\Gamma$ satisfies the \emph{looping stability} condition when for any limit history $H$ and limit ordinals $\alpha$ and $\alpha'$, writing $H^\alpha$ for the word in which $H$ is repeated $\alpha$ times, we have:
\begin{align*}
\Gamma(H^\alpha) = \Gamma(H^{\alpha'})
\end{align*}
Equivalently, in a formulation that will be used more often, for any limit history $H$ and limit ordinal $\alpha$:
\begin{align*}
\Gamma(H^\alpha) = \Gamma(H^\om)
\end{align*}
\end{definition} 

\begin{remark}\label{rmk:looping_stability}

Why do we need looping stability? Without it, it seems impossible to establish a general looping condition as was established in the $\Sigma_2$ case and as we'll do in Proposition \ref{prop:looping_condition_gamma}, for the operators satisfying the condition of looping stability. By \say{looping condition}, we mean a criterion that lets us decide, by looking at at some point the history of a machine (so looking at a given and fixed past of the computation), whether the machine is at this point seen to be looping (and so producing a statement about the future of the computation). In other word, it is a criterion that  allows us to decide whether a machine is looping without having to look at its whole computation through $\mo$.
To give an idea of what would be possible without this hypothesis and the difficulty it poses, imagine $\gamma$ a cell-by-cell operator acting as the $\limsup$ operator with the only difference being that for some gigantic additively closed $\tau$, $\gamma((01)^{\tau}) = 0$. 
This operator would be stable and asymptotic. We could also arrange for it to be contraction proof. And with $\tau > \Sigma_{\sup}$, it would be repeating from stage $\Tau_{\sup} = \Sigma_{\sup}$ onward and up to stage $\tau$ where it would magically exit this loop. 

In particular, refining this idea, Theorem \ref{thm:contre_exemple_machine} shows that the main theorem of this section, that is that under some conditions on $\Gamma$, $\Sigma_\Gamma = \Tau_\Gamma$, does not hold without the hypothesis of looping stability.
\end{remark}

As suggested, the first benefit of the condition of looping stability is that it gives a general looping condition for the operators of the class.
\begin{prop}\label{prop:looping_condition_gamma}
Let $\Gamma$ be an operator which is looping stable and asymptotic.
Then the following looping condition holds: a $\Gamma$-machine is looping if and only if there are two ordinals $\alpha$, $\beta$ such that stages $\alpha$, $\alpha + \beta$ and $\alpha + \beta \cdot \om$ share the same snapshot.
\end{prop}
\begin{proof}
This is a direct application of looping stability.
\end{proof}

The following corollary will offer a bit more flexibility when using Proposition \ref{prop:looping_condition_gamma}.

\begin{cor}\label{cor:looping_condition_gamma}
Let $\Gamma$ be an operator which is looping stable and asymptotic.
Then the following looping condition holds: a $\Gamma$-machine is looping if and only if there are two ordinals $\alpha$, $\beta$ such that $\alpha$ and $\alpha + \beta$ share the same snapshot and such that the snapshot $\alpha + \beta \cdot \om$ also appears between stages $\alpha$ and $\alpha+\beta$.
\end{cor}
\begin{proof}
Let $\alpha'$ be the least stage between stages $\alpha$ and $\alpha+\beta$ at which the snapshot of stage $\alpha + \beta \cdot \om$ also appears.
As $\alpha$ and $\alpha+\beta$ share the same snapshot, by asymptoticity, the computation segment $[\alpha, \alpha+\beta[$ is the same as the computation segment $[\alpha+\beta, \alpha+\beta \cdot 2[$. In particular, the snapshot of stage $\alpha + \beta \cdot \om$ also appears at some stage $\alpha' + \beta'$ between stages $\alpha+\beta$ and stages $\alpha+\beta\cdot 2$ and such that $\alpha' + \beta' \cdot \om = \alpha + \beta \cdot \om$. And so, $\alpha'$ and $\beta'$ satisfy the looping condition of Proposition \ref{prop:looping_condition_gamma}.
\end{proof}

Further, Proposition \ref{prop:looping_condition_gamma} justifies the following definition.
\begin{definition}
Let $\Gamma$ be an asymptotic operator that satisfies the looping stability and $m$ a $\Gamma$-machine. We say that $m$ is \emph{seen to be looping at some stage $\delta$} if there are two ordinals $\alpha$ and $\beta$ such that stages $\alpha$, $\alpha + \beta$ and $\alpha + \beta \cdot \om$ share the same snapshot and such that $\alpha + \beta \cdot \om = \delta$.
\end{definition}

\begin{prop}\label{prop:new_snapshot}
Let $\Gamma$ be an operator that satisfies the looping stability as well as being asymptotic and let $m$ be a $\Gamma$-machine that does not halt. Then at any stage $\alpha$, if the machine is not seen to be looping before stage $\alpha \cdot \om^\om$, there is a snapshot that did not appear strictly before stage $\alpha$ that appears strictly before stage $\alpha \cdot \om^\om$.
\end{prop}
\begin{proof}
Let $m$ be some machine and $\alpha$ an ordinal such that the machine is not seen to be looping before stage $\alpha \cdot \om^\om$. We look at its computation starting from stage $\alpha$ onwards. 
We consider the snapshot $s_0$ that appears in $m$ at stage $\alpha$. 
Either this snapshot appears for the first time, and we're done. Or this snapshot appears at some earlier stage $\alpha_0 < \alpha$.
This means that the snapshot $s_0$ (at stage $\alpha_0$) leads to another occurrence of the snapshot $s_0$ (at stage $\alpha$). By asymptoticity of $\Gamma$, the snapshot $s_0$ (at stage $\alpha$) must also lead to an ulterior occurrence of $s_0$ and so on, for at least $\om$ repetitions. After those $\om$ repetitions, we write $s_1$ for the snapshot of $m$. Observe first that $s_1 \neq s_0$ as otherwise $m$ would be seen to be looping by Proposition \ref{prop:looping_condition_gamma}.
And even stronger: by Corollary \ref{cor:looping_condition_gamma}, $s_1$ does not appear in the whole segment that spans between the two occurrences of $s_0$ and that repeats itself. Hence, again, if $s_1$ does not appear for the first time, there is some least stage $\alpha_1 < \alpha_0$ at which the snapshot $s_1$ occurs.

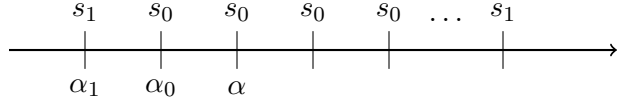
\begin{figure}[h]
\begin{center}
\begin{tikzpicture}
\draw[thick, ->]  (0,0) -- (8,0);

\draw (1, 0.25) -- (1, -0.25);
\node at (1,0.5) {$s_1$};
\node at (1,-0.5) {$\alpha_1$};

\draw (2, 0.25) -- (2, -0.25);
\node at (2,0.5) {$s_0$};
\node at (2,-0.5) {$\alpha_0$};

\draw (3, 0.25) -- (3, -0.25);
\node at (3,0.5) {$s_0$};
\node at (3,-0.5) {$\alpha$};

\draw (4, 0.25) -- (4, -0.25);
\node at (4,0.5) {$s_0$};

\draw (5, 0.25) -- (5, -0.25);
\node at (5,0.5) {$s_0$};

\node at (5.75,0.4) {$\ldots$};

\draw (6.5, 0.25) -- (6.5, -0.25);
\node at (6.5,0.5) {$s_1$};


\end{tikzpicture}\caption{The snapshot $s_0$ repeats $\om$ times and if the snapshot $s_1$ does not appear for the first time after this repetition it must have appeared before stage $\alpha_0$.}
\label{fig:repetition_s_Beta}
\end{center}
\end{figure}

With the same reasoning, as long as the repetition of the segments between the occurrences of $s_i$ does not yield a new snapshot, it yields a snapshot $s_{i+1}$ that appeared earlier, at some least ordinal stage $\alpha_{i+1} < \alpha_i$ (as, by hypothesis, the machine isn't seen to be looping yet). This yields a decreasing sequence of ordinals $(\alpha_i)$.
Consequently, by well-orderdness, there must be some snapshot $s_k$ that appears for the first time at some stage $\alpha_k$ for a finite~$k$. 
 
And how late can this snapshot occur? At worst, the first repeating segment between the occurrences of $s_0$ has length $\alpha$. Then the second repeating segment, between the occurrences of $s_1$ would have length $\alpha \cdot \om$. And, more generally, the repeating segment between the occurrences of $s_i$ would have length $\alpha \cdot \om^i$. This yields the bound
\begin{align*}
\sum_{i < \om} \alpha\cdot\om^{i+1} = \alpha \cdot \om^\om
\end{align*} 
for the appearance of a new snapshot.
\end{proof}

\begin{prop}\label{prop:gamma_machine_loops_countable_stage}
Let $\Gamma$ be a suitable operator that satisfy the looping stability as well as being asymptotic. Then there is some countable stage $\alpha$ such that any $\Gamma$-machine either stopped or is seen to be looping before stage $\alpha$.
\end{prop}

\begin{proof}
Consider ordinal $\omega_2$. By the downward Löwenheim-skolem theorem, we can find limit ordinals $\alpha_1 < \alpha_2 < \alpha_3$ of cardinality $\aleph_1$ and such that the $L_{\alpha_i}$ are elementary substructures of $L_{\omega_2}$. Moreover we can arrange for $\alpha_3$ to be strictly greater than $\alpha_2 \cdot \om$.  Further, by a result of Devlin \cite[II. 5.5]{devlin}, we can show in $L$ that no new real numbers appear in the constructible universe after construction stage $\om_1$. 
Hence, we have:
\begin{align*}
L_{\omega_2} \models \exists \alpha_1 < \alpha_2 < \alpha_3 & \ L_{\alpha_1} \prec L_{\alpha_2} \prec L_{\alpha_3} \\
&\wedge \Lim(\alpha_1) \wedge \Lim(\alpha_2) \wedge \Lim(\alpha_3) \\
&\wedge \alpha_2 \cdot \om < \alpha_3 \\
&\wedge \text{no new real number appear after construction stage } \alpha_1
\end{align*}

Again by the downward Löwenheim-Skolem theorem there exists some countable ordinal $\beta$ such that $L_{\beta} \prec L_{\omega_2}$. This yields three limit countable ordinals $\beta_1 < \beta_2 < \beta_3$ that satisfy the previous formula in $L_\beta$. We show that any $\Gamma$-machine $m$ is seen to be looping before stage $\beta_3$.

Let $m$ be a $\Gamma$-machine computing from the empty input. First, as $\Gamma$ is a suitable operator there is some predicate $\varphi(i, m, 0, k)$ that defines in $L_\alpha$, for any limit $\alpha$, the value of any cell $i$ of $m$ at limit stage $\alpha$ (the $0$ in $\varphi$ stands for the empty input). As $L_{\beta_1} \prec L_{\beta_2}$ (and as those are limit stages), whether in $L_{\beta_1}$ or $L_{\beta_2}$, the computation of the limit rule for some cell $i$ using $\varphi$ yields the same value. Hence, the snapshots of $m$ at stages $\beta_1$ and $\beta_2$ match. For convenience, we let $\beta_0 < \beta_2$ be the least ordinal stage that share the same snapshot as $\beta_2$ and we write $x_2$ the real number appearing at those two stages. As seen in Example \ref{ex:asymptotic}, by asymptoticity of $\Gamma$ this gives rise to a repeating pattern where the segment between stages $\beta_0$ and $\beta_2$ repeats at least $\om$ stages. Let $\beta'_2$ be the limit stage directly after those $\om$ repetitions of the segment of computation $[\beta_0, \beta_2[$ and $x'$ be the real number appearing at this stage. As $\beta'_2 \leqslant \beta_2 \cdot \om$, we also have that $\beta'_2 < \beta_3$. This state of affairs is represented in Figure~\ref{fig:schema_beta_1_2_3}.

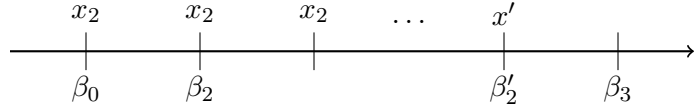
\begin{figure}[h]
\begin{center}
\begin{tikzpicture}
\draw[thick, ->]  (0,0) -- (9,0);

\draw (1, 0.25) -- (1, -0.25);
\node at (1,0.5) {$x_2$};
\node at (1,-0.5) {$\beta_0$};

\draw (2.5, 0.25) -- (2.5, -0.25);
\node at (2.5,0.5) {$x_2$};
\node at (2.5,-0.5) {$\beta_2$};

\draw (4, 0.25) -- (4, -0.25);
\node at (4,0.5) {$x_2$};

\node at (5.25,0.4) {$\ldots$};

\draw (6.5, 0.25) -- (6.5, -0.25);
\node at (6.5,0.5) {$x'$};
\node at (6.5,-0.5) {$\beta'_2$};

\draw (8, 0.25) -- (8, -0.25);
\node at (8,-0.5) {$\beta_3$};

\end{tikzpicture}\caption{The segment that spans between $\beta_0$ and $\beta_2$ repeats $\om$ times and produces $x'$.}
\label{fig:schema_beta_1_2_3}
\end{center}
\end{figure}

By Corollary \ref{cor:looping_condition_gamma}, $m$ is seen to be looping if $x'$ appears at a limit stage in the segment of computation $[\beta_1, \beta_2[$ (as it would then produce the same snapshot as the one appearing at stage $\beta'_2$.)  As $x'$ appears as a snapshot in $m$ at stage $\beta'_2$, it can be defined in $L_{\beta'_2}$. Hence $x' \in L_{\beta'_2+1}$. As $\beta'_2 < \beta_3$, $x'$ is also in $L_{\beta_3}$. So, with $x_2$ being the snapshot that appeared at stage $\beta_2$ and with $C^m(\alpha)$ the function that maps some ordinal stage $\alpha$ to the real number that describes the snapshot at stage $\alpha$ in the computation of $m$:
\begin{align*}
L_{\beta_3} \models \exists \beta_2 < \beta'_2 \, \Lim(\beta_2) \wedge \Lim(\beta'_2) \wedge C^m(\beta_2) = x_2 \wedge C^m(\beta'_2) = x'
\end{align*}

Now, and that is the corner stone of this proof, $L_{\beta_3}$ sees that no new real numbers appeared in the constructible universe after stage $\beta_1$. This means that $x'$ (as well as $x_2$) are in $L_{\beta_1}$ and in $L_{\beta_2}$ and that, by elementarity, we can reflect the previous sentence down to $L_{\beta_2}$. This means that, as depicted in Figure~\ref{fig:schema_delta_beta_1_2_3}, in the computation of $m$ up to stage $\alpha_2$, the real number $x'$ appeared at some limit stage $\delta'_2$ such that there is an earlier limit stage $\delta_2 < \delta'_2$ at which the real number $x_2$ appeared. By definition of $\beta_0$, it must be less than $\delta_2$ and $x'$ appeared at a limit stage in the segment of computation $[\beta_0, \beta_2[$. Hence, by Corollary \ref{cor:looping_condition_gamma}, $m$ is seen to be looping before stage $\beta_3$. 

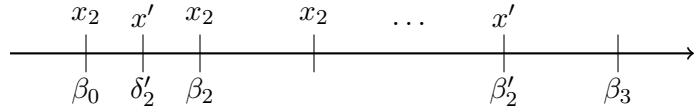
\begin{figure}[h]
\begin{center}
\begin{tikzpicture}
\draw[thick, ->]  (0,0) -- (9,0);

\draw (1, 0.25) -- (1, -0.25);
\node at (1,0.5) {$x_2$};
\node at (1,-0.5) {$\beta_0$};

\draw (1.75, 0.25) -- (1.75, -0.25);
\node at (1.75,0.5) {$x'$};
\node at (1.75,-0.5) {$\delta'_2$};

\draw (2.5, 0.25) -- (2.5, -0.25);
\node at (2.5,0.5) {$x_2$};
\node at (2.5,-0.5) {$\beta_2$};

\draw (4, 0.25) -- (4, -0.25);
\node at (4,0.5) {$x_2$};

\node at (5.25,0.4) {$\ldots$};

\draw (6.5, 0.25) -- (6.5, -0.25);
\node at (6.5,0.5) {$x'$};
\node at (6.5,-0.5) {$\beta'_2$};

\draw (8, 0.25) -- (8, -0.25);
\node at (8,-0.5) {$\beta_3$};

\end{tikzpicture}\caption{By elementarity, there is $\delta_2' < \beta_2$ such that $C^m(\delta_2') = x'$}
\label{fig:schema_delta_beta_1_2_3}
\end{center}
\end{figure}
\end{proof}

\section{Toward higher-order and many-symbol ITTMs}\label{sec:higher_order_ittms}

\begin{theorem}\label{thm:2_sym_machines}
$\Gamma_{\sup}$ and $\Gamma_{\inf}$ are the only simulational operators with two symbols that satisfy the looping stability condition.
\end{theorem}

\begin{proof}
First, as seen, $\Gamma_{\sup}$ is a suitable and simulational operator. With the same reasoning, so is its symmetric, $\Gamma_{\inf}$.
Then, let $\Gamma$ be a two-symbol suitable and simulational operator. As $\Gamma$ is cell-by-cell we can write $\gamma$ the cell restriction of $\Gamma$. We now want to show that $\gamma$ is either $\gamma_{\sup}$ or $\gamma_{\inf}$. So let $h \in {}^{<On}2$ a limit cell history on two symbols. We distinguish two cases
\begin{itemize}[label={--}]
\item If there is $s = 0$ or $s = 1$ and some $\alpha$ such that the word $s^\alpha$ is a final segment of $h$ then by asymptoticy $\gamma(h) = \gamma(s^\alpha)$ and by stability $\gamma(h) = s$.
\item Otherwise, this means that both $0$ and $1$ are cofinal in the word $h$. We show that all such $h$ are mapped to the same symbol. Take $h_1$ and $h_2$ two such history who have cofinally $0$'s and $1$'s. From $h_1$ and $h_2$ we can construct by transfinite induction $h'_1$ and $h'_2$ such that $h'_1$ and $h'_2$ are stutter-free and $h'_1 \simeq_{ctr} h_1$ and $h'_2 \simeq_{ctr} h_2$. As $\gamma$ is contraction-proof, $\gamma(h'_1) = \gamma(h_1)$ and $\gamma(h'_2) = \gamma(h_2)$. As $h'_1$ and $h'_2$ are stutter-free, they are simply the regular alternation of $0$ and $1$ : $010101\ldots$ or $101010\ldots$ Hence, there are two ordinals $\alpha$ and $\beta$ such that $h'_1 = (01)^\alpha$ (or $h'_1 = (10)^\alpha$) and $h'_2 = (01)^\beta$ (or $h'_2 = (10)^\beta$). In any case, by asymptoticity and looping stability (which can be applied cell-by-cell as $\Gamma$ is cell-by-cell; see Remark \ref{rmk:refining_with_cell_by_cell}.), $\gamma(h'_1) = \gamma((01)^\om) = \gamma(h'_2)$ and from there $\gamma(h_1) = \gamma(h_2)$. Hence those histories form a single equivalence class w.r.t.\ $\gamma$: either they are all mapped to $1$ and $\gamma = \gamma_{\sup}$. Either they are all mapped to $0$ and $\gamma = \gamma_{\inf}$.
\end{itemize} 
\end{proof}

First, observe that this result is really proper to $2$-symbol $\Gamma$-machines. With $3$ symbols or more, there are many intricate ordinal words that can't be simplified by contraction. Hence this theorem shows how, on the contrary to what we are used to in computability, working with only two symbols is actually a real constraint here. This shows the necessity to consider $n$-symbol machines. 
This will let us consider more complex rules that will still satisfy basic operator constraints, like being simulational.
Still, we will want to build on the previous result established for the $2$-symbol operators $\limsup$ and $\liminf$. To this extent, we will be interested in operators \emph{enhancing} other operators.

\begin{definition}
Let $\Gamma_n$ be an $n$-symbol operator and $\Gamma_k$ be a $k$-symbol operator with $k < n$. We say that $\Gamma_n$ \emph{enhances} $\Gamma_k$ when for any $H \in {}^{<\mo}({}^{\om} k)$ :
\begin{align*}
\Gamma_k(H) = \Gamma_n(H)
\end{align*}
\end{definition}

\begin{remark}
Enhancement is a powerful feature as it allows an operator to behave like another operator it enhances by simply sticking to a subset of symbols. In the case of enhancement of the operators $\Gamma_{\sup}$ or $\Gamma_{\inf}$ is it moreover relatively easy to obtain in the setting of well behaved operators. 

Indeed, take $\Gamma$ an $n$-symbol simulational and looping stable operator with $n>2$. If we suppose that at limit stages $\Gamma$ does not produce symbols for a cell that were not cofinal in the history of the cell (which is a relatively mild assumption that we refer to as \say{strong stability}), then $\Gamma$ enhances either the operator $\Gamma_{\sup}$ or $\Gamma_{\inf}$. To see this, just consider the restriction of $\Gamma$ to $0$ and $1$: it yields a simulational and looping stable $2$-symbol operator and, by Theorem \ref{thm:2_sym_machines}, it must be either $\Gamma_{\sup}$ or $\Gamma_{\inf}$

Still, as much as such enhancement is convenient we may not want to impose this hypothesis of strong stability on our operators. One main reason being that, even without this hypothesis, a simulational and looping stable operator can still emulate a classical ITTM in a fairly straightforward and faithful way. We will call this behavior \emph{emulation}.
\end{remark}

\begin{definition}\label{def:emulation}
	Let $\Gamma_n$ and $\Gamma_k$ be respectively an $n$-symbol operator and a $k$-symbol operator with $k < n$. We say that $\Gamma_n$ \emph{emulates} $\Gamma_k$ when for each $\Gamma_k$-machine $m_k$ there exists a $\Gamma_n$-machine $m_n$ such that:
	\begin{itemize}
		\item at any limit stage the snapshots of $m_k$ and $m_n$ match.
		\item all snapshots of $m_k$ appearing between two limit stages also appear between the same two limit stages and in the same order in $m_n$ but they may be separated by a finite amount of other snapshots.
		\item all snapshot of $m_n$ that are not from $m_k$ have a symbol $s \not\in k$ written in one of their cell.
	\end{itemize}
\end{definition}

With this definition, it is clear that any emulation of a machine will produce a very similar computation. In particular, with the notations of Definition \ref{def:emulation}, $m_k$ writes (or eventually writes) $x$ if and only if $m_n$ does. While $m_n$ may accidentally write more real numbers, this will be enough for our purpose as we will mostly want to harvest the deciding and writing power of $\limsup$ machines (deciding whether a real number is a code for an ordinal, writing a code for $L_\alpha$ given a code for $\alpha$, $\dots$) Proposition \ref{prop:gamma_emulates_limsup} motivates this definition and Proposition \ref{prop:enhancement_implies_emulation} links it with previous definition of enhancement.

\begin{prop}\label{prop:enhancement_implies_emulation}
If $\Gamma_n$ enhances $\Gamma_k$ then it also emulates it.
\end{prop}
\begin{proof}
With the notations of Definition \ref{def:emulation}, it is enough for $m_n$ to take $m_k$ seen as a $\Gamma_n$-machine.
\end{proof}

\begin{prop}\label{prop:gamma_emulates_limsup}
Let $\Gamma$ be an $n$-symbol simulational and looping stable operator. Then $\Gamma$ emulates the operator $\Gamma_{\sup}$.

\end{prop}

\begin{proof}
Let $\Gamma$ be an $n$-symbol simulational and looping stable operator. By Theorem \ref{thm:2_sym_machines}, we can suppose w.l.o.g.\ that $n>2$. For this proposition, the difficulty lies in the fact that there may not be two symbols in $n$ with respect to which $\Gamma$ acts like $\Gamma_{\sup}$. That is, it may be the case that for all $i, j \in n$, there is $k \in n-\left\{i, j \right\}$ such that
\begin{align*}
\Gamma((ij)^\om) = k
\end{align*}

To circumvent this issue, we use a simple enough trick. Consider the word $w = 1 2 3 \dots (n-1) 0$ made from all letters in $n$ starting from $1$. Let $k \in n$ such that $\Gamma(w^\om)=k$. With $w' = k (k+1) \ldots (n-1) 0 1 \ldots k-1$, all the letters in $n$ starting from $k$, observe that, by asymptoticty, $\Gamma(w^\om) = \Gamma((w')^\om) = k$. Hence, by renaming the symbols and arranging $w$, we can suppose that $\Gamma(w^\om)=1$.

With this, given a classical ITTM $m$, we describe the emulating machine $m_\Gamma$ of Definition \ref{def:emulation}. It behaves like $m$ with the only difference that when $m$ writes a $0$ over a $1$, $m_\Gamma$ writes, one after the other in the same cell: $2$, $3$, $\dots$, $n-2$, $n-1$ and finally $0$. When $m$ writes a $1$ over a $0$, $m_\Gamma$ simply does so.

This way, when the history of some cell in the classical ITTM $m$ reads $(10)^\om$, it will read $(w)^\om$ in the history of this cell in $m_\Gamma$. And by construction, the operator $\Gamma$ maps this history to $1$. That is, after the limit stage $\om$, the history of this cell in the $\Gamma$-machine reads $(w)^\om 1$ and the computation goes on. From there, by an induction relying on this reasoning and on the property of looping stability, we can show that the simulation carries on in a faithful way, respecting the conditions of Definition \ref{def:emulation}. 

\end{proof}

We will see that a large part of the results for the $\limsup$ machines involving either only writable or only clockable ordinals are true and proved with the exact same proof for $n$-symbol machines able to emulate the classical $2$-symbol machine. The following few results illustrate this.

The first two results also ensure that the writing and clocking constants are well-behaved and meaningful, as they are in the ITTM setting. Note that when possible we include the hypothesis that $\Gamma$ emulates $\Gamma_{\sup}$ (or indifferently $\Gamma_{\inf}$) rather than the stronger hypothesis of looping stability.

\begin{prop}
Let $\Gamma$ be a simulational $n$-symbol operator that emulates the operator $\Gamma_{\sup}$. Then the constants related to the three different kinds of $\Gamma$-writable ordinals are distinct. That is: $\lG < \zG < \SG$
\end{prop}
\begin{proof}

$\lG < \zG$ : At some point, all machines that write an ordinal have written their ordinal. So we consider the following machine: it simulates $\UG$ (which exists as $\Gamma$ is simulational) and at each step it writes the sum of all ordinals that have been written by any of the machines that halted. Observe that summing ordinals is possible by emulation of a $\Gamma_{\sup}$-machine.
When all writable ordinals have be written, i.e. when their respective machine stopped, the machine we are describing has eventually written their sum (ordered by the code of the machine) which is by definition greater or equal to $\lG$.

$\zG < \SG$ : With the same reasoning, at some point all eventually writable ordinal appeared in $\UG$. Computing the sum of all ordinals appearing in $\UG$ at the same time (we can't distinguish between those which are eventually writable and accidentally writable) will at some point accidentally write an ordinal greater or equal to $\zG$.
\end{proof}

\begin{prop}\label{prop:write_inf_clock}
Let $\Gamma$ be a simulational $n$-symbol operator that emulates the operator $\Gamma_{\sup}$. Then the supremum of any kind of writable is smaller than or equal to its respective kind of clockable. That is, $\lG \leqslant \gG$, $\zG \leqslant \eG$, $\SG \leqslant \TG$.
\end{prop}
\begin{proof}

Suppose that $\lG > \gG$, that is that $\gG$ is writable and consider the following machine toward a contradiction: it writes a code for $\gG$, then it emulates the $\Gamma_{\sup}$-machine that counts through an ordinal with this code to count for $\gG$ steps after which is stops. The emulations takes at least $\gG$ steps. So this machine effectively clocks some ordinal greater or equal to $\gG$, which is a contradiction.

Suppose that $\zG > \eG$, that is that $\eG$ is eventually writable and consider the following machine toward a contradiction: it simulates the machine that eventually writes a code for $\eG$ and, again using the emulation of a $\limsup$ machine, each time a code for an ordinal is written on its output, it counts through this ordinal and then copies the output of this machine to its own output. When the simulated machine stabilizes on a code for $\eG$, the main machine copies it on its output after at least $\eG$ steps, hence eventually clocks an ordinal greater or equal to $\eG$, which is again a contradiction.

As for the last inequality, it is proved exactly as in the first part of the proof of Proposition \ref{prop:sigma_2_tau_2}, replacing $\Sigma$ and $\Tau$ respectively by $\SG$ and $\TG$. Moreover, albeit slightly trickier, it is the same technique as used with the two previous inequalities.
\end{proof}

\begin{prop}\label{prop:closure_prop_enhancing}
Let $\Gamma$ be a simulational $n$-symbol operator that emulates the operator $\Gamma_{\sup}$. Then $\lG$, $\zG$ and $\SG$ are multiplicatively closed.
\end{prop}
\begin{proof} 
We show this for $\SG$. Let $\alpha$ and $\beta$ two accidentally writable ordinals. We want to show that $\alpha \cdot \beta$ is accidentally writable as well. Observe first that the product of two ordinals is computable with a $\Gamma$-machine. This comes from the fact that $\Gamma$ emulates the $\limsup$ operator: there is a $\limsup$ machine that computes this function, and this machine can be emulated by a $\Gamma$-machine.
Now suppose w.l.o.g.\ that (a code for) $\alpha$ appears earlier than (a code for) $\beta$ on one of the tapes of the machines simulated in $\UG$ (and $\UG$ exists as $\Gamma$ is simulational.) We consider the following computation: the machine simulate $\UG^1$, a copy of $\UG$, and at each step of $\UG^1$, writing $s^1$ for its snapshot, it launches $\UG^2$ a fresh copy of $\UG^2$ and it simulates this copy until it reaches the snapshot $s^1$. Meanwhile, it computes the product of any two ordinals appearing in $s^1$ and in $\UG^2$. When $\UG^1$ computed far enough so that $\beta$ appears in $s^1$, $\alpha$ appeared earlier and it will appear again in the computation of $\UG^2$ until it reaches snapshot $s^1$; hence the ordinal $\alpha \cdot \beta$ will be accidentally writable by this machine before $\UG^2$ reaches the snapshot $s^1$.

For $\lG$ and $\zG$, it relies in the same way on the fact that the product of two ordinals is computable, and it is easier as $\alpha$ and $\beta$ can be multiplied once eventually written (resp.\ written.)
\end{proof}

\begin{remark}
While, as illustrated above, many results on ITTMs are readily generalized to $\Gamma$-machines (provided basic hypothesis on $\Gamma$, e.g.\ being simulational), there is one interesting result which does not easily lend itself to generalization. As kindly pointed out by Bruno Durand, this is the case of the Speed-up Lemma of \cite{hamkins_ittm}. This lemma reads: \say{if $\alpha + k$ is clockable for $k \in \om$, then so is $\alpha$} and its proof is a clever argument which relies on combinatorial aspects of the $\limsup$ rule. It is clear that this argument can't be ported as is to $\Gamma$-machines and unclear whether it can be proved differently for those.

Thankfully, the consequences of this fact are limited. One reason is that when it comes to the behavior of infinite machines from a macro perspective (as mostly done here, where the ordinals we consider are limit ordinals closed under many different kind of operations), this lemma is not needed. It is however useful when conducting a finer analysis of gaps and particularly of their beginning and ends. While I reckon most of the analysis regarding the beginning of gaps (see e.g. \cite{durand_gaps_admissibles} and \cite{welch_summary}) and their links to admissible ordinals may be extended to $\Gamma$-machines, the absence of the Speed-up Lemma makes it less clear when it comes to the end of gaps. In particular, this leads to the following question: does there exist a simulational operator $\Gamma$ for which a gap ends at a non-limit ordinal stage?
\end{remark}

To specify the relationship between writable real numbers and the level of the constructible hierarchy, we can encode sets with real numbers.  This is done in a similar fashion as with ordinals. Observe however that, without consequences, it does not yield the exact same encoding for ordinals seen as well-orders and ordinals seen as sets.

\begin{definition}[Encoding of sets]
We say that a real number $x$ encodes a set $a$ when $x$ describes a transitive relation $E$ on $\om$ such that $(\om, E) \simeq (TC(\left\{ a \right\}), \in)$. As with ordinals, a set $a$ is writable (resp. eventually writable and accidentally writable) by a $\Gamma$-machine when a code for $a$ is writable (resp. eventually writable and accidentally writable) by a $\Gamma$-machine.
\end{definition}

\begin{prop}\label{prop:set_to_writable}
Let $\Gamma$ be a simulational $n$-symbol operator that emulates the operator $\Gamma_{\sup}$. If a set $a$ is
in $L_{\lG}$ (resp. in $L_{\zG}$ and in $L_{\SG}$) then it is writable (resp. eventually writable and accidentally writable) by a $\Gamma$-machine.
\end{prop}
\begin{proof}
 If $a \in L_{\SG}$, there is an accidentally writable ordinal $\alpha$ such that $a \in L_\alpha$. Hence, a code for $a$ is computable from a code for $L_\alpha$. Now consider the following computation that will accidentally write $a$ under the assumption that $a$ is in $L_{\SG}$: as $\Gamma$ is simulational, the machine can enumerate all ordinals below $\SG$. For each of those ordinals $\nu$, as $\Gamma$ emulates the $\limsup$ operator, the machine can write a code for $L_\nu$ on its working tape and try to extract the subcode that would be a code for $a$. Before $\alpha$ appears, what it thus produces is likely gibberish. When $\alpha$ appears, a code for $a$ is accidentally written. This works similarly if $a$ is in $L_{\zG}$ or $L_{\lG}$: the machine simply eventually writes or writes the right $L_\alpha$ and computes a code for $a$ from it.
\end{proof}

\begin{remark}
We may be tempted, under the hypothesis that $\Gamma$ emulates the operator $\Gamma_{\sup}$, to try and immediatly establish the converse implication of Proposition \ref{prop:set_to_writable} (which will be a corollary of Theorem \ref{thm:sigma_tau_looping_condition}). That is to show, for the first case, that if $x$ is accidentally writable then $x \in L_{\SG}$. It would simply build on the work which was already done for the operator $\Gamma_{\sup}$. Suppose that $x$ is accidentally writable by $\Gamma$, then as $\Gamma$ emulates $\Gamma_{\sup}$ a $\Gamma$-machine could use this $x$ to accidentally write greater and greater ordinals in $\Sigma^x_{\Gamma_{\sup}}$ (which is $\Sigma_{\Gamma_{\sup}}$ relativised to computations with $x$ as input). And as we might think that $x \in L_{\Sigma^x_{\Gamma_{\sup}}}$ (after all, $x$ is $x$-writable), this would mean, \emph{a fortiori}, that $x \in L_{\SG}$. But it happens that $x \in L_{\Sigma^x_{\Gamma_{\sup}}}$ simply does not hold in the general case (the fact that $x$ is $x$-writable simply says that $x \in L_{\Sigma^x_{\Gamma_{\sup}}}[x]$). Welch first considered in \cite{welch_operator} the (analogous) set \begin{align*}
F_0 = \left\{ x \in {}^\om 2 \mid x \in L_{\lambda^x_{\Gamma_{\sup}}} \right\}
\end{align*} which he explains is akin to the set 
\begin{align*}
Q =  \left\{ x \in {}^\om 2 \mid x \in L_{\omckx{x}} \right\}
\end{align*}
Both of those sets are thin (i.e. they do not admit non-empty and non-unit closed subsets) which implies that uncountably many real numbers $x$'s are not in $F_0$ or $Q$. A proof of which we give the idea can be found in \cite{kechris_countable_ana_sets}. Consider $F_0$ and suppose that it is not thin. That is, it contains a perfect subset $P$ (i.e.\ closed subset with not isolated points). $P$ is continuously isomorphic to ${}^\om 2$, so for our purpose we can assume that $P = {}^\om 2$. We then consider the application $f : x \mapsto \lambda^x$. It induces a pre-well-ordering $\preceq$ on ${}^\om 2$:
\begin{align*}
x \preceq y \iff \lambda^x \leqslant \lambda^y 
\end{align*}
This pre-well-order has order type $\om_1$ (the $\lambda^x$ are unbounded in $\om_1$) and is $\Sigma_1^1$ (under our the assumption that $P \subset F_0$ and modulo the isomorphism, which implies that $x \in \lambda^x$, we have that $\lambda^x \leqslant \lambda^y$ is equivalent to $x$ being $y$-computable.) Hence, this yields a Lebesgue measurable (as a subset of ${}^\om 2 \times {}^\om 2$) pre-well-order of the real numbers which can be shown with Fubini's theorem to yield a contradiction. 
\end{remark}

We now state the main theorem of this article.

\begin{theorem}\label{thm:sigma_tau_looping_condition}
Let $\Gamma$ be an $n$-symbol suitable, looping stable and simulational operator. Then ${\Sigma_\Gamma = \Tau_\Gamma}$.
\end{theorem}

We will prove this result step by step. For all that follows in this subsection, we let $\Gamma$ be an $n$-symbol suitable, looping stable and simulational operator. By Theorem \ref{prop:gamma_emulates_limsup}, we can suppose w.l.o.g.\ that it emulates the $2$-symbol operator $\Gamma_{\sup}$.

\begin{prop}\label{prop:eg_inf_TG}
The supremum of the $\Gamma$-e.c.\ ordinals is strictly less than that of the $\Gamma$-a.c.\ ordinals. That is, $\eG < \TG$.
\end{prop}
\begin{proof}
We show that in $\UG$ a new real number appears at stage $\eG$, that is that $\UG$ a.c.\ stage $\eG$. 

For this, observe that once a machine converges, that is once its output stabilizes definitely, it computes through a set of snapshot that never occurred in the computation of this machine before it converged. Indeed, suppose it did earlier, then by asymptoticity the machine would already have converged at this earlier stage. Hence this defines two disjoint sets of snapshot: those before the machine converged and those after, that we can call the converging snapshots. Moreover, as a machine e.c.\ the stage at which it converges, $\eG$ is the first stage at which all converging machines actually converged. Hence it is the first stage at which all the snapshots of those machine belong to their respective set of converging snapshots. This naturally stays true in their simulation done by the universal machine. Hence the real number on the working tape of the universal machine at stage $\eG$ appeared for the first time.
\end{proof}

\begin{prop}\label{prop:eg_inf_sg_imp_thm}
If the supremum of the eventually clockable ordinals is smaller than or equal to that of the accidentally writable ordinals, then the supremum of the accidentally writable ordinals is equal to that of the accidentally clockable ordinals. That is, if $\eG \leqslant \SG$ then $\SG = \TG$.
\end{prop}
\begin{proof}
We suppose that $\eG \leqslant \SG$. We show that, in this case, in any $\Gamma$-machine $m$, nothing new appears after stage $\SG$. This will show that $\SG \geqslant \TG$ and, by Proposition \ref{prop:write_inf_clock}, that $\SG = \TG$.

Suppose that there is some $\Gamma$-machine $m$, a real number $x_m$ and a least ordinal stage $\alpha_m \geqslant \SG$ such that $x_m$ appears for the first time on a tape of $m$ at stage $\alpha_m$. We consider the following computation: it simulates $m$ and at each step of $m$, writing $x$ the real number written on the relevant tape, it simulates a fresh instance of $\UG$. For each ordinal $\alpha$ accidentally written by $\UG$, its simulate a fresh copy of $m$ for $\alpha$ steps and looks whether $x$ appears in $m$ during those $\alpha$ steps. If $x$ appeared for the first time in $m$ before stage $\SG$, this subcomputation will eventullay find an ordinal below $\SG$ great enough so that $x$ is found in the simulation of $m$. Actually, $x_m$ is the first real number accidentally writable by $m$ that can't be found by this subcomputation. Effectively, the computation we designed eventually writes $x_m$ and eventually clocks some ordinal greater or equal to $\alpha_m$. By our assumptions, $\alpha_m$ is itself greater than $\SG$ and so greater than $\eG$. Hence the machine e.c. an ordinal greater than $\eG$ which is a contradiction.
\end{proof}

Now, in virtue of Proposition \ref{prop:eg_inf_sg_imp_thm}, to prove Theorem \ref{thm:sigma_tau_looping_condition}, it suffices to prove that the case $\zG < \SG < \eG < \TG$ leads to a contradiction. 
This case is represented in Figure \ref{fig:topology_writing_clocking_looping_stability}. In the rest of the proof, spanning almost until the end of the section, we suppose toward a contradiction that the constants of $\Gamma$ give rise to this situation.

\begin{figure}[h]
\begin{center}
\begin{tikzpicture}
\draw[thick, ->]  (0,0) -- (12,0);

\draw (2.5, 0.25) -- (2.5, -0.25);
\node at (2.5,0.5) {$\zG$};

\draw (5, 0.25) -- (5, -0.25);
\node at (5,0.5) {$\SG$};


\draw (8, 0.25) -- (8, -0.25);
\node at (8,0.5) {$\eG$};

\draw (11, 0.25) -- (11, -0.25);
\node at (11,0.5) {$\TG$};

\end{tikzpicture}
\caption{Situation of the main constants of the operator $\Gamma$ if $\SG < \eG$.}
\label{fig:topology_writing_clocking_looping_stability}
\end{center}
\end{figure}
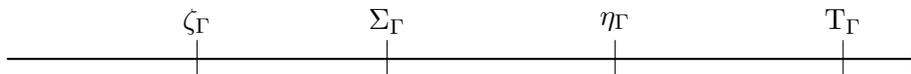

\begin{prop}\label{prop:reals_appear_cof}
In the computation of $\UG$, there are $\SG$ distinct snapshots appearing before stage $\SG$.
\end{prop}
\begin{proof}
As $\SG \leqslant \TG$ by Proposition \ref{prop:write_inf_clock}, $\UG$ is not seen to be looping before stage $\SG$. Then for any $\alpha<\SG$, a code for $\alpha \cdot \om^\om$ is writable from a code for $\alpha$. Hence $\alpha \cdot \om^\om < \SG$ and by Proposition \ref{prop:new_snapshot} for any $\alpha < \SG$ there is a new snapshot (that is distinct from all previous snapshots) that appears after stage $\alpha$ and before stage $\alpha \cdot \om^\om$. As $\SG$ is multiplicatively closed, we can repeat this reasoning with $\alpha \cdot \om^\om$ and so on; which yields $\SG$ distinct snapshots.
\end{proof}

\begin{prop}\label{prop:write_alphath_distinct_snapshot}
Let $x_\alpha$ be a code for some ordinal $\alpha < \SG$. Then the $\alpha^{th}$ distinct snapshot appearing in the computation of $\UG$ is $x_\alpha$-$\Gamma$-writable.
\end{prop}
\begin{proof}
We show that there is a machine that, given (the code of) some ordinal $\alpha$ as input, can write the $\alpha^{th}$ distinct snapshot appearing in $\UG$ if it exists and that never halts if it does not.
More precisely it will use $\alpha$ to write down the $\alpha+1$ first distinct snapshots of $\UG$.
The machine works as follows: using $\alpha$ given as input to arrange the information (that is it splits the working tape or some part of it into $\alpha$ virtual tape, each of those being able to store a real number describing a snapshot), it will inductively look for the $\beta$ first distinct snapshots for all $\beta < \alpha$. For $\beta = 0$, the first distinct snapshot of $\UG$ is its initial snapshot. Then, for any $\beta < \alpha$ such that the machine saved the $\beta$ first distinct snapshots, that is such that every $\delta < \beta$ correspond to a distinct saved snapshot, the machine simulates $\UG$ and, at each stage of this simulation, checks whether this snapshot is one of the snapshots saved as the $\delta^{th}$ distinct snapshot for some $\delta < \beta$. If the machine finds one that isn't part of those and consequently is a new snapshot, it saves it as the $\beta^{th}$ distinct snapshot. If the machine eventually finds $\alpha$ distinct snapshots, then the next distinct snapshot it finds (if it finds it) with the same procedure is the $\alpha^{th}$ snapshot it was looking for and it can stop. If at any point the machine does not find the wanted next distinct snapshots, it never halts.
\end{proof}

\begin{prop}\label{prop:snapshot_sgplusun_ew}
In the computation of $\UG$, the order-type of the distinct snapshots is greater than or equal to $\SG+1$ and the $\SG^{th}$ distinct snapshot (that is the last of the $\SG+1$ first snapshots) is eventually writable.
\end{prop}
\begin{proof}
As, by Proposition \ref{prop:eg_inf_sg_imp_thm}, the working hypothesis $\SG < \eG$ implies that $\SG < \TG$, there is some machine computation in which a real number appears for the first time after stage $\SG$. As this machine is simulated by the universal machine $\UG$, there is a least stage $\alpha \geqslant \SG$ such that a snapshot $s$ appears for the first time at this stage on the working tape of $\UG$. By Proposition \ref{prop:reals_appear_cof} this snapshot is the $\SG^{th}$ distinct snapshot (that is the last of the $\SG+1$ first snapshots) appearing in the computation of $\UG$.

We show that this snapshot $s$ is eventually writable. Consider this machine: it simulates $\UG$ in a simulation that we call $\UG^1$.  For each snapshot $s^1$ of $\UG^1$, it first copies it on its output and launches a new simulation $\UG^2$ to enumerate the ordinals $\alpha \in \SG$. More precisely, using $\UG^2$, the main machine looks for an ordinal $\alpha$ such that $s^1$ is the $\alpha^{th}$ distinct snapshot of $\UG$. This can be done using the machine described in the proof of Proposition \ref{prop:write_alphath_distinct_snapshot}. If such an $\alpha$ is found, the computation goes on with the simulation of $\UG^1$.
Now, when $s^1$, the snapshot of $\UG^1$, is the $\beta^{th}$ distinct snapshot of $\UG^1$ for $\beta \in \SG$, the simulation $\UG^2$ obviously yields at some point the correct $\alpha$, that is such that $\alpha = \beta$. On the other hand, when $\beta \geqslant \SG$ the simulation of $\UG^2$ never finds an $\alpha$ big enough and the machine never halts. This happens for the first time with the $\SG^{th}$ snapshot of $\UG^1$ and this snapshot is actually eventually writable, as wanted. 
\end{proof}

\begin{prop}\label{prop:topology_ordinals}
For the operator $\Gamma$, there exists some stage $\Alpha < \eG$ such that for each accidentally writable ordinals, there is a machine in which this ordinal appears before stage $\Alpha$.
\end{prop}
\begin{proof}
We use the previously proven fact that there is some eventually writable snapshot that appears after $\SG$ distinct snapshots appeared in $\UG$ to study the time of first appearance of ordinals in any computation. Let $s_\alpha$ be this eventually writable snapshot, as defined in the proof of Proposition \ref{prop:snapshot_sgplusun_ew}, and $\alpha \geqslant \SG$ that stage at which it appears.

We show using this snapshot $s_\alpha$ that there exists some stage $\Alpha < \eG$ such that all accidentally writable ordinals appear in $\UG$ before this stage $\Alpha$. While this is slightly different than the statement of the proposition, remember that at any limit stage $\nu$, inside of the computation of the universal machine $\UG$, all machine also reached stage $\nu$ of their simulation. In other words, in its simulations of other machines, the universal machine is never delayed by more than $\om$ steps. Hence, as $\eG$ is additively closed, it is equivalent to prove this result for $\UG$ or, as initially stated, for all machines individually.

So we consider the following computation of a machine that we call $\mathcal{M}$. It simulates $\UG^1$ a copy of $\UG$, and at each step of $\UG^1$, writing $s$ for its snapshot, $\mathcal{M}$ writes $s$ on its own output tape and launches $\UG^2$, a fresh copy of $\UG$, in order to enumerate the ordinals below $\SG$. For each ordinal $\beta$ appearing in this fresh copy of $\UG$, the main machine launches $\UG^3$, a third fresh copy of $\UG$ and checks whether the snapshot $s$ of the simulation of $\UG^1$ appears in the $\beta$ first steps of $\UG^3$. If it does, the simulation of $\UG^1$ carries on and this procedure is repeated with the next snapshots appearing in $\UG^1$.
If it does not, the simulation of $\UG^2$ continues with its enumeration of the accidentally writable ordinals. Hence, when the simulation of $\UG^1$ reaches stage $\alpha$, the snapshot $s_\alpha$ is written on the output tape of the main machine. As this snapshot does not appear before stage $\SG$, it will never be found by the simulation of $\UG^3$ that is only conducted through accidentally writable stages and $\mathcal{M}$ effectively eventually writes $s_\alpha$. Now the interesting part is: which ordinal is then e.c.\ by this computation?

Let $\Alpha \leqslant \TG$ be the least ordinal stage such that for each accidentally writable ordinal, there is a machine in which it appears before stage $\Alpha$.
We claim that if $\Alpha > \eG$, that is if there is some accidentally writable ordinal $\sigma$ that does not appear before stage $\eG$ in any machine, then the previous computation e.c.\ some ordinal greater than $\eG$. 
First, if such an ordinal $\sigma$ exists, all ordinals greater than $\sigma$ also appears after stage $\eG$. Then, as $\sigma$ is accidentally writable, $\sigma < \SG$ and by Proposition \ref{prop:reals_appear_cof} there is some snapshot $t$ that appears for the first time in $\UG$ after stage $\sigma$. Consequently, when $t$ appears in the first simulation of $\UG$ in $\mathcal{M}$, $\UG^2$, the second simulation of $\UG$, won't stop until it finds some ordinal greater than $\sigma$. And to do this it will compute for at least $\eG$ steps. Consequently, in such a case, when $\mathcal{M}$ converges, it has computed for more than $\eG$ steps and it e.c.\ an ordinal greater than $\eG$; which is a contradiction. Hence $\Alpha \leqslant \eG$. What is left to show is that $\Alpha$ can't be equal to $\eG$.

Suppose now that $\Alpha = \eG$: by minimality of $\Alpha$, for any $\alpha < \eG$, there is some accidentally writable ordinal $\sigma$ that does not appear in $\UG$ before stage $\alpha$. Moreover, by the definition of $\eG$, for any $\alpha < \eG$, there are machines that converge after stage $\alpha$. In particular there is some stage $\alpha_\zeta < \eG$ such that $\zG$ appears for the first time at stage $\alpha_\zeta$ and some machine $m$ that converges at stage $\alpha_m > \alpha_\zeta + \om$, as depicted in figure \ref{fig:cas_A_eta}. With it, we design the following computation $\mathcal{N}$: using a first simulation of $\UG$, that we call $\UG^1$, $\mathcal{N}$ enumerates the accidentally writable ordinals and for every ordinal $\alpha$ it finds, it does the following: it simulates in parallel a copy of $m$ and $\UG^2$, a new copy of $\UG$, until ordinal $\alpha$ or greater appears in $\UG^2$. As $\alpha$ appeared in the simulation $\UG^1$, it must appear at some point in $\UG^2$. When $\alpha$ has been found in $\UG^2$, it is written on the output of $\mathcal{N}$ and then the simulation of $m$ is carried on until there is a proof that $m$ had not yet converged. That is until the output of $m$ is modified for the first time after $\alpha$ has been found. When and if this happens, the main computation carries on with simulation of $\UG^1$ and starts this procedure again with the next $\alpha'$ it produces.

Now, as $m$ is a converging machine, it will have converged (that is reached the stage after which its output won't be modified anymore) after some $\alpha$ has been found when, before reaching the point at which this $\alpha$ is found, the parallel computation of $m$ and $\UG^2$ was carried for at least $\alpha_m$ steps. For this parallel simulation of $\UG^2$ and $m$ to be carried for at least $\alpha_m$ steps, it must have been looking for an ordinal greater than $\zG$; as $\zG$ appeared before, at stage $\alpha_\zeta$.
Observe also that $\zG$ appears at stage $\alpha_\zeta$ implies that any ordinal $\beta < \zG$ appears before stage $\alpha_\zeta + \om$. 
Moreover, by the assumption that $\Alpha = \eG$, ordinals great enough (that is greater than $\zG$ and further, great enough to appear for the first time after stage $\alpha_m$) are indeed found after stage $\alpha_m$. Consequently, at this point,
$m$ is definitely converging and so is $\mathcal{N}$. And this is only possible because such an ordinal greater than $\zG$ has been encountered in $\UG^1$, that $\UG^2$ then found at some computation stage greater than $\alpha_m$. So $\mathcal{N}$ has eventually written this ordinal greater than $\zG$, which is a contradiction.

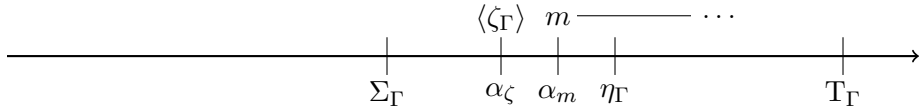
\begin{figure}[h]
\begin{center}
\begin{tikzpicture}
\draw[thick, ->]  (0,0) -- (12,0);


\draw (5, 0.25) -- (5, -0.25);
\node at (5,-0.5) {$\SG$};


\draw (6.5, 0.25) -- (6.5, -0.25);
\node at (6.5,-0.5) {$\alpha_\zeta$};
\node at (6.5,0.5) {$\langle\zG\rangle$};

\draw (7.25, 0.25) -- (7.25, -0.25);
\node at (7.25,-0.5) {$\alpha_m$};
\node at (7.25,0.5) {$m$};
\draw (7.5, 0.535) -- (9,0.535); \node at (9.4, 0.52) {\ldots};

\draw (8, 0.25) -- (8, -0.25);
\node at (8,-0.5) {$\eG$};

\draw (11, 0.25) -- (11, -0.25);
\node at (11,-0.5) {$\TG$};

\end{tikzpicture}
\caption{Case where $\eG = \Alpha$. Some code for $\zeta$ appears at stage $\alpha_\zeta$ and $m$ converges at an ulterior stage $\alpha_m$.}
\label{fig:cas_A_eta}
\end{center}
\end{figure}

\end{proof}

Now, by Proposition \ref{prop:topology_ordinals}, we obtain the state of affairs described in figure \ref{fig:topology_ordinals} for the $\Gamma$-machines.

\begin{figure}[h]
\begin{center}
\begin{tikzpicture}
\draw[thick, ->]  (0,0) -- (12,0);

\draw (2.5, 0.25) -- (2.5, -0.25);
\node at (2.5,0.5) {$\zG$};

\draw (5, 0.25) -- (5, -0.25);
\node at (5,0.5) {$\SG$};

\draw (7, 0.25) -- (7, -0.25);
\node at (7,0.5) {$\Alpha$};

\draw (8, 0.25) -- (8, -0.25);
\node at (8,0.5) {$\eG$};

\draw (11, 0.25) -- (11, -0.25);
\node at (11,0.5) {$\TG$};

\end{tikzpicture}
\caption{Situation of the main constants of the operator $\Gamma$. In this case, any accidentally writable ordinal appeared in some machine before stage $\Alpha$.}
\label{fig:topology_ordinals}
\end{center}
\end{figure}
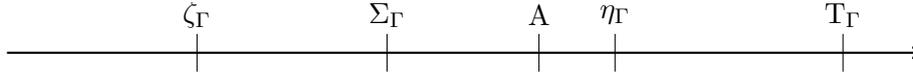

\begin{prop}\label{prop:sigma_model}
In the situation described in Figure \ref{fig:topology_ordinals} there is an eventually writable snapshot $s_{\Sigma}$ and a machine $m_{\Sigma}$ such that given as input $s_{\Sigma}$, any set theoretic formula $\varphi$ and any parameter $p \in L_{\SG}$, this machine decides whether $L_{\SG} \models \varphi(p)$.
\end{prop}
\begin{proof}
Let $m$ be a machine that converges after all accidentally writable ordinal appeared in $\UG$. Such a machine exists as, by Proposition \ref{prop:topology_ordinals}, $A < \eG$.

First we show that we can eventually write the snapshot at which $m$ converges (while the output of $m$ stabilizes at some point, its whole snapshot may not stabilize.) 
Consider the following computation: it simulates $m$ and at each stage where the output appears to have stabilized (that is the output did not change for at least one step), the main machine saves the snapshot of $m$. When $m$ starts to converge definitely, its output is never modified anymore and consequently neither does its saved snapshot. We call $s_\Sigma$ this particular snapshot.
What is now interesting is that when this snapshot appears in $m$, itself simulated in $\UG$, all ordinals appeared in $\UG$. Indeed, it does not appear earlier as otherwise, by asymptoticity of $\Gamma$, the machine would also be converging earlier.

Now we describe the machine $m_{\Sigma}$ that decides whether $L_\Sigma \models \varphi(p)$ with $\varphi$, $p$ and $s_\Sigma$ given as input. As $s_\Sigma$ is given as input, by looking at the simulation of $m$ inside $U_\Gamma$, $m_{\Sigma}$ knows when $m$ reaches snapshot $s_\Sigma$ and so, it knows when all ordinal below $\SG$ appeared. 
Moreover, as the operator $\Gamma$ emulates the usual $\Sigma_2$-rule, given a code for any ordinal $\alpha$, a code for $L_\alpha$ is computable from it. With this in mind, $m_{\Sigma}$ can use $\UG$ to enumerate $L_{\SG}$ and it will know when all elements of $L_{\SG}$ appeared in the iteration. Using this, it works inductively as follows: If $\varphi(p)$ is $\Delta_0$, it is absolute and the machine can directly evaluate it. 
Then, when $\varphi(p) = \exists x \psi(x, p)$ is a $\Sigma_n$ formula: it starts by simulating $\UG$ and it looks at the ordinals that appear in it until the simulation of $m$ inside $\UG$ reaches snapshot $s_\Sigma$. When this occurs, all ordinals below $\SG$ appeared and the machine effectively enumerated all accidentally writable ordinals. Meanwhile, before $s_\Sigma$ appears in $m$, for all ordinal $\alpha < \SG$ appearing in $\UG$, it enumerates $L_\alpha$. For each $x$ coding an element of $L_\alpha$ it inductively (on $\psi$) decides whether $L_{\SG} \models \psi(x, p)$. If it does, that is if $L_{\SG} \models \psi(x, p)$, then $L_{\SG} \models \varphi(p)$ and $m_{\Sigma}$ halts and outputs true. If it doesn't, the iteration of $L_\Sigma$ goes on. If at the end of the iteration, no $x$ such that $L_\sigma \models \psi(x, p)$ has been found, then it outputs false. And the case where $\varphi$ is $\Pi_n$ is dealt with by considering $\neg \varphi$. Observe also, as the induction is finite, that there is no difficulties organizing the information from the different recursive calls. 

\end{proof}

\begin{prop}\label{prop:extension_zG_SG}
$L_{\SG}$ is an end-elementary extension (e.e.e) of $L_{\zG}$. That is, $L_{\zG} \prec L_{\SG}$.  
\end{prop}
\begin{proof}
We show this by induction. By absoluteness $L_{\zG} \prec_{\Sigma_0} L_{\SG}$. Then let $\varphi$ be a $\Sigma_n$ formula such that for some $p \in L_{\zG}$, $L_{\SG} \models \varphi(p)$. We show that $L_{\zG} \models \varphi(p)$. 
By Proposition \ref{prop:set_to_writable}, $p$ is eventually writable. Writing $\varphi(p) = \exists x \psi(x, p)$ with $\psi$ a $\Delta_0$ formula, by Proposition \ref{prop:sigma_model}, given $p$, some $x$ and $s_\Sigma$, $m_\Sigma$ can decide whether $L_{\SG} \models \psi(x, p)$. As $p$ and $s_\Sigma$ are eventually writable, $\wt{x}$, the first $x$ appearing in $\UG$ such that $L_{\SG} \models \psi(x, p)$ is eventually writable as well. So we consider $\psi(\wt{x}, p)$, true in $L_{\SG}$. As both $\wt{x}$ and $p$ are eventually writable, they are in $L_{\zG}$ and we can apply the induction hypothesis: $L_{\zG} \models  \psi(\wt{x}, p)$. And so, $L_{\zG} \models \varphi(p)$.

Conversely, if $L_{\SG} \not\models \varphi(p)$, then $L_{\SG} \models \forall x \, \neg\psi(x, p)$. Hence for all $\wt{x} \in L_{\zG}$, $L_{\SG} \models \neg\psi(\wt{x}, p)$ with $\neg\psi(\wt{x}, p)$ $\Sigma_{n-1}$. So by the previous implication, $L_{\zG} \models \neg\psi(\wt{x}, p)$. That is, $L_{\zG} \models \forall x \, \neg\psi(x, p)$, i.e. $L_{\zG} \not\models \varphi(p)$.
\end{proof}

\begin{cor}\label{cor:zeta_model}
Given as input $s_{\Sigma}$, any set theoretic formula $\varphi$ and any parameter $p \in L_{\zG}$, the machine $m_{\Sigma}$ decides whether $L_{\zG} \models \varphi(p)$.
\end{cor}
\begin{proof}
By elementarity it is enough to decide whether $L_{\SG} \models \varphi(p)$.
\end{proof}

In this proof toward a contradiction we now have obtained fairly strong results. The fact that $L_{\zG} \prec L_{\SG}$ is clearly way too powerful to be established for any $\Sigma_n$ operator $\Gamma$ and the fact that we can decide whether $L_{\SG} \models\varphi(p)$ seems a direct contradiction of the intuition that $L_{\SG}$ is a bound on what the $\Gamma$-machines may apprehend. Still, reaching a conclusion from there is still somewhat involved. To harvest the previous result, Corollary \ref{cor:zeta_model}, we show how we can use a system of notations to virtually work with $L_{\zG \cdot 2}$ in $L_{\zG}$.

\begin{definition}[Notation]
We define a system of notation in $L_{\zG}$. We write $\top$ for some distinguished tautological sentence. A set $a$ is a notation if and only if:
\begin{itemize}
\item Either $a = \langle \top, x, 0, \beta \rangle$ for $x \in L_{\beta}$ and $\beta \in \zG$.
\item Or $a = \langle \varphi, b_1, \dots, b_n, \alpha, \beta \rangle$ where $\varphi$ is a set-theoretic formula (encoded by a set) without the symbol of equality, $\alpha > 0$ and $\beta$ are ordinals in $\zG$ and the ordered pair $(\alpha, \beta)$ is called the \emph{rank} of $a$, written $\rk(a)$, and $b_1, \dots, b_n$ are notations of rank strictly less (w.r.t.\ lexicographical order) than $(\alpha, \beta)$.
\end{itemize}
\end{definition}

\noindent We now use those notations to encode sets in $L_{\zG^2}$, the level $(\zG)^2$ of the constructible universe.
\begin{definition}[Notation of a set]
For $x \in L_{\zG^2}$, the \emph{least notation of the set $x$}, $\wb{x}$, is a notation inductively defined as follows.
\begin{itemize}
\item If $x \in L_{\beta}$ for some least $\beta \in \zG$:
\begin{align*}
\wb{x} \coloneqq \langle \top, x, 0, \beta \rangle
\end{align*}
\item If $x \in L_{\zG \cdot \alpha +\beta +1 } - L_{\zG \cdot \alpha +\beta}$, with $\alpha > 0$, then $x$ was defined over $L_{\zG \cdot \alpha +\beta}$ by some least formula $\varphi$ (w.r.t.\ some fixed order on formulas) and least parameters $p_1, \dots, p_n \in L_{\zG \cdot \alpha +\beta}$ (w.r.t.\ $<_L$). And by extensionality we can w.l.o.g. suppose that the symbol of equality is not used in $\varphi$. So we define:
\begin{align*}
\wb{x} \coloneqq \langle \varphi, \wb{p_1}, \dots, \wb{p_n}, \alpha, \beta \rangle
\end{align*}
\end{itemize}
\end{definition}

\begin{definition}[Set of a notation]
For a notation $a \in L_{\zG}$, the \emph{set of the notation $a$}, $\widehat{a}$, is inductively defined as follows:
\begin{itemize}
\item If $a = \langle \top, x, 0, \beta \rangle$:
\begin{align*}
\wh{a} \coloneqq x
\end{align*}
\item If $a = \langle \varphi, b_1, \dots, b_n, \alpha, \beta \rangle$ with $\alpha > 0$:
\begin{align*}
\wh{a} \coloneqq \left\{ y \in L_{\zeta \cdot \alpha + \beta} \mid L_{\zeta \cdot \alpha + \beta} \models \varphi(y, \wh{b_1}, \dots, \wh{b_n} ) \right\}
\end{align*}
\end{itemize}
\end{definition}

We draw the attention of the reader to some clear but important features of this system of notations.

\begin{prop}\label{prop:features_of_notation}~\
\begin{itemize}[label={--}]
\item The operator $x \mapsto \nott{x}$ defined on $L_{\zG^2}$ is injective, $\set{\nott{x}} = x$ and the operator $a \mapsto \set{a}$ is surjective onto $L_{\zG^2}$. That is, every set in $L_{\zG^2}$ has a notation. 
\item A notation only uses finitely many notations in its definition.
\item For a notation $a$ of rank $(\alpha, \beta)$, $\set{a} \in L_{\zG\cdot \alpha + \beta +1}$.
\item For $x \in L_{\zG\cdot\alpha + \beta+1}$ with $\alpha, \beta < \zG$, $x$ has a notation of rank less than $(\alpha, \beta)$.
\end{itemize}
\end{prop}
\begin{proof}~\
\begin{itemize}[label={--}]
\item By definition, it is clear that $x \mapsto \nott{x}$ is injective and that $\wh{\wb{x}} = x$. Hence $a \mapsto \set{a}$ is surjective onto $L_{\zG^2}$.
\item Observe that the inductive definition of a notation induces a tree in which each node is a notation. In this tree, every node has finitely many children and by well-orderdness every path is finite. Hence the tree itself is finite.
\item By induction: notations of rank $(0, \beta)$ clearly yield elements of $L_{\zG}$. For a notation $a = \langle \varphi, b_1, \dots, b_n, \alpha, \beta \rangle$ with $\alpha > 0$, by induction $b_1, \dots, b_n$ are notations for some sets $p_1, \dots, p_n$ (i.e. $\set{b_i} = p_i$) of rank strictly less than $(\alpha, \beta)$ and so they are in $L_{\zeta\cdot\alpha + \beta}$. Hence, as a set defined over $L_{\zeta\cdot\alpha + \beta}$, $\set{a}$ is in $L_{\zeta\cdot\alpha + \beta+1}$. 
\item Simply consider $\nott{x}$, the least notation of $x$. 
\end{itemize}
\end{proof}

\begin{prop}\label{prop:decide_with_notations}
There exists a $\Gamma$-machine such that: given as input $s_\Sigma$, a formula $\varphi$, two ordinals $\alpha, \beta \in \zG$ with $\alpha>0$ and notations $p_1, \dots, p_n$ of rank strictly less than $(\alpha, \beta)$, it decides whether
\begin{align*}
L_{\zeta\cdot\alpha + \beta} \models \varphi(\set{p_1}, \dots, \set{p_n})
\end{align*}
\end{prop}
\begin{proof}
We show this by induction on $(\alpha, \beta)$. We provide a description of a machine $\mathcal{M}$ that uses the inductive definition of the notations to inductively decide whether $L_{\zeta\cdot\alpha + \beta} \models \varphi(\set{p_1}, \dots, \set{p_n})$. Through this description we will actually describe two recursive subroutines relying on each other. Namely:
\begin{itemize}
\item A subroutine that given $\varphi$, two ordinals $\alpha, \beta \in \zG$ with $\alpha>0$ and notations $p_1, \dots, p_n$, decides whether $L_{\zeta\cdot\alpha + \beta} \models \varphi(\set{p_1}, \dots, \set{p_n})$.
\item A subroutine that given notations $a$ and $b$, decides whether $\set{a} \in \set{b}$.
\end{itemize}

The main difficulty will be the fact that we can't decide whether some set $a$ is a notation. Deciding whether $a$ is of the form $\langle \varphi, b_1, \dots, b_n, \alpha, \beta \rangle$ for some sets $b_1, \dots, b_n$ is fairly straightforward if we forget the condition on $\alpha$ and $\beta$ being in $L_{\zG}$. However deciding this, that is whether $\alpha, \beta \in \zG$, is clearly not doable with what we have established so far. And even if $\alpha$ and $\beta$ are known to be in $\zG$, taking a set that looks like a notation of lower rank, that is of rank $(\alpha', \beta') <_{lex} (\alpha, \beta)$, clearly does not guarantee with the lexicographical order that $\beta'$ is in $\zG$.

Still, despite this difficulty, let us first suppose that there is some machine $m_{\langle \rangle}$ that, given the code of some set $a \in L_{\SG}$, decides whether $a$ is a notation. We now describe $\mathcal{M}$ using $m_{\langle \rangle}$. It works as a triple imbricated induction on $(\alpha, \beta)$ for the outermost one, the rank of $\varphi$ for the next one and the rank of its parameters for the innermost one. We begin with the induction on $(\alpha, \beta)$.

\begin{itemize}[label={$\bullet$}, leftmargin=*]
\item Base case is $(\alpha, \beta) = (1, 0)$. In this case, by hypothesis, the rank of any $p_i$ is $(0, \beta_i)$ with $\beta_i \in \zG$. Hence, for any of those, $p_i = \langle \top, x_i, 0, \beta_i \rangle$ with $\set{p_i} = x_i$ and $\mathcal{M}$ can extract a code for $x_i$ given a code for $p_i$. And by Corollary \ref{cor:zeta_model}, using $s_\Sigma$ which is given as input, it can decide whether $L_{\zeta} \models \varphi(x_1, \ldots, x_n)$.\\

\item Then, when $(\alpha, \beta) >_{lex} (1,0)$, we proceed by induction on $\varphi$. 
\begin{itemize}[label={$\rhd$}]
\item Suppose that $\varphi$ is $\Delta_0$. Then there is not quantifiers and any atomic formula is of the form $\set{p_i} \in \set{p_j}$. We show by induction on $\rk(p_i)$ that $\mathcal{M}$ can decide such atomic formulas using the notations $p_i$ and $p_j$. 
\begin{itemize}
\item Suppose first that we can write $p_i = \langle \top, x_i, 0, \beta_i \rangle$. This means that $x_i \in L_{\zG}$. Then, either $p_j = \langle \top, x_j, 0, \beta_j \rangle$ and $\mathcal{M}$ can extract codes for $x_i$ and $x_j$ and use them to decide whether $\set{p_i} \in \set{p_j}$, as it simply means $x_i \in x_j$. Or $p_j = \langle \psi, c_1, \ldots, c_m, \alpha_j, \beta_j \rangle$ and $\set{p_i} \in \set{p_j}$ if and only if:
\begin{align*}
L_{\zG \cdot\alpha_j + \beta_j} \models \psi(x_i, \set{c_1}, \dots, \set{c_m})
\end{align*}
As $(\alpha_j, \beta_j) <_{lex} (\alpha, \beta)$, taking $p_i$ for a notation of $x_i$, $\mathcal{M}$ can inductively (by to the induction hypothesis on $(\alpha, \beta)$) decide this.

\item Now, if $p_i = \langle \psi_i, b_1, \ldots, b_m, \alpha_i, \beta_i \rangle$ and $p_j = \langle \psi_j, c_1, \ldots, c_k, \alpha_j, \beta_j \rangle$. 
By definition, $\set{p_i} \in \set{p_j}$ if and only if $\set{p_i}$ is in $L_{\zeta \cdot\alpha_j + \beta_j}$ and satisfies $\psi_j$, the formula that defines $\set{p_j}$ over $L_\zeta\cdot\alpha_j + \beta_j$. That is, if and only if: 
\begin{align*}
\begin{cases}
\set{p_i} \in L_{\zeta \cdot \alpha_j + \beta_j} \\
L_{\zeta \cdot\alpha_j + \beta_j} \models \psi_j(\set{p_i}, \set{c_1}, \dots, \set{c_k})
\end{cases}
\end{align*}
First $\set{p_i} \in L_{\zeta \cdot \alpha_j + \beta_j}$ if and only if there exists a notation $q_i$ such that $\rk(q_i) <_{lex} (\alpha_j, \beta_j)$ and with $\set{q_i} = \set{p_i}$. We can w.l.o.g.\ suppose that $\rk(p_i) \geqslant_{lex} (\alpha_j, \beta_j)$ as otherwise $p_i$ satisfies the requirement. And under this assumption, given some notation $q_i$ such that $\rk(q_i) <_{lex} (\alpha_j, \beta_j)$, $\set{q_i} = \set{p_i}$ if and only if for all notation $r_i$ such that $\rk(r_i) < \rk(p_i)$ we have
\begin{align}\label{form:equi_ri}
\set{r_i} \in \set{q_i} \iff \set{r_i} \in \set{p_i} 
\end{align}
Hence using $m_\Sigma$ (to enumerate codes for sets in $L_\Sigma$) and $m_{\langle \rangle}$ (to decide which real numbers encode a notation), $\mathcal{M}$ can enumerate the notations $q_i$ such that $\rk(q_i) <_{lex} (\alpha_j, \beta_j)$ and for each of those $q_i$ enumerate all the $r_i$ such that $\rk(r_i) <_{lex} \rk(p_i)$ and for each of those $r_i$ inductively decide (as $\rk(r_i) <_{lex} \rk(p_i)$) whether the equivalence (\ref{form:equi_ri}) holds; which in turn enables $\mathcal{M}$ to decide whether $\set{q_i} = \set{p_i}$ for all the $q_i$ it encounters; which finally lets it decide whether $\set{p_i} \in L_{\zeta \cdot \alpha_j + \beta_j}$.

As for the second part, deciding whether $L_{\zeta \cdot\alpha_j + \beta_j} \models \psi_j(\set{p_i}, \set{c_1}, \dots, \set{c_k})$ can inductively be done since $(\alpha_j, \beta_j) <_{lex} (\alpha, \beta)$.

\item Last case, if $p_i = \langle \psi, b_1, \ldots, b_m, \alpha_i, \beta_i \rangle$ and $p_j = \langle \top, y, 0, \beta_j \rangle$, then $\set{p_i} \in \set{p_j}$ if and only if there is a notation $q_i = \langle \top, x, 0, \beta_i' \rangle$ with $\set{q_i} = \set{p_i}$ and $x \in y$. And this is decidable with the same reasoning as in the previous case, still under the assumption that notations are decidable using $m_{\langle \rangle}$.

\end{itemize}
Hence, as $\mathcal{M}$ can decide all the atomic formulas in the $\Delta_0$ formula $\varphi$, it can decide $\varphi$.

\item Now, suppose that $\varphi$ is $\Sigma_n$. We write $\varphi = \exists x \, \psi(x, \set{p_1}, \dots, \set{p_n})$. By the results of Proposition \ref{prop:features_of_notation}, $x \in L_{\zeta \cdot \alpha + \beta}$ if and only if $x$ has a notation of rank strictly less than $(\alpha, \beta)$. So, $\mathcal{M}$ does the following: it uses again the fact that given the snapshot $s_\Sigma$ and the machine $m_{\langle \rangle}$ it can enumerate $L_\Sigma$ (and know when it has exhausted it). $\mathcal{M}$ enumerates $L_{\SG}$ and for each notation $a \in L_\Sigma$ such that $\rk(a) < (\alpha, \beta)$ it tests whether $L_{\zG \cdot \alpha + \beta} \models \psi(\set{a}, \set{p_1}, \dots, \set{p_n})$, which it can do by the inductive hypothesis of the induction on the rank of $\varphi$. And ${L_{\zeta\cdot\alpha + \beta} \models \varphi(\set{p_1}, \dots, \set{p_n})}$ if and only this is true for some notation $a$ with $\rk(a) < (\alpha, \beta)$. 

\item When $\varphi$ is $\Pi_n$, we write $\varphi = \forall x \psi$. As in the previous case, $\mathcal{M}$ enumerates the notation $a$ in $L_{\SG}$ and inductively decides whether $L_{\zG \cdot \alpha + \beta} \models \psi(\set{a}, \set{p_1}, \dots, \set{p_n})$. If it fails for some $a$, $\mathcal{M}$ outputs false. If not (and, using $s_\Sigma$, $\mathcal{M}$ knows when it went through all notations in $L_{\SG}$), it outputs true.
This conclude the description of $\mathcal{M}$ under the assumption that $m_{\langle \rangle}$ was given.
\end{itemize}
\end{itemize}

Now, we need to provide a description of $m_{\langle \rangle}$. As we cannot decide whether some set is a notation because it involves deciding whether an ordinal is less than $\zG$, we introduce the slightly more general concept of $\emph{quasi-notation}$. Quasi-notations are defined inductively, as notations, with the only difference that for the case $a = \langle \varphi, b_1, \dots, b_n, \alpha, \beta \rangle$ or for the base case $a = \langle \top, x, 0, \beta' \rangle$ we only require $\alpha, \beta$ and $\beta'$ to be in $\SG$ (instead of $\zG$). The main interest of quasi-notations is that it will be fairly easy to decide whether a set is a quasi-notation: in our case, the conditions, as stated in the proposition, regarding the inputs that are given to $\mathcal{M}$ implies that they are all eventually writable. Consequently, every set that appears in a computation using them is at least accidentally writable. In particular any ordinal appearing is in $\SG$. In such a context, deciding whether a set is a quasi-notation is easy enough given the inductive definition of quasi-notations. So, we can consider the machine $m^{q}_{\langle \rangle}$ that, given some set $a \in L_{\SG}$, decides whether $a$ is a quasi-notation. And with it, we can finish the description of $\mathcal{M}$: it simply uses $m^{q}_{\langle \rangle}$ as if it was the fictitious $m_{\langle \rangle}$. Obviously, for the moment, we don't have any guarantees regarding the behavior of $\mathcal{M}$ computing with $m^{q}_{\langle \rangle}$ instead of $m_{\langle \rangle}$. We will work our way toward this.

First, another important fact regarding quasi-notations is the following: if $a$ is a quasi-notation and $a$ is eventually writable, then $a$ is a notation. To see this, it is enough to observe that, by finite induction, all ordinals appearing in the inductive definition of $a$  will be eventually writable as well and so in $\zG$.
From there the idea is the following: if we manage to prove that the quasi-notations used in the computation of $\mathcal{M}$ (which uses $m^{q}_{\langle \rangle}$) are eventually writable, then it means that they are actually all notations! And so, despite $\mathcal{M}$ working with $m^{q}_{\langle \rangle}$ which simply recognizes quasi-notation, it would actually behave as if it was working with $m_{\langle \rangle}$, which recognizes notations.

So we consider the computation of $\mathcal{M}$ using $m^{q}_{\langle \rangle}$, with input $s_\Sigma$, a formula $\varphi$, two ordinals $\alpha, \beta \in \zG$ with $\alpha>0$ and notations $p_1, \dots, p_n$ of rank strictly less than $(\alpha, \beta)$, as stated above. As noted, all those inputs are eventually writable. Further, by well-orderdness of the lexicographical order on $\SG \times \SG$, the machine also terminates when working with quasi-notation.
So, suppose that $\mathcal{M}$ answers positively, that is, it thinks that $L_{\zeta\cdot\alpha + \beta} \models \varphi(\set{p_1}, \dots, \set{p_n})$ (but this may not be what it actually computed, working with quasi-notations). Suppose also that $\varphi$ is $\Sigma_n$: $\varphi(\set{p_1}, \dots, \set{p_n}) = \exists x \psi(x, \set{p_1}, \dots, \set{p_n})$. As $\mathcal{M}$ answered positively, this means that it found a quasi-notation $a_x$ of rank strictly less than $(\alpha, \beta)$ for which it thinks that ``$L_{\zG\cdot\alpha + \beta} \models \varphi(\set{a_x}, \dots, \set{p_n})$''. As noted earlier, working with the lexicographical order, the rank of $a_x$ being strictly less than $(\alpha, \beta)$ does not imply that it is in $\zG \times \zG$. But the inputs are all eventually writable. Consequently so is $a_x$, as it would be easy to ask $\mathcal{M}$ to write this $a_x$ on its output. This, combined with the previous fact, shows that $a_x$ is actually a notation of rank strictly less than $(\alpha, \beta)$. That is, $\set{a_x} \in L_{\zeta \cdot \alpha + \beta}$ and $\mathcal{M}$ was correct in taking $\set{a_x}$ as a potential witness for $\varphi$.
This gives the main ingredient to show by induction, again on $\varphi$, that $\mathcal{M}$, despite working with quasi-notations, yields the desired result.

\begin{itemize}
\item If $\varphi$ is $\Delta_0$: the only quasi-notations involved are the notations $p_1, \dots, p_n$. 
However when deciding whether $\set{p_i} \in \set{p_j}$,
the machine will use quasi-notations and may be wrong. That is, it may wrongfully think that $\set{p_i} \in \set{p_j}$ because it used a quasi-notation $q_i$ that isn't a notation.
But if it is the case, as the inputs involved are eventually writable, we can eventually write the first quasi-notation $q_i$ that $\mathcal{M}$ uses to decide that $\set{p_i} \in \set{p_j}$, thinking for example that \say{$\set{q_i} = \set{p_i}$}. But this implies that $q_i$ is a notation. And so that $\mathcal{M}$ was correct in picking this $q_i$. Then if it thinks that $\set{q_i} = \set{p_i}$ working with quasi-notations, this is \emph{a fortiori} true with notations.

And so, inductively, we could reproduce the induction scheme used to define the machine to show that $\mathcal{M}$ is correct for each of this decisions, because (a) when something holds for all quasi-notations, it \emph{a fortiori} does for all notations and (b) when something does not hold for a quasi-notation, we can arrange so that $\mathcal{M}$ eventually writes the first quasi-notation appearing in its computation for which the statement does not hold, hence showing that it is actually a notation (the important point in this previous assertion being that, by well-orderdness, at any step of the induction scheme there are only finitely many sets involved--namely the parameters that get created and carried along the inductive way--and that, by induction, all of them are eventually writable.) 

This shows that $\mathcal{M}$ is correct in every of its subchoices done to decide whether the formula $\varphi$ is true and consequently that it is correct when deciding whether $\varphi$ is true.

\item If $\varphi$ is $\Sigma_n$: we write $\varphi(\set{p_1}, \dots, \set{p_n}) = \exists x \psi(x, \set{p_1}, \dots, \set{p_n})$. Suppose that $\mathcal{M}$ answers positively. That is, it found a quasi-notation $a_x$ such that it answers positively with $\psi$, $a_x$ and $p_1, \dots, p_n$ as inputs. As seen, this implies that $a_x$ is eventually writable, so that it is actually a notation and that $\set{a_x} \in L_{\zG \cdot \alpha + \beta}$. By induction, this means that $L_{\zG\cdot\alpha + \beta} \models \psi(\set{a_x}, \set{p_1}, \dots, \set{p_n})$ and so that $L_{\zG\cdot\alpha + \beta} \models \varphi(\set{p_1}, \dots, \set{p_n})$. If it answered negatively, this is done as in the next case, replacing \say{positively} by \say{negatively}.

\item If $\varphi$ is $\Pi_n$, writing $\varphi = \forall x \psi$, the machine answering positively means that for all quasi-notation $a_x$ of rank strictly less than $(\alpha, \beta)$, it answered positively with inputs $\psi$, $a_x$ and $\set{p_1}, \dots, \set{p_n}$. In particular, this means that it answers positively with those inputs for all notation $a_x$ of rank strictly less than $(\alpha, \beta)$ (as notations are also quasi-notations). By induction, it means that for all notation $a_x$ of rank strictly less than $(\alpha, \beta)$, $L_{\zG\cdot\alpha + \beta} \models \psi(\set{a_x}, \set{p_1}, \dots, \set{p_n})$. And by Proposition \ref{prop:features_of_notation}, this implies that $L_{\zG\cdot\alpha + \beta} \models \forall x \, \psi(x, \set{p_1}, \dots, \set{p_n})$. Again, the negative case is done like the $\Sigma_n$ positive case, replacing \say{positively} by \say{negatively}.
\end{itemize}
\end{proof}

\begin{prop}
The ordinal $\zG$ is a gap of real numbers in the constructible universe of length $\zG^2$. That is: \begin{align*}
(L_{\zG^2} - L_{\zG}) \cap \mathcal{P}(\om) = \emptyset
\end{align*}
\end{prop}
\begin{proof}
Suppose that some new real number $x$ is defined at some stage $\gamma \in [\zG, \zG^2[$. That is, there is a formula $\varphi$ and parameters $p_1, \dots, p_n \in L_{\gamma}$ such that:
\begin{align*}
x = \left\{ n \in \om \mid L_\gamma \models \varphi(n, p_1, \dots, p_n) \right\}
\end{align*}
and $x \not\in L_\gamma$. As $\gamma < \zG^2$ it can be written $\zG \cdot \alpha + \beta$ with both $\alpha$ and $\beta$ in $L_{\zG}$. By Proposition~\ref{prop:features_of_notation}, all $p_i$ have a notation in $L_{\zG}$, that is an eventually writable notation. Further, for any $n \in \om$, a code for $\nott{n}$ is easily computable. So we can consider the following machine: it starts by eventually writing $\varphi$, $\alpha$, $\beta$, $\wb{p_1} \dots, \wb{p_n}$ and $s_\Sigma$ on some part of its working tape. Once they are written (that is with the usual dovetailling technique of Proposition \ref{prop:dovetailing_sim_operator}), for each $n \in \om$, it computes $\wb{n}$ and uses the parameters it eventually wrote to simulate the machine $\mathcal{M}$ described in Proposition \ref{prop:decide_with_notations}. With it, it decides (and save the result on another part of its working tape) whether $L_{\zG \cdot \alpha + \beta} \models \varphi(n, p_1, \dots, p_n)$ for all $n \in \om$. Once done for all $n$, it can eventually write $x$ on its output, contradicting the fact that $x$ was a new real number appearing at stage $\gamma+1 > \zG$.
\end{proof}

We can now finish the proof of the main theorem.

\begin{proof}[Proof of Theorem \ref{thm:sigma_tau_looping_condition}]
We use the fact that $\zG$ starts a big gap in the constructible universe to show that it is seen to be looping at stage $\zG \cdot \om$. It relies on Corollary \ref{cor:looping_condition_gamma}, in the same way as the proof of Proposition \ref{prop:gamma_machine_loops_countable_stage} does.

We show first that given some accidentally writable ordinal $\alpha$, it is possible for a machine to eventually decide (that is to eventually write $1$ if it is true and $0$ otherwise) whether $L_\alpha \prec L_{\SG}$. In Proposition \ref{prop:sigma_model} we showed that there is a real number $s_\Sigma$ and a machine $m_\Sigma$ such that, given $s_\Sigma$ and some $\varphi$ and $p$ as input, $m_\Sigma$ decides whether $L_{\SG} \models \varphi(p)$. Hence, consider the following machine $\mathcal{N}$: it eventually writes $s_\Sigma$, the snapshot of Proposition \ref{prop:sigma_model}, and for all $\varphi(p)$ with $p \in L_\alpha$ and such that $L_\alpha \models \varphi(p)$, it uses this snapshot (that is with the usual dovetailing construction presented in proof of Proposition \ref{prop:dovetailing_sim_operator}), to eventually decide (once $s_\Sigma$ has been eventually writable, $m_\Sigma$, as simulated in the computation we are describing, always halt!) whether $L_{\SG} \models \varphi(p)$. Once $\mathcal{N}$ iterated through all formulas $\varphi$ and parameters $p \in L_\alpha$, it eventually decides whether $L_\alpha \prec L_{\SG}$.

Now, using $\mathcal{N}$, it is possible to eventually write, if there is one, the first ordinal $\alpha$ such that $L_\alpha \prec L_{\SG}$. Indeed, as once $s_\Sigma$ has been eventually written, the simulation of $m_\Sigma$ inside some machine with $s_\Sigma$ as input always halts, we can design a machine that looks for $\alpha$ such that $L_\alpha \prec L_{\SG}$ and with the usual dovetailling technique, that restarts each time the machine that eventually writes $s_\Sigma$ changes its output. And there is one such $\alpha$ by Proposition \ref{prop:extension_zG_SG}. So there is $\alpha < \zG$ such that $L_\alpha$ is an end-elementary substructure of $L_{\SG}$, as otherwise this computation would eventually write $\zG$ itself. On the other side of $\zG$, if there was a bound on those $\alpha$'s, this bound would be eventually writable and greater than $\zG$. Hence such $\alpha$'s are cofinal in $\SG$. This yields, among many others, two ordinals $\alpha$ and $\alpha'$ such that:
\begin{align*}
    \begin{cases}
    \alpha < \zG < \alpha' \\
    L_{\alpha} \prec L_{\zG} \prec L_{\alpha'}\\
    \end{cases}   
\end{align*}

As the operator $\Gamma$ is suitable, it is defined by some formula $\varphi_{\Gamma}$ at the different levels of the constructible universe. So, in particular, this chain of e.e.e.\  means, looking at the computation of $\UG$, that stages $\alpha$, $\alpha'$ and $\zG$ all share the same snapshot. 
Hence, for any $\nu, \nu'$ in $\left\{\alpha, \alpha',\zG \right\}$ with $\nu < \nu'$ and by asymptoticity, the machine will repeat itself $\om$ times between stages $\nu'$ and $\nu' \cdot \om$. However, it does not yet imply that the machine is looping. More precisely, it does not escape this repetition of the segment $[\nu, \nu'[$ and is actually looping if and only if, by Corollary \ref{cor:looping_condition_gamma}, the snapshot occurring after this repetition is the same as a snapshot occuring between $\nu$ and $\nu'$. 
So we will show that the snapshot appearing at stage $\zG \cdot \om$ in the computation of $\UG$, that is after some final segment of the computation below $\zG$ is repeated $\om$ times, was actually part of said segment of computation.

To do this, we consider first $s_{\zG}$, the snapshot of the computation of $\UG$ at stage $\zG$. As $L_\alpha \prec L_{\zG}$, this snapshot also occurs at stage $\alpha$ and $s_{\zG}$ is definable (as a real) over $L_\alpha$. Hence, $s_{\zG} \in L_{\zG}$. We let $\alpha_0 \leqslant \alpha$ be the least stage at which this snapshot occurs.
Then, we let $s_{\zG \cdot \om}$ be the snapshot (again seen as a real) of $\UG$ occurring at stage $\zG \cdot \om$. At this stage, by asymptoticity, the segment of computation between stages $\alpha_0$ and $\zG$ has been repeated $\om$ times. Moreover, $s_{\zG \cdot \om}$ is definable over $L_{\zG \cdot \om}$ and so it is in $L_{\zG \cdot \om +1 }$. But the gap of real numbers starting at $\zG$ and of length $\zG^2$ spans over $\zG \cdot \om+1$. So $s_{\zG \cdot \om}$ is also in $L_{\zG}$.
Now we consider the following sentence, writing $S(\nu)$ for the snapshot of $\UG$ at some stage $\nu$:
\begin{align*}
\exists \nu, \nu' \, (\nu < \nu' \wedge S(\nu) = s_{\zG} \wedge S(\nu') = s_{\zG \cdot \om})
\end{align*}
It is naturally true in $L_{\alpha'}$ as it witnesses stages $\alpha_0$ and $\zG \cdot \om$ (As $\zG$, the ordinal $\alpha'$ is multiplicatively closed). Moreover, this sentence is actually expressible in the language of $L_{\zG}$ as both $s_{\zG}$ and $s_{\zG \cdot \om}$ are in $L_{\zG}$. As such, by elementarity, it can be reflected down to $L_{\zG}$. This implies that there is some stage $\nu' < \zG$ such that its snapshot is $s_{\zG \cdot \om}$ and such that it appears after an occurrence of the snapshot $s_{\zG}$. As $\alpha_0$ is the least stage whose snapshot is $s_{\zG}$, $\alpha_0 < \nu'$.
As wanted, the repeating segment $[\alpha_0, \zG[$ in the computation of $\UG$ yields at stage $\zG \cdot \om$ a snapshot that is part of it, namely that appeared at stage $\nu'$. Hence $\UG$ is seen to be looping at stage $\zG \cdot \om < \SG < \TG$, which is a contradiction.
\end{proof}

The two other structural equalities are now corollaries of the main theorem.  

\begin{cor}\label{cor:zeta_eta}
Let $\Gamma$ be an $n$-symbol suitable, looping stable and simulational operator. Then $\zeta_\Gamma = \eta_\Gamma$.
\end{cor}
\begin{proof}
By Proposition \ref{prop:gamma_emulates_limsup}, $\Gamma$ emulates, w.l.o.g., the operator $\Gamma_{\sup}$. So we can apply Proposition \ref{prop:write_inf_clock} which yields $\zeta_\Gamma \leqslant \eta_\Gamma$. Now, suppose that $\zeta_\Gamma < \eta_\Gamma$. This means that there is some machine $m$ and ordinal $\alpha$ such that $m$ converges at stage $\alpha$ with $\zG < \alpha < \eG$. As $\SG = \TG$ by Theorem \ref{thm:sigma_tau_looping_condition}, $\alpha < \SG$. Then we can consider the following machine: using $\UG$, it enumerates ordinals below $\SG$. For each ordinal $\nu$ produced by $\UG$, the main machine copies it to its output and then simulates $\nu$ steps of a fresh simulation of $m$ and, after those $\nu$ steps, goes on with the simulation of $m$ until its output changes. If it does, the computation goes on with the simulation of $\UG$ and the enumeration of accidentally writable ordinals. If it doesn't, it actually eventually writes the ordinal $\nu$. And the output of the simulation of $m$ does not change if and only if $\nu \geqslant \alpha$. Hence this computation eventually writes the first ordinal greater than $\alpha$ to appear, which contradicts the fact that $\zG < \alpha$.
\end{proof}

\begin{cor}\label{cor:lambda_gamma}
Let $\Gamma$ be an $n$-symbol suitable, looping stable and simulational operator. Then $\lambda_\Gamma = \gamma_\Gamma$.
\end{cor}
\begin{proof}
Again by Proposition \ref{prop:gamma_emulates_limsup} and Proposition \ref{prop:write_inf_clock}, $\lambda_\Gamma \leqslant \gamma_\Gamma$. Then, if a machine halts after stage $\lambda_\Gamma$, it must still do so before stage $\SG$ in virtue of Theorem \ref{thm:sigma_tau_looping_condition}. But then, it is easy to simulate this machine along the accidentally writable ordinals to look for and write the accidentally writable ordinal stage at which this machine halts, which proves that every halting stage is actually also writable.
\end{proof}

\begin{cor}
Let $\Gamma$ be an $n$-symbol suitable, looping stable and simulational operator. A set $a$ is
in $L_{\lG}$ (resp. in $L_{\zG}$ and in $L_{\SG}$) if and only if it is writable (resp. eventually writable and accidentally writable) by a $\Gamma$-machine.
\end{cor}
\begin{proof}
The first implication is simply Proposition \ref{prop:set_to_writable}. We show the converse implication for accidentally writable real numbers, the two other cases are similar.

Suppose that the set $a$ is accidentally writable. That is a code for $a$, which we write $\hat{a}$, appears in the computation of some $\Gamma$-machine. As, by Theorem \ref{thm:sigma_tau_looping_condition}, $\Sigma_\Gamma = \Tau_\Gamma$, this codes appears in this computation before stage $\Sigma_\Gamma$ and so is itself (as a real) in $L_{\SG}$. 

Then, as with the operator $\Gamma_{sup}$, $\zG$ is an admissible ordinal and so $\SG$ is a limit of admissible ordinals--as otherwise the last admissible ordinal below $\SG$ would be eventually writable. So there is $\alpha$ admissible and smaller than $\SG$ such that $\hat{a} \in L_\alpha$.
As a code for $a$, $\hat{a}$ describes a transitive relation $E$ on $\om$ such that $(\om, E) \simeq (TC(\left\{ a \right\}), \in)$. The isomorphism described by it yields an application $f$ on $\om$ inductively defined as, using $E$ in infix notation:
\begin{align*}
f(i) = b \longleftrightarrow \forall c \in b \, \exists jEi \,  (f(j) = c) \wedge \forall j E i \, \exists c \in b \, (f(j) = c)
\end{align*}
As $L_\alpha$ is admissible, this function is definable in it by $\Sigma$-recursion and for some $i_a$, $f(i_a) = a$. Also, for some $E$-least $i_0$, $i_0$, $f(i_0) = \emptyset \in L_{\alpha}$. Hence, by external induction resting on the use of $\Sigma$-collection, for all $i$, $f(i) \in L_{\alpha}$. In particular, $a \in L_{\alpha} \subset L_{\SG}$, as wanted.
\end{proof}

\begin{cor}[$\lambda$-$\zeta$-$\Sigma$ theorem]\label{cor:lambda-zeta-Sigma}
Let $\Gamma$ be an $n$-symbol suitable, looping stable and simulational operator. Then the following chain of elementary end-extension holds.
\begin{align*}
L_{\lG} \prec_{\Sigma_1} L_{\zG} \prec_{\Sigma_2} L_{\SG}
\end{align*}
\end{cor}
\begin{proof}
This theorem was first established in \cite{welch_main} for the ITTM and, using Theorem \ref{thm:sigma_tau_looping_condition}, the proof is the same in the general case. We give one direction of the $\Sigma_2$ end-extension.

Suppose that some $\Sigma_2$ formula $\varphi(p) = \exists x \forall y \, \psi(x, y, p)$ is true in $L_{\zG}$, for some $p \in L_{\zG}$. Then, since $p \in L_{\zG}$, by previous corollary, it is eventually writable. Using it, a $\Gamma$-machine can look for $\wt{x}$, a witness of $\varphi(p)$, and it can eventually write it as well. Then, using now this $\wt{x}$ (that is using the previous machine which eventually writes $x$), another machine can eventually write, if it exists, the first $y$ it finds such that $L_{\SG} \models \neg \psi(\wt{x}, y, p)$. But such a $y$ can't be found in $L_{\zG}$, since $\varphi(p)$ is true in it. And it can't be found in $L_{\SG} - L_{\zG}$ either, since it would be eventually writable. So, in the end, $\varphi(p)$ is also true in $L_{\SG}$.
\end{proof}

\begin{remark}\label{rmk:when_does_gamma_machine_loops}
A question remains: for an operator $\Gamma$ that satisfies the condition of Theorem \ref{thm:sigma_tau_looping_condition}, is a $\Gamma$-machine that does not halt seen to be looping at stage $\SG$? In virtue of this theorem, we know that no new snapshot appears after stage $\SG$. It is however not yet enough to apply Corollary \ref{cor:looping_condition_gamma}. As it stands, using this corollary, it can be shown that, for such an operator, $\Gamma$-machines will be seen to be looping before stage $\SG \cdot \om^{\om}$.
Still, it may be possible to show that, in fact and as in the classical case, those machines are seen to be looping as soon as they reach stage $\SG$.
\end{remark}

\section{A counter-example without the looping condition}\label{sec:counter-example}

By Theorem \ref{thm:sigma_tau_looping_condition}, for an operator satisfying the looping condition (as well as being suitable and simulational and enhancing the $\limsup$ machine), we have $\SG = \TG$. 
In this section, to show that the looping condition is a necessary condition, we construct a suitable and simulational operator that does not satisfy the looping condition and for which $\SG < \TG$. The definition of this operator, the \emph{tick operator}, is done in Definition \ref{def:tick_operator} and is fairly straightforward. However, to prove that it indeed forms a counter-example, one convincing way to do it is to introduce another operator, the \emph{escaping operator} whose behavior is easier to study and to show that both those operators behave in a similar way. 

The intuition behind this counter-example, mentioned in Remark \ref{rmk:looping_stability}, is fairly simple as well: the machines ruled by the tick operator will be repeating for most of their computation and they will only exit those repetitions every tick, that is every $\tau$ steps for a very big $\tau$. After they exit their repetition, they compute a bit and \say{quickly} start to repeat again until the next tick. Hence with a well chosen $\tau$, we will be able to obtain great length of computation, that is greater than $\tau$, with only few different real numbers written in the computations, and so with small accidentally writable ordinals, that is smaller than $\tau$.

\begin{theorem}\label{thm:contre_exemple_machine}
There exists a suitable and simulational operator $\Gamma$ that enhances the operator $\Gamma_{sup}$ and such that $\SG < \TG$. 
\end{theorem}

Before defining the escaping operator that will help us prove this result for the tick operator we need some preliminary results.

\begin{definition}
For $n \in \om$, the $n$-symbol $\limsup$ operator is the cell-by-cell operator $\Gamma_{\sup}^n$ defined by the single cell operator $\gamma_{\sup}^n$ itself defined as follows. For $h \in {}^{<On}n$,
\begin{align*}
\gamma_{\sup}^n(h) = k \text{ when } k \in n \text{ is the greatest integer cofinal in } h
\end{align*}
with the convention that if $h$ is not a limit word, its last letter is considered cofinal in it.
\end{definition}

\begin{prop}[Looping condition for $\Gamma_{\sup}^n$-machine]\label{prop:looping_condition_n-symbols_limsup}
A $\Gamma_{\sup}^n$-machine is looping if and only if there are two limit stages $\mu < \nu$ sharing the same snapshot and such that for each cell $i$, writing $s_\nu$ for the snapshot at stage $\nu$ and $s_\nu[i]$ for the value of cell $i$ in $s_\nu$:
\begin{align*}
\max_{\nu' \in \left[\mu, \nu \right[}(C_i(\nu')) = s_\nu[i]
\end{align*}
\end{prop}
\begin{proof}
The condition $\max_{\nu' \in \left[\mu, \nu \right[}(C_i(\nu')) = s_\nu[i]$ simply comes from the fact that once the segment of computation $[\mu, \nu[$ is repeating, the $\limsup$ of the value of some cell $i$ after it repeated some limit ordinal amount of time is equal to the maximum value of the cell $i$ in this segment of computation $[\mu, \nu[$. This ensures that the snapshot $s_\nu$ also appears after the segment $[\mu, \nu[$ has been repeated a limit ordinal amount of time.
\end{proof}

\begin{prop}
For $n \in \om$, the $n$-symbol $\limsup$ operator $\Gamma_{\sup}^n$ is a suitable and simulational operator.  
\end{prop}
\begin{proof}
This is the same proof as done for the usual $2$-symbol $\limsup$ operator.
\end{proof}

\begin{prop}\label{prop:n-symbols_limsup}
For $n>1$, the $n$-symbol $\limsup$ operator $\Gamma_{\sup}^n$ is as powerful as the usual $\limsup$ operator $\Gamma_{\sup}$. That is $\SGsupn = \SGsup$ and $\TGsupn = \TGsup$
\end{prop}
\begin{proof}
For $n>1$, $\Gamma_{\sup}^n$ enhances $\Gamma_{\sup}$, so we naturally have that $\SGsupn \geqslant \SGsup$ and $\TGsupn \geqslant \TGsup$. In the other direction, we show that a $\Gamma_{\sup}$-machine can simulate a $\Gamma_{\sup}^n$-machine.

For finite machines, an $n$-symbol machine can be easily simulated by a $2$-symbol machine by representing a cell of the $n$-symbol machine by $n-1$ cells of the $2$-symbol machines. We choose the following encoding: the $n$-symbol cell that reads $k \in n$ is represented by $n-1$ cells (namely, $2$-symbol cells) that read, once aggregated, $1^k 0^{n-k-1}$. Observe that at a limit stage, those are equivalent:
\begin{itemize}[label=--]
\item The $n$-symbol operator yields $k$ for cell $i$.
\item $k$ is the greatest cofinal value of $i$ in the $n$-symbol cell up to this limit stage.
\item $1^k 0^{n-k-1}$ is the aggregated value of the $n-1$ $2$-symbol cell simulating $i$ with the longest $1$-prefix and which is cofinal in this limit stage.
\item the $\limsup$ operator yields $1^k 0^{n-k-1}$ for the cells simulating $i$.
\end{itemize}
Hence, at any limit stage, the $n$-symbol $\limsup$ operator is correctly simulated by the usual $2$-symbol $\limsup$ operator and with it the simulation is faithfully carried out through all the ordinals.
\end{proof}

\begin{definition}
For $\Gamma$ an asymptotic operator and $H$ a limit segment of history, we say that $H$ is a \emph{looping pattern w.r.t.\ $\Gamma$} or a \emph{$\Gamma$-looping pattern} when a machine with a history of the form $H_0 H^\om$ is looping and, more precisely, looping over the segment of history $H$.
\end{definition}

In the previous definition, we ask $\Gamma$ to be asymptotic in order to ensure that different initial segments $H_0$ won't lead to different behaviors after $H$ repeated $\om$ times. So it is equivalent to ask that one machine with this history loops or to ask that any machine with this history loops.

\begin{example}\label{ex:looping_pattern}
Take $\Gamma^n_{\sup}$, the $n$-symbol $\limsup$ operator. Then, with Proposition \ref{prop:looping_condition_n-symbols_limsup} in mind, $H$ is a looping pattern w.r.t.\ $\Gamma^n_{\sup}$ if and only if for all cell $i$:
\begin{align*}
H[0][i] = \max_{\nu < |H|}(H[\nu][i])
\end{align*}
Hence, if for some $\Gamma^n_{sup}$-machine the history reads at some point $H_0 H^\om$, we know that the machine is looping and that the segment of history $H$ will repeat indefinitely. 
\end{example}

\begin{definition}[the $3$-symbol escaping machine]\label{def:jump_operator}
The \emph{$3$-symbol escaping operator} $\Gamma_{esc}$ is a $3$-symbol $\limsup$ operator with the only difference that its behavior changes after the machine has been repeating itself for some limit ordinal amount of times (or when such repeating stages are cofinal). It is called a "escaping" operator as this change of behavior can be seen as a Turing jump escaping the loop.

It can be defined using the $\limsup$ operators $\Gamma_{102}$ and $\Gamma_{210}$. Those are defined using cell-by-cell $\limsup$ operators $\gamma_{102}$ and $\gamma_{210}$. Both are usual $3$-symbol $\limsup$ operators working on the alphabet $3$ with the only difference being the order of priority of the letters in $3$. That is for a limit cell history $h \in {}^{<On}3$,
\begin{equation*}
  \gamma_{102}(h) = 
    \begin{cases}
      1 & \text{if } 1\text{'s are cofinal in } h  \\
      0 & \text{if } \gamma_{102}(h) \neq 1 \text{ and } 0\text{'s are cofinal in } h \\
      2 & \text{if } \gamma_{102}(h) \neq 1 \text{ and } \gamma_{102}(h) \neq 0
    \end{cases}       
\end{equation*}
and
\begin{equation*}
\gamma_{210}(h) =
    \begin{cases}
      2 & \text{if } 2\text{'s are cofinal in } h\\
      1 & \text{if } \gamma_{210}(h) \neq 2 \text{ and } 1\text{'s are cofinal in } h  \\
      0 & \text{if } \gamma_{210}(h) \neq 2 \text{ and } \gamma_{210}(h) \neq 1
    \end{cases}       
\end{equation*}
From there, $\Gamma_{esc}$ is defined as follow. For a limit history $H \in {}^{<On}({}^\om 3)$:
\begin{equation*}
\Gamma_{esc}(H) =
    \begin{cases}
      \Gamma_{210}(H) & \text{ if } H = H_0 \cdot (H_1)^\om \text{ for a } \Gamma_{102}\text{-looping pattern } H_1 \\
      \Gamma_{210}(H) & \text{ if } H \text{ is a limit of terms of the form } H_0 \cdot (H_1)^\om \\ 
      & \text{ with the } H_1 \text{'s being } \Gamma_{102}\text{-looping patterns} \\
      \Gamma_{102}(H) & \text{ else}\\
    \end{cases}       
\end{equation*}   

\end{definition}

\begin{prop}
The escaping operator $\Gamma_{esc}$ is a $\Delta_4$-suitable operator. It is moreover stable, asymptotical and contraction-proof but not cell-by-cell.
\end{prop}
We give this proof of this first fact for the sake of completeness, but it is actually enough to be convinced that $\Gamma$ is a $\Sigma_n$-suitable operator for any $n$.
\begin{proof}
As it is built from two $\limsup$ operators, the escaping operator is easily seen to be stable and contraction-proof. It is also asymptotical as when the history $H$ can be written $H = H_0 (H_1)^\om$ (with $H_0$ possibly empty), any final segment of $H$ can be written in the same way (possibly with a different $H_0$). Same goes when $H$ is a limit of such terms. And it is not cell-by-cell as for a given cell $i$, the operator needs to consider the whole machine history $H$ rather than only the cell history $h_i$.

As for it being $\Delta_4$-suitable, the definition provides the skeleton that we can use to define a formula $\varphi_{esc}$ which suits Definition \ref{def:suitable_operator}. The details rest on a recursive definition, as in the proof of Proposition \ref{prop:lim_sup_suitable}. A detailed proof with the definition of $\varphi_{esc}$ can be found in \cite[Proposition 5.5.9]{these}.

\end{proof}

\begin{prop}\label{prop:looping_condition_gamma_loop}
A $\Gamma_{esc}$-machine is looping if and only if there are two limit stages $\mu < \nu$, both ruled by the operator $\Gamma_{210}$, that is both following the repetition of a $\Gamma_{102}$-looping pattern or a limit of thereof, sharing the same snapshot and such that writing $s_\nu$ for the snapshot at stage $\nu$ and writing $\max^{210}$ for the $\max$ operator with the order $2 \succ 1 \succ 0$, we have:
\begin{align}\label{eq:max_loop_condition}
\max^{210}_{\nu' \in \left[\mu, \nu \right[}(C_i(\nu')) = s_\nu[i]
\end{align}
In this case, the machine is \emph{seen to be looping} at stage $\nu$.
\end{prop}

\begin{proof}
Suppose that a $\Gamma_{esc}$-machine is looping and let $H$ be the shortest repeating pattern and $\alpha$ be the stage at which it starts looping. Let $H_0$ be the history up to stage $\alpha$. Then the history reads: $H_0 H H H \ldots$. In particular, when $H$ has repeated $\om$ times, it reads $H_0 H^\om$. As it is looping, $\Gamma_{esc}(H_0 H^\om) = H[0]$. Let $\mu$ be the stage below which the history $H_0 H^\om$ spans. We show that stage $\mu$ is ruled by the operator $\Gamma_{210}$, that is that $H^\om$ is a $\Gamma_{102}$-looping pattern or limit of such looping patterns.
Suppose that $H^\om$ isn't.
Then it is easy to see that for any limit $\alpha$, $H^\alpha$ is neither a $\Gamma_{102}$-looping pattern nor a limit of looping patterns. Bu then it means that after this point, $\Gamma_{esc}$ always acts as the operator $\Gamma_{102}$. But then $H^\om$ satisfies Proposition \ref{prop:looping_condition_n-symbols_limsup} and $H^\om$ is a $\Gamma_{102}$-looping pattern, which contradict our first assumption.
Hence, and by definition of the escaping operator, $\Gamma_{esc}(H_0 H^\om) = \Gamma_{210}(H_0 H^\om)$, i.e. the stage $\mu$ below which the history$H_0 H^\om$ spans is ruled by the $\Gamma_{210}$-operator. Then, let $\nu$ be the stage corresponding to the history $H_0 H^\om H^\om$. So the segment of history delimited by $\mu$ and $\nu$ is $H^\om$ itself. With the same reasoning, stage $\nu$ is also ruled by the $\Gamma_{210}$-operator and $\Gamma_{esc}(H_0 H^\om H^\om) = H[0] = \Gamma_{esc}(H_0 H^\om)$, which means that the snapshot at stage $\mu$ and $\nu$ match. Eventually, the fact that the machine is repeating implies that $\Gamma_{esc}((H^\om)^\om) = H[0]$. But also $\Gamma_{esc}((H^\om)^\om) = \Gamma_{210}((H^\om)^\om) $ and this ensures that condition (\ref{eq:max_loop_condition}) on the repeating segment (here $H^\om$), is satisfied as well.

Conversely, suppose that some $\mu$ and $\nu$ satisfying the hypothesis are given. Let $H_0$ be the segment of history that spans between stages $0$ and $\mu$ and $H$ the segment that spans between stages $\mu$ and $\nu$. As the snapshots at stages $\mu$ and $\nu$ match and $\Gamma_{esc}$ is asymptotical, the history reads $H$ again after $H_0 H$ appeared. And this goes on for at least $\om$ repetitions of $H$. As $\nu$ is ruled by the operator $\Gamma_{210}$, so are all stages below which $H_0 H^k$ spans for a natural number $k$. Hence, as a limit of stage limit of stages ruled by $\Gamma_{210}$, the stage below which spans $H_0 H^\om$ is also ruled by $\Gamma_{210}$. Hence $\Gamma_{esc}(H_0 H^\om) = \Gamma_{210}(H_0 H^\om)$. And $\Gamma_{210}(H_0 H^\om) = \Gamma_{210}(H_0 H)$ by hypothesis (\ref{eq:max_loop_condition}). And $\Gamma_{210}(H_0 H) = \Gamma_{210}(H_0) = H[0]$ by the fact that $\mu$ and $\nu$ share the same snapshot and are both ruled by $\Gamma_{210}$. Hence the stage corresponding to $H_0 H^\om$ also shares the same snapshot and is ruled by $\Gamma_{210}$. Further, carrying this reasoning on by transfinite induction, shows that the machine is looping. 

\end{proof}

\begin{prop}\label{prop:gamma_loop_loops_at_delta}
Let $m$ be a $\Gamma_{esc}$-machine that does not halt and let $(\alpha, \beta, \delta)$ the lexicographicaly least triple of additively closed ordinals such that $\alpha < \beta < \delta$ and
\begin{align*}
L_{\alpha} \prec_{\Sigma_4} L_{\beta} \prec_{\Sigma_4} L_{\delta}
\end{align*} 
Then $m$ is seen to be looping at the latest at stage $\delta \cdot \om$.
\end{prop}
\begin{proof}
By the e.e.e.\ hypothesis, the snapshots at stage $\alpha$, $\beta$ and $\delta$ match. Hence, by asymptoticty, the segment of history $[\alpha, \beta[$ repeats $\om$ times up to stage $\beta \cdot \om$. If at stage $\beta$ the operator $\Gamma_{esc}$ acts as $\Gamma_{210}$ then it does so at any stage $\beta \cdot k$ and also at stage $\beta \cdot \om$, as a limit of stages ruled by $\Gamma_{210}$. Else, suppose that at stage $\beta$ the operator $\Gamma_{esc}$ acts as $\Gamma_{102}$. In this case, we can show that $H_\beta$, the segment of history that spans between stages $\alpha$ and $\beta$ is a $\Gamma_{102}$-looping pattern, which contradicts this supposition. This comes from the fact that by elementarity, for any cell $i$, any symbol appearing in $i$ between stages $\alpha$ and $\beta$ also appears in $i$ cofinally in stage $\alpha$, and so that the value of cell $i$ at stages $\alpha$ and $\beta$ is at least greater (w.r.t.\ $\preceq_{210}$) that the maximum of the values appearing in cell $i$ between those stages. In definitive, $\beta \cdot \om$ must be ruled by the operator $\Gamma_{210}$

We write $s_{\beta \cdot \om}$ for the snapshot that appears at this stage. It may be different from $s_{\beta}$. However, $s_\delta = s_\beta$ and, again by asymptoticity, the segment of history $[\beta, \delta[$ starts repeating from stage $\delta$ and at least up to stage $\delta \cdot \om$. This is depicted in Figure \ref{fig:gamma_loop_machine_looping}. With the same justification, stage $\delta \cdot \om$ is ruled by the operator $\Gamma_{210}$.

We claim that stages $\beta \cdot \om$ and $\delta \cdot \om$ satisfy the looping condition of Proposition \ref{prop:looping_condition_gamma_loop}. First, as shown, at both those stages, operator $\Gamma_{esc}$ acts like $\Gamma_{210}$. Then, writing 
$H_\delta$ for the history that spans between stages $\beta$ and $\delta$, observe that $H_\beta \sqsubset H_\delta$. Hence, the repetition of $H_\delta$ encloses that of $H_\beta$. This implies that for any cell $i$:
\begin{align*}
\Gamma(H_\beta ^ \om)[i] \preceq_{210} \Gamma(H_\delta ^ \om)[i]
\end{align*}
The inequality may be strict for some cell $i$ if and only if a symbol greater (w.r.t.\ $\preceq_{210}$) than $\Gamma(H_\beta \cdot \om)[i]$ appears between stages $\beta \cdot \om$ and $\delta$ in the cell $i$. We show that this is not possible. To do this, we show that:
\begin{align*}
\max^{210}_{\nu \in \left[\alpha, \beta \right[}(C_i(\nu)) \succeq_{210} \max^{210}_{\nu \in \left[\beta \cdot \om, \delta \right[}(C_i(\nu))
\end{align*}
Suppose that this is not the case. This means that there is some symbol $s \in 3$ appearing in cell $i$ between stages $\beta \cdot \om$ and $\delta$ and greater than any symbol appearing in $i$ between stages $\alpha$ and $\beta$. The second half of the affirmation translates to:
\begin{align*}
L_\beta \models \forall \nu > \alpha \, C_i(\nu) \prec_{210} s
\end{align*}
But reflecting this affirmation up to $L_\delta$ contradicts the fact that $s$ appeared between stages $\beta \cdot \om$ and $\delta$. This shows that stages $\beta \cdot \om$ and $\delta \cdot \om$ share the same snapshot and moreover that the condition (\ref{eq:max_loop_condition}) of Proposition \ref{prop:looping_condition_gamma_loop} holds, which ends the proof.

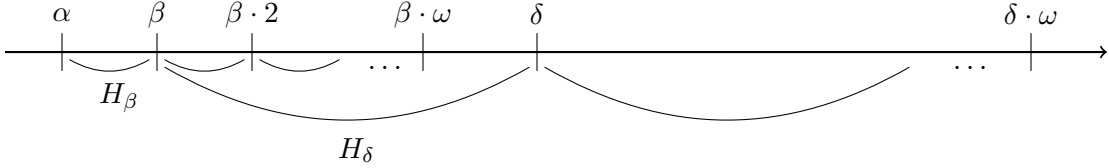
\begin{figure}[h]
\begin{center}
\hspace*{-.475in}
\begin{tikzpicture}
\draw[thick, ->]  (-0.75,0) -- (13.75,0);

\draw (0, 0.25) -- (0, -0.25);
\node at (0,0.5) {$\alpha$};

\node at (0.75, -0.6) {$H_\beta$};

\draw (1.25, 0.25) -- (1.25, -0.25);
\node at (1.25,0.5) {$\beta$};

\draw (2.5, 0.25) -- (2.5, -0.25);
\node at (2.5,0.5) {$\beta \cdot 2$};

\draw[bend right]  (0+0.1, -0.1) to (1.25-0.1, -0.1);
\draw[bend right]  (1.25+0.1, -0.1) to (2.5-0.1, -0.1);
\draw[bend right]  (2.5+0.1, -0.1) to (3.75-0.1, -0.1);

\node at (4.25,-0.2) {$\ldots$};

\draw (4.75, 0.25) -- (4.75, -0.25);
\node at (4.75,0.5) {$\beta \cdot \om$};

\node at (3.875, -1.3) {$H_\delta$};

\draw (6.25, 0.25) -- (6.25, -0.25);
\node at (6.25,0.5) {$\delta$};
\draw[bend right]  (1.25+0.1, -0.2) to (6.25-0.1, -0.2);
\draw[bend right]  (6.25+0.1, -0.2) to (11.25-0.1, -0.2);

\node at (11.95,-0.2) {$\ldots$};

\draw (12.75, 0.25) -- (12.75, -0.25);
\node at (12.75,0.5) {$\delta \cdot \om$};

\end{tikzpicture}
\caption{Computation of a $\Gamma_{esc}$-machine between stages $\alpha$ and $\delta \cdot \om$.}
\label{fig:gamma_loop_machine_looping}
\end{center}
\end{figure}
\end{proof}

We now define the $\tau$-ticking operator which will be, in definitive, the operator that interests us to prove Theorem \ref{thm:contre_exemple_machine}. However, the operator $\Gamma_{esc}$ will help us define the right $\tau$-tick operator (which depends of an ordinal $\tau$) as well as help us study its behavior. It is worth noting that the idea behind the ticking operator is similar to Habic's cardinal recognizing ITTMs, which we mentioned earlier as a model that does not satisfy the condition of looping stability.

\begin{definition}[the $3$-symbol tick machine]\label{def:tick_operator}
For $\tau$ a limit ordinal, the \emph{$\tau$-ticking operator} $\Gamma_{\tau}$ is a $3$-symbol $\limsup$ operator with the only particularity that its behavior is altered at limit stages that are multiple of $\tau$. It is called a "ticking" operator as this change of behavior can be seen as a regular tick. $\Gamma_{\tau}$-machines will naturally be able to harvest this tick and, at least, to know when their current stage is a multiple of $\tau$.

$\Gamma_{\tau}$ can be defined as a cell-by-cell operator $\gamma_\tau$ defined using $\gamma_{102}$ and $\gamma_{210}$; themselves defined as in Definition \ref{def:jump_operator}. For a limit history $h \in {}^{<On}3$:

\begin{equation*}
\gamma_\tau(h) =
    \begin{cases}
      \gamma_{210}(h) & \text{ if } |h| = \tau \cdot \alpha  \text{ for some ordinal } \alpha\\
      \gamma_{102}(h) & \text{ else}\\
    \end{cases}       
\end{equation*}

That is before stage $\tau$ and more generally at any stage that is not a multiple of $\tau$, a $\Gamma_\tau$-machine behaves like a $\Gamma_{\sup}^3$ machine with the order $1 \succ 0 \succ 2$ on its alphabet. And at stages $\tau \cdot \alpha$, a $\Gamma_\tau$-machine behaves like a $\Gamma_{\sup}^3$ machine with the order $2 \succ 1 \succ 0$ on its alphabet.
\end{definition}

\begin{definition}[uniformly characterizable ordinals]\label{def:characterizable}
An ordinal $\alpha$ is \emph{uniformly characterizable} with respect to the constructible hierarchy if there is a sentence $\psi_\alpha$ such that for all real number $y$:
\begin{align*}
\begin{cases}
L_\alpha[y] \models \psi_\alpha \\
\forall\beta < \alpha \, L_\beta[y] \not\models \psi_a
\end{cases}
\end{align*}
It is $\Sigma_n$-\emph{uniformly characterizable} when $\psi_\alpha$ is $\Sigma_n$.
\end{definition}

\begin{prop}\label{prop:big_charac_ordinal}
There exists an additively closed $\Pi_2$-uniformly characterizable ordinal $\chi$ strictly greater than $\Sigma_{esc}$ and $\Tau_{esc}$, which are respectively the supremum of the ordinals accidentally written by a $\Gamma_{esc}$-machine and the supremum of the ordinals a.c.\ by a $\Gamma_{esc}$-machine.
\end{prop}

\begin{proof}
As $\Gamma_{esc}$ is not cell-by-cell, it is not simulational and we can't apply Proposition \ref{prop:write_inf_clock}. However, we know that every accidentally writable real number is accidentally writable before stage $\Tau_{esc}$ and that $\Tau_{esc} < \delta$ with $\delta$ defined in Proposition \ref{prop:gamma_loop_loops_at_delta}. So, for any $\Gamma_{esc}$-a.w.\ ordinal $\alpha$, a code for $\alpha$ is in $L_\delta$ and $\alpha$ is in $L_{\delta^+}$ where $\delta^+$ is the least admissible ordinal greater than $\delta$. Hence, we also know that $\Sigma_{esc} \leqslant \delta^+$.

Now let $\chi$ be the least ordinal that witnesses the existence of $\alpha$, $\beta$, $\delta$ and $\delta^+$ such that the following is true in $L_\chi$:
\begin{align*}
\begin{cases}
L_{\alpha} \prec_{\Sigma_4} L_{\beta} \prec_{\Sigma_4} L_{\delta} \\
\delta^+ \geqslant \delta \\
 L_{\delta^+} \models \mathrm{KP}  \\
\forall \mu, \nu \, \exists \xi \, (\xi = \alpha + \beta)
\end{cases}
\end{align*}
This description yields a uniform (as it only quantifies over ordinals or set of the form $L_\alpha$) characterization of $\chi$ strictly greater than $\Sigma_{esc}$ and $\Tau_{esc}$. (And note that because of Induction, $\mathrm{KP}$ is not finitely axiomatizable but this is not an issue when dealing with structures that already satisfy Foundation.)
\end{proof}

\begin{prop}\label{prop:multiplicity_sc_ordinals}
Given a $\Sigma_n$-uniformly characterizable ordinal $\alpha$, there exists a $\Sigma_m$ sentence $M_\alpha$ with $m = \max(n, 3)$ and such that for all ordinal $\nu$:
\begin{align*}
\exists \mu > 0 \, (\nu = \alpha \cdot \mu) \longleftrightarrow L_\nu[y] \models M_\alpha
\end{align*}
\end{prop}
\begin{proof}
Let $\psi_\alpha$ as defined in Definition \ref{def:characterizable} and consider the following formula.
\begin{align*}
M_\alpha = (\psi_\alpha \wedge \forall \beta \, L_\beta \not\models \psi_\alpha)  \vee  \exists \alpha_0 \, & \left[ L_{\alpha_0} \models \psi_\alpha \wedge \forall \beta < \alpha_0 \ L_\beta \not\models \psi_\alpha \right.  \\ 
&  \left. \wedge \forall \mu \, (\exists \gamma \, (\gamma = \alpha_0 \cdot \mu) \implies \forall \delta < \alpha_0 \ \exists \gamma \, (\gamma = \alpha_0 \cdot \mu + \delta)) \right]
\end{align*}
Then $L_\nu \models M_\alpha$ if and only if: either $\nu = \alpha$, in which case $L_\nu \models \psi_{\alpha}$ and $\nu$ is the least such, or $\nu > \alpha$ and for all $\alpha \cdot \mu < \nu$, we also have $\alpha \cdot \mu + \delta < \nu$ for any $\delta < \alpha$. Which is equivalent to saying that $\nu$ is a multiple of $\alpha$. Having in mind that, given $\beta$, the predicate $L_\beta \models \psi_\alpha$ is $\Sigma_1$, it is clear that $M_\alpha$ is $\Sigma_m$ with $m = \max(n, 3)$.
\end{proof}

\begin{prop}
For $\tau$ an additively closed and $\Sigma_3$-characterizable limit ordinal, the $\tau$-tick operator $\Gamma^\tau_{tick}$ is a $\Delta_4$-suitable and simulational operator. If $\tau$ is greater than $\Sigma_{\sup}$ this operator does not satisfy the looping stability condition.
\end{prop}
\begin{proof}
As it is built from two $\limsup$ operators, the $\tau$-tick operator is easily seen to be stable, cell-by-cell and contraction-proof for any limit ordinal $\tau$. It is only asymptotical when $\tau$ is additively closed as otherwise, taking a final segment of an history could displace the position of the next tick, that is of the next stage multiple of $\tau$.

Then, observe that it is easy for a machine to detect the "ticking" stages (i.e. stages multiple of $\tau$): it can simply have a cell whose value regularly alternates between $1$ and $2$ and it will read $2$ at a limit stage if and only if this stage is a ticking stage. From there, if $\tau > \Sigma_{\sup}$, we can design a machine that waits for stage $\tau $ to exit a repetition that repeated for more than $\om$ times, which falsifies the looping stability condition.

As for it being $\Delta_4$-suitable, the definition again provides the skeleton to define a formula and Proposition \ref{prop:multiplicity_sc_ordinals} ensures that with can distinguish with a formula of bounded complexity whether the limit stage is a multiple of $\nu$. For a complete proof, the reader is referred to \cite[Proposition 5.5.16]{these}.

\end{proof}

\begin{proof}[Proof of Theorem \ref{thm:contre_exemple_machine}]

We claim that for a great enough $\tau$, any $\tau$-tick machine will behave like the escaping machine having the same code with the only difference that it will be considerably \say{slower} as it will often be repeating for considerable amount of time (namely $\tau$ steps) before exiting the repetition and acting again like its escaping counterpart. Hence, it will still produce the same accidentally writable ordinals. More precisely, for $\tau = \chi$, the admissible characterizable ordinal defined in Proposition \ref{prop:big_charac_ordinal} such that both $\tau > \Sigma_{esc}$ and $\tau > \Tau_{esc}$, we will show that we have $\Tau_{\tau} > \tau$ and $\Sigma_{esc} = \Sigma_{\tau}$. This implies that:
\begin{align*}
\Sigma_{\tau} = \Sigma_{esc} < \tau < \Tau_{\tau}
\end{align*}
which proves the theorem.
We now need to show that both $\Sigma_{esc} = \Sigma_{\tau}$ and $\Tau_{\tau} > \tau$ hold. Observe that the second equality is immediate as $\tau$ is clearly clockable by a $\tau$-tick machine. For the first one, we make the following claim.

\begin{claim}
Let $m$ be a machine code, $m_{esc}$ the $\Gamma_{esc}$-machine with the code $m$ and ruled by the escaping operator and $m_{\tau}$ the $\Gamma_\tau$-machine with the code $m$ and ruled by the the $\tau$-tick operator. We write $(\alpha_\nu)$ for the limit stages of $m_{esc}$ that are ruled by $\Gamma_{210}$ (that is stages where some looping pattern $H_1$ has been repeating for $\om$ limit times, or limit of those) and $(\beta_\nu)$ the limit stages of $m_{\tau}$ that also are ruled by $\Gamma_{210}$ (that is stages multiple of $\tau$). Then, we claim that:
\begin{itemize}
\item For any ordinal $\nu$, the snapshots at stage $\alpha_\nu$ in $m_{esc}$ and $\beta_\nu$ in $m_\tau$ match.

\item For any ordinal $\nu$, we can write $H_0 \cdot H_1^\om$ for the history of $m_{esc}$ up to stage $\alpha_{\nu+1}$ and the history of $m_{\tau}$ up to stage $\beta_{\nu+1}$ reads $H_0 \cdot H_1^\tau$.

\end{itemize}
\end{claim}

\begin{proof}
We show this by induction. Consider stages $\alpha_0$ and $\beta_0$. By definition, before those stages, respectively in machines $m_{esc}$ and $m_{\tau}$, the limit operators behaved simply as $\Gamma_{102}$. Hence, both those machines started repeating at the latest from stage $\Sigma_{\sup}$ onward. That is, the history of $m_{esc}$ below $\alpha_0$ reads $H_0 H_1^\om$ where $H_1$ is some history of the machine below $\Sigma_{\sup}$ and at most the history between stages $\zeta_{\sup}$ and $\Sigma_{\sup}$ while the history of $m_\tau$ below $\beta_0$ reads (by closedness of $\tau$ and as $\tau > \Sigma_{\sup}$) $H_0 H_1^\tau$. Hence, as stages $\alpha_0$ and $\beta_0$ are both ruled by the operator $\Gamma_{210}$ and as this operator is asymptotic and satisfies the looping stability, both stages share the same snapshot.

Then, we consider successor stages $\alpha_{\nu+1}$ and $\beta_{\nu+1}$ for some successor ordinal $\nu+1$. By induction hypothesis, $\alpha_\nu$ and $\beta_\nu$ share the same snapshot. By asymptoticity, the computation of both machines match while they are both ruled by the operator $\Gamma_{102}$, that is, respectively, until stages $\alpha_{\nu+1}$ and $\beta_{\nu+1}$. And, for the first time after stage $\alpha_\nu$, the machine $m_{esc}$ produces an history of the form $H_0 H_1^\om$ for $H_1$ a looping pattern 
at stage $\alpha_{\nu+1}$. Hence the history below stage $\alpha_{\nu+1}$ can be written $H_0 H_1^\om$. As $\tau > \Tau_{esc}$, $\tau$ is also greater than the length of the history $H_0 \cdot H_1^\om$ and the machine $m_\tau$ is repeating $H_1$ up to stage $\beta_{\nu+1}$, the history before this stage reads $H_0 \cdot H_1^\tau$. Again, by the $\limsup$ definition of operator $\Gamma_{210}$, this implies that $m_{esc}$ and $m_\tau$ share the same snapshot at respectively stages $\alpha_{\nu+1}$ and $\beta_{\nu+1}$.

Eventually, for some limit $\nu$: Consider the sequence $(\alpha_{\nu'})_{\nu' < \nu}$ and the ordinal $\bigcup_{\nu' < \nu} \alpha_{\nu'}$. First, as the escaping operator also acts like the $\Gamma_{210}$ operator at stages that are limit of repeating histories (i.e. limit of histories of the form $H_0 \cdot H_1^\alpha$), we have:
\begin{align*}
\alpha_{\nu} = \bigcup_{\nu' < \nu} \alpha_{\nu'}
\end{align*}
On the other hand, this naturally holds for $\beta_{\nu}$, as a limit of multiple of $\tau$:
\begin{align*}
\beta_{\nu} = \bigcup_{\nu' < \nu} \beta_{\nu'}
\end{align*}
Moreover, by induction hypothesis, we can control the history before stages $\alpha_\nu$ and $\beta_{\nu}$. Namely, the history before stage $\beta_{\nu}$ is the same as that before $\alpha_{\nu}$ with the only difference that between some stages $\beta_{\nu'}$ and $\beta_{\nu'+1}$, it reads $H^{\nu'}_0 \cdot (H^{\nu'}_1)^\tau$ instead of $H^{\nu'}_0 \cdot (H^{\nu'}_1)^\om$ for some specific $H^{\nu'}_0$ and $H^{\nu'}_1$. Hence, by the definition of the $\limsup$ rule, stages $\alpha_{\nu}$ and $\beta_{\nu}$ share the same snapshot and this proves the claim.
\end{proof}
Finally, from the previous claim, it is clear that the machines $m_{esc}$ and $m_\tau$ sharing the same machine code accidentally write the same ordinals. And with it, more generally, that as wanted $\Sigma_{esc} = \Sigma_\tau$; which ends the proof.
\end{proof}

\section{Conclusion}


Simply put, the aim of this paper was to try and answer the following question: \say{Why did we choose the $\limsup$ rule for the definition of the ITTM and not another rule? And what makes the ITTM with this limit rule work so well?} Clearly this question, as is, does not have a satisfying answer. This is surely why it is bugging and, at the same time, a great kick start for an inquiry. So we started by investigating into some compelling results of the ITTM. What came into light, as presented in Section \ref{sec:introducing_the_universal_machine}, is the key role of the universal ITTM in the study of the model of Hamkins and Lewis. Then, in the same section, upon looking more closely at the universal ITTM and at its construction, we saw how said construction is actually grounded on a set of  four implicit properties, which the $\limsup$ rule satisfies. And what makes this observation potent is that, once formalized, those four properties actually are sufficient for a limit rule to yield a model of infinite Turing machine in which a universal machine exists. So working backward from there, we defined the concept of simulational operator, which is a limit operator (that is the (class) function behind a limit rule) which satisfies this set of four properties. Those operators, as said, yield a model of infinite Turing machine for which there exists a universal machine.  

Then, unstacking further, our aim was to establish for this class of operators (i.e.\ for the models of machine they yield) the same high-level results as were established for the classical ITTM. However, the concept of simulational operator alone falls a bit short of this goal.
But when we add another constraint, which we called looping stability, we take a more than significant step in this direction. First, as shown in Proposition \ref{prop:gamma_emulates_limsup}, the operators from this new class (that is, the simulational and looping stable operators) can emulate the computation of a classical ITTM. But further, among other similarities, they can be said to be as \say{well-behaved}, and that in virtue of the main theorem, Theorem \ref{thm:sigma_tau_looping_condition}, which states that for such an operator $\Gamma$ definable in set theory, $\Sigma_\Gamma$ equals $\Tau_\Gamma$. This equality, which holds for the classical ITTM, is one of the most compelling results in the study of ITTMs and it draws fruitful links between it and the fine structure theory. Finally, in Section \ref{sec:counter-example}, we prove that the constraint of looping stability is indeed needed, in the sense that there exists a simulational operator definable in set theory for which the previous equality does not hold.

Another significant result is the $\lambda$-$\zeta$-$\Sigma$ theorem of \cite{welch_main}. As done in Corollary \ref{cor:lambda-zeta-Sigma} It can be established in the setting of simulational and looping stable operators using Theorem \ref{thm:sigma_tau_looping_condition}. However, this theorem deals with operators definable in set theory with formulas of any complexity. That is, while the $\limsup$ operator is defined by a $\Sigma_2$ formula, other operators may be defined by higher-order formulas. And a question is whether, for those operators, we can establish a higher-order equivalent of the $\lambda$-$\zeta$-$\Sigma$ theorem. This was partially answered in the Chapter 6 of the thesis of Johan Girardot. In this chapter, a higher-order operator is designed. It satisfies the constraints which we established here and, for the model of infinite machine it yields, we can generalize the $\lambda$-$\zeta$-$\Sigma$ theorem up to a $\Sigma_3$ elementary end-extension. 

\bibliographystyle{plainurl}
\bibliography{sources}

\end{document}